\documentclass[12pt]{article}
\usepackage{amsmath,amsthm,amssymb,amsfonts,latexsym,bbm,epsfig,pstricks}
\usepackage[latin9]{inputenc}
\usepackage{indentfirst}
\usepackage{color}
\usepackage{mathtools}
\newtheorem{theo}{Theorem}[section]
\newtheorem{prop}[theo]{Proposition}
\newtheorem{lem}[theo]{Lemma}
\newtheorem{cor}[theo]{Corollary}
\newtheorem{remark}[theo]{Remark}

\newtheorem{defn}[theo]{Definition}
\newtheorem{exa}[theo]{Example}

\def\qed{{\hfill $\square$ \bigskip}}

\newcommand{\ra}{\rightarrow}

\newcommand{\ce}{\mbox{$\mathcal E$}}
\newcommand{\cf}{\mbox{$\mathcal F$}}

\newcommand{\bp}{\mbox{$\mathbb P$}}

\newcommand{\br}{\mbox{$\mathbb R$}}
\newcommand{\bn}{\mbox{$\mathbb N$}}
\newcommand{\be}{\mbox{$\mathbb E$}}

\newcommand\E{\mathbb E}
\newcommand{\cb}{\mathcal B}

\def\qed{\hbox {\rlap{$\sqcup$}{$\sqcap$}}}
\DeclareMathOperator{\Ima}{Im}

\newtheorem{innercustomthm}{Theorem}
\newenvironment{customthm}[1]
  {\renewcommand\theinnercustomthm{#1}\innercustomthm}
  {\endinnercustomthm}

\begin{document}
\title{%An approach to traces of random walks on the boundary of a hyperbolic group via reflecting extensions of Dirichlet forms
Besov spaces and random walks on a hyperbolic group: boundary traces and reflecting extensions of Dirichlet forms} 
\author{Pierre Mathieu\thanks{Aix-Marseille Universit\'{e}, CNRS, Centrale Marseille,
 I2M UMR 7373, 13453 Marseille, FRANCE. {\tt pierre.mathieu@univ-amu.fr}}
 \and Yuki Tokushige\thanks{The University of Manchester, Oxford Road,
Manchester, M13 9PL,
United Kingdom. {\tt yuki.tokushige@manchester.ac.uk}}}

\date{}
\maketitle

%{\red Some change in the abstract. Many changes in Part 5 up to Corollary \ref{new-new}.} 
%{\blue Please check comments in blue.}

\begin{abstract} 
We show the existence of a trace process at infinity for random walks on hyperbolic groups of conformal dimension $<2$ and relate it to the existence of a reflecteing random walk.  
To do so, we employ the theory of Dirichlet forms which connects the theory of symmetric Markov processes to functional analytic perspectives. 
We introduce a family of Besov spaces associated to random walks and prove that they are isomorphic to some of the Besov spaces constructed from the co-homology of the group studied in Bourdon-Pajot (2003). 
 We  also study the regularity of harmonic measures of random walks on hyperbolic groups using the potential theory associated to Dirichlet forms.
\end{abstract}

\section{Introduction}\label{intro}
In this introduction, we will explain jump processes on a boundary induced by a stochastic process (e.g. Brownian motions or random walks) on an {\it inner} domain (or graph). The induced jump process on a boundary and the corresponding Dirichlet form
 illustrate probabilistic and analytic aspects of the problem we consider in this paper.
We start with two motivational examples. ({\bf N.B.} In Section 1 and 2, we will mention many notions whose definitions are given in subsequent sections. In particular, readers who are not familiar with Dirichlet forms may skip the first two sections and see Section 3 and 7.)\\\par

{\bf Jump process on the circle. %induced by a reflecting Brownian motion on the disc
 } Consider a reflecting Brownian motion $({\tt BM}_t)$ on a 2-dimensional closed disc ${\cal D}:=\{(x,y)\in\mathbb{R}^2:x^2+y^2\leq1\}$ started at the origin. The path of $({\tt BM}_t)$ after its first hitting to the circle $S^1:=\{(x,y)\in\mathbb{R}^2:x^2+y^2=1\}$ can be decomposed into countably many boundary excursions\footnote{ An excursion is a continuous path that starts from some point on $S^1$, also ends on $S^1$ and lies in the interior of $\cal D$ in the meantime. The original path of $({\tt BM}_t)$  can be reconstructed by gluing together the sequence of its excursions off $S^1$ on the right time-scale. We refer to \cite{Bu} for details.}.

 Intuitively speaking, by forgetting how $({\tt BM}_t)$ moves in the interior of ${\cal D}$, we obtain the jump process on $S^1$ that jumps from starting points of boundary excursions to their endpoints. \par
 In order to turn this intuition into a rigorous construction, we will use the theory of {\it Dirichlet forms}. See \cite{FOT}  and \cite{CF} for background. Specifically, Example 1.2.3 in \cite{FOT} is particularly relevant to what we will explain below. 
 
The Dirichlet form corresponding to $({\tt BM}_t)$ on ${\cal D}$ is given by the Dirichlet integral
\begin{align*}
{\cal E}^{\cal D}(f,g):=\dfrac{1}{2}\int_{\cal D}\left(\dfrac{\partial f}{\partial x}\dfrac{\partial g}{\partial x}+\dfrac{\partial f}{\partial y}\dfrac{\partial g}{\partial y}\right)dxdy,
\end{align*}
with  domain %{\blue (changed the notation for the Sobolev space since $H$ is used for harmonic extensions.)}
\begin{align*}
W^{1,2}({\cal D}):=\left\{f\in L^2({\cal D}):\ \dfrac{\partial f}{\partial x},\dfrac{\partial f}{\partial y}\in L^2({\cal D})\right\}.
\end{align*}

In order to introduce the Dirichlet form corresponding to the jump process on $S^1$ induced by $({\tt BM}_t)$, we recall the {\it Poisson integral}, which gives the harmonic extension of some prescribed boundary value. For a function $\phi:S^1\to\mathbb{R}$, define its Poisson integral, denoted by $H^{\cal D}\phi$, as follows: 
\begin{align}\label{poisson int}
H^{\cal D}\phi(z):=\int_0^{2\pi}\dfrac{1-r^2}{1-2r\cos(\theta-\theta')+r^2}\phi(\theta')\dfrac{d\theta'}{2\pi},
\end{align}
where $z=re^{i\theta}$ and $\frac{d\theta'}{2\pi}$ is the uniform probability measure on $S^1$. It is well-known that the following Dirichlet form corresponds to the jump process on $S^1$ described above:
\begin{align}\label{circle}
{\cal F}^{S^1}&:=\{\phi\in L^2(S^1, d\theta/2\pi):\ H\phi\in W^{1,2}({\cal D})\},\nonumber\\
{\cal E}^{S^1}(\phi,\psi)&:={\cal E}^{\cal D}(H^{\cal D}\phi,H^{\cal D}\psi)\ \ \text{for}\ \phi,\psi\in{\cal F}^{S^1}.
\end{align}

Moreover, ${\cal E}^{S^1}$ has a more explicit expression called the {\it Douglas integral}:
\begin{align}\label{douglas}
{\cal E}^{S^1}(\phi,\psi):=\dfrac{\pi}{4}\int_0^{2\pi}\int_0^{2\pi}\sin^{-2}\left(\dfrac{\theta-\theta'}{2}\right)(\phi(\theta)-\phi(\theta'))(\psi(\theta)-\psi(\theta'))\dfrac{d\theta}{2\pi}\dfrac{d\theta'}{2\pi}. 
\end{align} 

The probability measure $\frac{d\theta}{2\pi}$ on $S^1$ is the first hitting distribution of $({\tt BM}_t)$ started at the origin. Interested readers may consult page 13-14 of \cite{FOT} for the proof of \eqref{circle} and \eqref{douglas}.\par
We remark that the integral kernel $\sin^{-2}\left(\frac{\theta-\theta'}{2}\right)$ in \eqref{douglas} is comparable to $|\theta-\theta|^{-2}$ when $|\theta-\theta'|\ll1$. This shows a glimpse of the fact that certain Besov spaces naturally appear in the study of induced jump processes on a boundary. \par
The Markov process corresponding to the Dirichlet form $({\cal E}^{S^1}, {\cal F}^{S^1})$ is the jump process on $S^1$ induced by a reflecting Brownian motion on ${\cal D}$.
It is called the {\it boundary trace} of  reflecting Brownian motion in the literature.

Another rigorous description of this induced jump process, 
which is more directly related to what we intuitively explained at the beginning, is to use the theory of {\it time-changes} of Markov processes. 
As a matter of fact, it is known that the induced jump process on $S^1$ is a reflecting Brownian motion time-changed by the inverse of  its boundary local time on $S^1$.
 See Section 5.3, Example ($3^{\circ}$) in \cite{CF} for detailed discussions
 \footnote { The boundary local time, $L_t$, is a non-decreasing continuous additive functional of $({\tt BM}_t)$ that only increases at times when the reflecting Brownian motion belongs to $S^1$. Let $(L^{-1}_t)$ be the right-continuous %generalized 
 inverse of the function $t\to L_t$. Then the time-changed process $t\to {\tt BM}\circ{L^{-1}_t}$ is a jump process on $S^1$ and its Dirichlet form is the one defined in (\ref{circle}).}.  
 
 The example of a reflecting Brownian motion on ${\cal D}$ already has a hyperbolic flavor. 
 Note that an Euclidean Brownian motion on ${\cal D}$ is a time change of a hyperbolic Brownian motion on the Poincar\'e disc
 \footnote{\label{foot3}
 The Poincar\'e disc is the interior of ${\cal D}$ equipped with the Riemannian metric 
 $\frac{4(dx^2+dy^2)}{(1-x^2-y^2)^2}$ and the volume form $\frac{4dxdy}{(1-x^2-y^2)^2}$.
 The hyperbolic Brownian motion in the Poincar\'e disc, say $({\tt w}_t)$, can be constructed as the solution of the stochastic differential equation 
$$ d{\tt w}_t=\frac {1}{2}(1-\Vert {\tt w}_t\Vert^2) d{\tt B}_t,$$
 %$$d{\tt w}_t=\frac 1{\sqrt{2}}(1-\Vert {\tt w}_t\Vert) d{\tt B}_t\,,$$ 
 %$$\blue{ \text{I think this should be}\ d{\tt w}_t=\frac {1}{2}(1-\Vert {\tt w}_t\Vert^2) d{\tt B}_t}$$
 where $\Vert .\Vert$ is the Euclidean norm. 
Define $$A(t)=\int_0^t \frac {(1-\Vert{\tt w}_s\Vert^2)^2} {4}ds.$$
%$$A(t)=\int_0^t \frac {(1-\Vert{\tt w}_s\Vert)^2} {2}ds.$$ 
%$$\blue{ \text{I think this should be}\ A(t)=\int_0^t \frac {(1-\Vert{\tt w}_s\Vert^2)^2} {4}ds.}$$
Then $A$ is an increasing bijection from $\mathbb{R}_+$ to the interval $[0,A(+\infty))$. Let $A^{-1}$ be its inverse. 
Then the process 
$t\to {\tt w}_{A^{-1}(t)}$ turns out to be a Euclidean Brownian motion considered up to its hitting time of $S^1$ 
i.e. it has the same law as the process $({\tt B}_t)$ up to time  
 $T_1=\inf\{t>0\,;\, \Vert{\tt B}_t\Vert=1\}$. 
 %{\red BTW there seems to be a typo in Lalley-Sellke page 174 in the definition of $\rho$. I hope my formulas here are more correct. } 
}.
 Since the operation of taking a boundary trace is independent the time-scale of the process, the boundary jump process of reflecting Brownian motion we already discussed can also be constructed as the boundary process of hyperbolic Brownian motion. \footnote{ Recall from Footnote \ref{foot3} that the hyperbolic Brownian motion is in fact a time-change of planar Brownian motion only up to its first hitting of $S^1$. Therefore  the construction of the boundary process from hyperbolic Brownian motion involves as an extra step the reconstruction of reflecting Brownian motion from Brownian motion killed on $S^1$.  A similar issue arises in the case of jump processes on  boundaries of trees discussed in the next paragraph. We address it in our context in Section \ref{reflecting-rw} of the paper.} 
 
 Let us finally mention a link with Calder\'on's inverse problem related from Electrical Impedance Tomography. Calder\'on's problem is to recover the unknown conductivity $\kappa$ in the elliptic equation
 $$\nabla\cdot(\kappa\nabla f)=0\ \ \text{in}\ {\tt Dom}$$
 from the measurements on $\partial({\tt Dom})$ expressed by the Dirichlet-to-Neumann map. 
 In the paper \cite{PS} the authors give a probabilistic interpretation of Calder\'on's inverse problem
using the boundary trace process of a reflecting diffusion process on an Euclidean domain ${\tt Dom}\subset\mathbb{R}^d$ ($d\geq2$). \\\par
 
{\bf Jump process on the boundary at infinity of an infinite tree. } 
Motivated by the example explained above, Kigami \cite{Ki} studied an analogous problem for transient random walks on an infinite tree. Suppose that we have a rooted infinite tree $T$. For simplicity, we consider a simple random walk $({\tt RW}_n)$ on $T$ in what follows. Furthermore, assume that $({\tt RW}_n)$ is transient.
Because of transience, roughly speaking, $({\tt RW}_n)$ escapes to "infinities" as $n\to\infty$. 
One way to make this loose description rigorous is to use the {\it geometric boundary} $\Sigma$ of $T$, which is 
the collection of infinite geodesics emanating from the root.
The simple random walk $({\tt RW}_n)$ almost surely converges as $n\to\infty$ to a random point of $\Sigma(T)$, say ${\tt RW}_\infty$, which is given by the loop-erasure of its trajectory.   
Denote by $\nu^T$ the distribution of ${\tt RW}_{\infty}$; this probability measure on $\Sigma(T)$ is called the {\it harmonic measure} of $({\tt RW}_n)$.
%This loose description can be made rigorous by employing the {\it Martin boundary}, which is a kind of notion of a boundary at infinity that can be defined for a pair of an infinite graph and a transient random walk on it. See \cite{Woe} for background.\par
%Let $\mathcal{M}^T$ be the Martin boundary associated to $T$ and $({\tt RW}_n)$. Suppose that the random walk started at the root of $T$. 
%Then it converges to a random point of $\mathcal{M}^T$ as $n\to\infty$.
%Denote by $\nu^T$ the distribution of ${\tt RW}_{\infty}$, which is a probability measure on $\Sigma(T)$ called the {\it harmonic measure} of $({\tt RW}_n)$.
%Denote by $\mathcal{E}^T$ the energy form that is naturally associated to $({\tt RW}_n)$, and let $\mathcal{F}^T:=\{f:T\to\mathbb{R};\ \mathcal{E}^T(f,f)<\infty\}$. 
We next define the energy form $(\mathcal{E}^T,\mathcal{F}^T)$ by
\begin{align*}
\mathcal{F}^T&:=\{f:T\to\mathbb{R}\ :\ \mathcal{E}^T(f,f)<\infty\}, \text{where}\\
\mathcal{E}^T(f,g)&:=\frac{1}{2}\sum_{x,y\in T: x\sim y}(f(x)-f(y))(g(x)-g(y))\ \text{for}\ f,g\in{\cal F}^T.
\end{align*}
%By the theory of Martin boundaries, we have a linear operator $H^T$, which is analogous to the Poisson integral \eqref{poisson int}, that transforms functions on $\mathcal{M}^T$ into harmonic functions on $T$. 
Analogously to the Poisson integral \eqref{poisson int}, we define a linear operator $H^T$ which transforms functions on $\Sigma(T)$ into harmonic functions on $T$ in the following way:
for a function $u:\Sigma(T)\to \mathbb{R}$ and $x\in T$, let 
\begin{align}\label{HAR-EX}
H^Tu(x):=\mathbb{E}[u({\tt RW}_{\infty})|{\tt RW}_0=x].
\end{align}
%This explanation is not completely precise and it requires the theory of {\it Martin boundaries} to make it rigorous.  See \cite{Woe} for background. 
%Why isn't it completely precise? It seems to me \ref{HAR-EX} does transform a function on the boundary into a harmonic function and this is easy to  check.  
We finally introduce the Dirichlet form $(\mathcal{E}^{\Sigma(T)}, \mathcal{F}^{\Sigma(T)})$ on $\Sigma(T)$ in an analogus way to \eqref{circle} as follows: 
\begin{align}\label{tree}
{\cal F}^{\Sigma(T)}&:=\{u\in L^2(\Sigma(T),\nu^T):\ H^Tu\in \mathcal{F}^T\},\nonumber\\
{\cal E}^{\Sigma(T)}(u,v)&:={\cal E}^{T}(H^{T}u,H^{T}v)\ \ \text{for}\ u,v\in{\cal F}^{\Sigma(T)}.
\end{align}
It is shown in \cite{Ki}, among other results, that $(\mathcal{E}^{\Sigma(T)}, \mathcal{F}^{\mathcal{M}^T})$ is {\it regular} on $L^2(\Sigma(T),\nu^{T})$, which is a condition that guarantees the existence of the Markov process on $\Sigma(T)$ corresponding to $(\mathcal{E}^{\Sigma(T)}, \mathcal{F}^{\Sigma(T)})$. Moreover, an explicit formula for ${\cal E}^{\Sigma(T)}(u,v)$, which is reminiscent of the Douglas integral \eqref{douglas}, is also obtained. See Theorem 5.6 in \cite{Ki}. Since $\Sigma(T)$ is homeomorphic to the Cantor set under suitable assumptions on the geometry of $T$, the Markov process corresponding to $(\mathcal{E}^{\Sigma(T)}, \mathcal{F}^{\Sigma(T)})$ can be viewed as a jump process on the Cantor set.
We finally notice that many arguments in \cite{Ki} heavily utilize explicit computations peculiar to tree structure. Therefore it seems very difficult to extend his arguments to graphs that have more robust geometric structure.\\\par

{\bf Jump process on the Gromov boundary at infinity of a hyperbolic group. 
} 
Motivated by two examples summarized above, we will consider in this paper an analogous problem for {\it hyperbolic groups}. A proper and geodesic metric space $(\tt{X}, \tt{D})$ is said to be {\it Gromov-hyperbolic} if the following condition is satisfied: there exists a constant $\delta>0$ such that for any geodesic triangle in ${\tt X}$, each side of the triangle is a subset of the $\delta$-neighborhood of the other two sides. A finitely-generated group is called a {\it hyperbolic group} if its Cayley graph is Gromov-hyperbolic. This property is independent of the choice of the generating set of the group since Gromov-hyperbolicity is known to be stable with respect to rough-isometries.\par 
To a given Gromov-hyperbolic metric space $(\tt{X}, \tt{D})$, we can associate a kind of boundary $\partial {\tt X}$ at infinity, called the {\it Gromov boundary}, which is defined as follows: let $x\in\tt{X}$ be a fixed base point. For $y,z\in\tt{X}$, we define the {\it Gromov product} $(y,z)_x$ of $y,z$ with respect to $x$ by
\begin{align}\label{def:product}
(y,z)_x:=\dfrac{{\tt D}(x,y)+{\tt D}(x,z)-{\tt D}(y,z)}{2}.
\end{align}
We say that the sequence $(x_n)\subset {\tt X}$ converges to infinity if
$\lim\inf_{i,j\to\infty}(x_i,x_j)_x=\infty$, and this notion does not depend on the choice of the base point.
Two sequences $(x_n), (y_n)$ converging to infinity are said to be equivalent if
$\lim\inf_{i,j\to\infty}(x_i,y_j)_x=\infty$. This condition indeed yields an equivalence relation and does not depend on the choice of the base point. We will write $[(x_n)]$ for the equivalence class of $(x_n)$. 
 We define the Gromov boundary $\partial{\tt X}$ by
 $$\partial X:=\{[(x_n)]:\ (x_n)\ \text{is a sequence in } {\tt X} \text{ and converges to infinity}\}.$$
 For $\xi\in\partial{\tt X}$ and $r\geq0$, we define
 $$V(\xi,r):=\{\eta\in\partial{\tt X};\ \text{there exist }(x_n),(y_n)\text{ s.t. }[(x_n)]=\xi,[(y_n)]=\eta,\ \liminf_{i,j\to\infty}(x_i,y_j)_x\geq r\}.$$
We endow $\partial {\tt X}$ with the topology whose basis of neighborhoods is given by $\{V(\xi,r): \xi\in\partial{\tt X},\ r>0\}.$ We extend the Gromov product to $\partial {\tt X}$ as follows: for $\xi,\eta\in\partial{\tt X}$,
\begin{align}\label{def:product-boundary}
(\xi,\eta)_x:=\sup\liminf_{i,j\to\infty}(x_i,y_j)_x,
\end{align}
where the supremum is taken over all sequences $(x_n),(y_n)$ with $\xi=[(x_n)]$ and $\eta=[(y_n)]$.
Interested readers may refer to \cite{GH} for more detailed information.
\par\medskip
Let $\Gamma$ be a non-elementary hyperbolic group, then 
$\Gamma$ is infinite, countable, non-amenable and discrete. Let $\mu$ be a symmetric probability measure on $\Gamma$ the support of which generates $\Gamma$. We need to assume a certain moment condition of $\mu$ for a reason we will explain later. Define a quadratic form $({\cal E}^{\mu},{\cal F}^{\mu})$ by
\begin{align*}
{\cal F}^{\mu}&:=\{f:\Gamma\to\mathbb{R}:\ {\cal E}^{\mu}(f,f)<\infty\}, \text{where}\\
{\cal E}^{\mu}(f,g)&:=\dfrac{1}{2}\sum_{x,y\in\Gamma}\mu(x^{-1}y)(f(x)-f(y))(g(x)-g(y))\ \text{for}\ f,g\in{\cal F}^{\mu}.
\end{align*}
Let $(R_n)$ be the RW driven by $\mu$ started at the identity ${\it id}$ of $\Gamma$. 
It is shown in \cite{Ka} that almost all trajectories of the random walk $(R_n)$ converge to some limit point $Z_\infty$ on the Gromov boundary $\partial\Gamma$ of $\Gamma$. 
The law of the random variable $Z_\infty$ is called the harmonic measure  of $(R_n)$ and denoted by $\nu$. 

 By this result, analogously to \eqref{poisson int} and \eqref{HAR-EX}, we have a linear operator $H$ that maps functions on $\partial\Gamma$ to harmonic functions on $\Gamma$: 
 $$Hu(x)=\E[u(x\cdot Z_\infty)]=\int u(x\cdot\xi)\, d\nu(\xi),$$ where $\E$ denotes the expectation with respect to the law of the random walk and $x\cdot\xi$ is the natural action of $\Gamma$ on its boundary. 
 Following the analogy with \eqref{circle} and \eqref{tree}, define the Dirichlet form on $L^2(\partial\Gamma,\nu)$ as follow:
\begin{align*}
{\cal F}^{\partial\Gamma,\mu}&:=\{u\in L^2(\partial\Gamma,\nu):\ Hu\in \mathcal{F}^{\mu}\},\nonumber\\
{\cal E}^{\partial\Gamma,\mu}(u,v)&:={\cal E}^{\mu}(Hu,Hv)\ \ \text{for}\ u,v\in{\cal F}^{\partial\Gamma,\mu}.
\end{align*}
%It is one of the main results of this paper to show the regularity of $({\cal E}^{\partial\Gamma,\mu},{\cal F}^{\partial\Gamma,\mu})$ on $L^2(\partial\Gamma,\nu)$ %, which guarantees the existence of the corresponding Markov process on $\partial\Gamma$, 
%under the assumption  that  the Ahlfors-regular conformal dimension of $\partial\Gamma$ is less than 2. 
It follows from  results in \cite{Na} and \cite{Sil} that ${\cal E}^{\partial\Gamma,\mu}(u,v)$ has an expression similar to the Douglas integral \eqref{douglas}. It reads 
\begin{align}\label{hyp douglas}
{\cal E}^{\partial\Gamma,\mu}(u,v)&=\int\int_{\partial\Gamma\times\partial\Gamma}(u(\xi)-u(\eta))(v(\xi)-v(\eta))\Theta^{\mu}(\xi,\eta)d\nu(\xi)d\nu(\eta)\ \ \text{for}\ u,v\in{\cal F}^{\partial\Gamma,\mu},\ \ \text{where}\nonumber\\
{\cal F}^{\partial\Gamma,\mu}&=B_2(\mu):=\left\{u\in L^2(\partial\Gamma,\nu);\ \int\int_{\partial\Gamma\times\partial\Gamma}(u(\xi)-u(\eta))^2\Theta^{\mu}(\xi,\eta)d\nu(\xi)d\nu(\eta)<+\infty\right\}.
\end{align}
See Proposition \ref{=} and Definition \ref{emu}. The integral kernel $\Theta^{\mu}(\cdot,\cdot)$ appearing in \eqref{hyp douglas} is the {\it Na\"{i}m kernel} introduced in \cite{Na}, see Definition \ref{naim kernel}. 

The expression \eqref{hyp douglas} by itself is not enough to show the regularity of $({\cal E}^{\partial\Gamma,\mu}, B_2(\mu))$ on $L^2(\partial\Gamma,\nu)$. To overcome this problem, we consider two {\it classes of Besov spaces} consisting of {\it (a)} Besov spaces associated to symmetric probability measures with finite second moments, and {\it (b)} Besov spaces associated to metrics in the {\it Ahlfors-regular conformal gauge}
introduced in \cite{BP} (see Definition \ref{def:arcd}). Under the assumption that the Ahlfors-regular conformal dimension of $\partial\Gamma$ is less than 2,
through a careful comparison between Besov spaces in these classes using martingale arguments, see Theorem \ref{reg}, we will obtain the regularity of those spaces {\it at once}.
This is where we need to assume that $\mu$ has a  finite second moment, since this enables us to compare energies of functions on $\Gamma$ up to multiplicative constants. See Proposition \ref{psc}.  \par
Having established the regularity property, we prove that the jump process on $\partial\Gamma$ corresponding to $({\cal E}^{\partial\Gamma,\mu}, B_2(\mu))$ is a certain time-change of a reflecting random walk in a sense to be made precise in Section \ref{reflecting-rw}. This last result completes the analogy with the case of a reflecting Brownian motion on ${\cal D}$. 
Finally, we prove that harmonic measures of random walks with a finite first moment are smooth in a potential theoretic sense. 

At the end of the paper, we will %thus have 
construct a boundary jump process that, as in the two motivating examples, is a time-change of a reflecting random walk under the assumption of the Ahlfors-regular conformal dimension being less than $2$. 
The boundary of a regular tree is a Cantor set of Ahlfors-regular conformal dimension $0$. The boundary of the two-dimensional disc considered in the first paragraph of this introduction is the circle $S^1$ whose Ahlfors-regular conformal dimension is $1$. Thus our condition on the Ahlfors-regular conformal dimension being less than $2$ extends to new examples the construction of boundary jump processes.

%We also study a certain relation between harmonic measures of RWs driven by different probability measures by utilizing potential theoretic aspects of the theory of Dirichlet forms. 

\section{Detailed summary of the paper}\label{summary}
The paper is based on  the interplay of the following three subjects:
 hyperbolic geometry, analysis on metric spaces, and
 the theory of Dirichlet forms and symmetric Markov processes associated to them.
 We begin this summary with briefly explaining
 the paper \cite{BP}, which is the starting point of our study and explains the interplay of 
 the first two of the three subjects mentioned above.\\\par
 {\bf Besov spaces.}\ \ \ Let $(Z,\rho)$ be a uniformly perfect compact metric space which carries a doubling measure.
 In \cite{BP}, the authors introduced a class of {\it Besov spaces} which is canonically associated to
 a certain conformal structure of $(Z,\rho)$.
 For any metric $d$ in the {\it Ahlfors-regular conformal gauge} $J_{AR}(Z,\rho)$ of $(Z,\rho)$,
 we define a Besov space $(\mathcal{E}^{Z,d},B_2(d))$ on $Z$ by
 \begin{align*}
 \mathcal{E}^{Z,d}(u,v)&:=\int\int_{Z\times Z}\frac{(u(\xi)-u(\eta))(v(\xi)-v(\eta))}{d(\xi,\eta)^{2q}}
 d\mathcal{H}_{d}(\xi)d\mathcal{H}_{d}(\eta),\\
 B_2(d)&:=\{u\in L^2(Z,\mathcal{H}_{d})\ ;\ \mathcal{E}^{Z,d}(u,u)<\infty\},
 \end{align*}
 where $q$ is the Hausdorff dimension of $(Z,d)$, and ${\cal H}_{d}$ is a Hausdorff measure
 of $d$.
 In \cite{BP}, the authors
  constructed a hyperbolic graph $\Gamma_{d}$ whose Gromov boundary is equivalent to $(Z,d)$.
 Then, they showed that the Besov space associated to a metric $d$ in $J_{AR}(Z,\rho)$
 is Banach isomorphic to the set of boundary values of 
 the elements in the $\ell_2$-{\it cohomology} group of $\Gamma_{d}$.
 They also showed that for any metrics $d,d'$ in $J_{AR}(Z,\rho)$,
 the two graphs $\Gamma_{d}$ and $\Gamma_{d'}$
 are quasi-isometric, and this quasi-isometry induces an isomorphism between the $\ell_2$-cohomologies 
 of $\Gamma_{d}$ and  $\Gamma_{d'}$.
 These results imply that all the Besov spaces associated to metrics $d$ in $J_{AR}(Z,\rho)$ are 
 Banach isomorphic, and thus are canonically associated to the conformal structure on $Z$. \par 
 
 In this paper, we shall consider the Besov spaces $B_2(d)$ when the compact metric space $(Z,\rho)$ is 
 the Gromov boundary of a non-elementary hyperbolic group $\Gamma$ equipped with a visual metric, see Definition \ref{df:visual}. 

% the first purpose of this paper is to find a relation between random walks on $\Gamma$ and the Bourdon-Pajot-Besov spaces.  %{\red I still don't understand:  I would rather say that 'the \cite{BP} construction applies when $(Z,\rho)$ is  the Gromov boundary $\partial\Gamma$ of a non-elementary hyperbolic group $\Gamma$ equipped with a visual metric $\rho_\Gamma$'. Random walks are introduced in the next paragraph. } 
 
 In the first part of this paper, we also define Besov spaces associated to random walks on $\Gamma$. 
 We will always assume that the driving measure $\mu$ of a random walk on $\Gamma$
 is symmetric and {\it admissible}, which means that the support of $\mu$ generates $\Gamma$.  
  For $k\geq1$, define $M_k$ to be the set of all symmetric admissible probability measures on $\Gamma$
 with finite $k$-th moment.
 For $\mu\in M_1$, we define a Besov space $(\mathcal{E}^{\partial\Gamma,\mu},B_2(\mu))$ associated to $\mu$ by
 \begin{align*}
 \mathcal{E}^{\partial\Gamma,\mu}(u,v)&:=\int\int_{\partial\Gamma\times \partial\Gamma}
 (u(\xi)-u(\eta))(v(\xi)-v(\eta)) \Theta^\mu(\xi,\eta) d\nu(\xi)d\nu(\eta),\\
 B_2(\mu)&:=\{u\in L^2(\partial\Gamma,\nu)\ ;\mathcal{E}^{\partial\Gamma,\mu}(u,u)<\infty\},
 \end{align*}
 where $\nu$ is the harmonic measure of the random walk driven by $\mu$, and
 $\Theta^\mu(\cdot,\cdot)$ is the {\it Na\"\i m kernel} associated to $\mu$.
 %These Besov spaces are Hilbert spaces whose norms are given by double integrals of the {\it Na\"\i m kernels} 
 %with respect to {\it harmonic measures}. 
 We prove that the Besov spaces associated  to different random walks are all isomorphic with each others and isomorphic to the Besov spaces in \cite{BP}. (See Proposition \ref{mumu} and Theorem \ref{iso}.)
% Moreover, with a careful investigation of the construction, we will show that
% intersections of the set of continuous functions $C(\partial\Gamma)$ on $\partial\Gamma$
% and those Besov spaces all coincide and the isomorphisms restricted on the intersection are all the identity map. 
 In the construction of the Besov spaces associated to random walks on $\Gamma$,
 the role of the $\ell_2$-cohomology of $\Gamma_{d}$ in \cite{BP} is played by the set of harmonic functions with
 a finite energy. 
 Moreover, the role of the quasi-isometry between $\Gamma_{d}$
 and $\Gamma_{d'}$ will be played by a stability result for bilinear forms of random walks on a group established in \cite{PSC}.
 (See Proposition \ref{psc} for the statement.)\\\par 
 
 {\bf Regular Dirichlet forms and conformal dimensions.}\ \ \ 
 The second purpose of this paper is to further investigate the probabilistic aspects of the Besov spaces introduced above
 by using the theory of {\it Dirichlet forms}. A Dirichlet form is a closed symmetric bilinear form on an $L^2$ space
 which satisfies a certain contraction property, called the {\it Markovian property}.
 In particular, for Dirichlet forms satisfying the regularity property (which roughly means that the domain
 of the form contains sufficiently many continuous functions, see Definition 2.6.), there is a well-known correspondence between
 regular Dirichlet forms and symmetric Markov processes. 
 We will explain basic facts on Dirichlet forms in Section 2.\par
 In Section \ref{subsec:regu}, we will prove that,  
 under the assumption that the {\it Ahlfors-regular conformal dimension} of $(\partial\Gamma,\rho_{\Gamma})$
 is strictly less than 2 (see Definition \ref{def:arcd}), the Besov spaces $(\mathcal{E}^{\partial\Gamma,d},B(d))$ and $(\mathcal{E}^{\partial\Gamma,\mu},B_2(\mu))$
 are regular Dirichlet forms for any $d\in J_{AR}(\partial\Gamma)$
 and $\mu\in M_2$. This result will allow us to construct 
 symmetric Markov processes associated to them. %(See Theorem \ref{reg}.)
 \begin{customthm}{\ref{reg}}
Assume the Ahlfors-regular conformal dimension of $\partial\Gamma$ is strictly less than 2.
Then for any $d\in J_{AR}(\partial\Gamma)$ and any 
 $\mu\in M_2$,
$(\mathcal{E}^{\partial\Gamma,d},B_2(d))$ and $(\mathcal{E}^{\partial\Gamma,\mu},B_2(\mu))$ are regular Dirichlet forms
 on $L^2(\partial\Gamma,{\cal H}_{d})$ and on $L^2(\partial\Gamma,\nu)$ respectively. 
\end{customthm} 
Our assumption on the Ahlfors-regular conformal dimension is optimal in the following sense: according to Theorem 0.3 in \cite{BP}, if a compact metric space $(Z,\rho)$ has the Ahlfors-regular conformal dimension equal to or greater than 2 and satisfies the Loewner property (see \cite[Chapter 8]{Hei}), then for any $d\in J_{AR}(Z,\rho)$, $B_2(d)$ only contains constant functions.\par
 Here is a brief sketch of the argument leading to Theorem \ref{reg}. Under our assumption on the Ahlfors-regular conformal dimension,
 there exists a metric $d_0$ belonging to $J_{AR}(\partial\Gamma)$
 such that $\dim(\partial\Gamma,d_{0})<2$, where $\dim$ is the Hausdorff dimension. It is easy to see that $(\mathcal{E}^{\partial\Gamma,d},B_2(d_0))$ is regular. \par
 Then, after a careful look at the isomorphism  between different Besov spaces, we can deduce that 
 all Lipschitz functions with respect to $d_0$ belong to all the Besov spaces of the form $(\mathcal{E}^{\partial\Gamma,d},B_2(d))$ and $(\mathcal{E}^{\partial\Gamma,\mu},B_2(\mu))$ as in the theorem.  
 We thus obtain the regularity of all these Besov spaces.\\\par 
 
  Examples of hyperbolic groups with Ahlfors-regular conformal dimension less than 2 are free groups, cocompact Fuchsian groups and carpet groups.
 In particular, it is proved in \cite{Ha1} that for any non-elementary hyperbolic group $G$ with planar boundary non-homeomorphic to the full sphere,
 the Ahlfors-regular conformal dimension of the boundary is strictly less than 2 if and only if
 $G$ is virtually isomorphic to a convex-cocompact Kleinian group. We refer to \cite{Ha1} and its references for other results on hyperbolic groups with planar boundaries.\\\par

 As a consequence of Theorem 5.1, by the general correspondence between regular Dirichlet forms and  Markov processes, 
 we conclude that each of the Besov spaces gives rise to a Markov process on $\partial \Gamma$. 
 The Besov space $(\mathcal{E}^{\partial\Gamma,d},B_2(d))$ corresponds to a strong Markov process (Hunt process) whose 
 reference measure is the Hausdorff measure $\mathcal{H}_{d}$ and whose jumping kernel is 
 $d(\cdot,\cdot)^{-2q}$,
 and the Besov space $(\mathcal{E}^{\partial\Gamma,\mu},B_2(\mu))$ corresponds to a strong Markov process whose 
 reference measure is the harmonic measure and whose jumping kernel is the Na\"\i m kernel.  \\\par 
 
 {\bf Reflected random walks.}\ \ \ 

 At the end of this paper, in Part \ref{reflecting-rw}, we will give a further probabilistic interpretation, now at the level of processes themselves. 
 More precisely, we show that 
 {\it there exists a Markov process $(W_t)$ with state space $\Gamma\cup\partial\Gamma$  that satisfies the following properties:\par
 (i) almost all trajectories of $(W_t)$ hit the boundary $\partial\Gamma$ in finite time;\par
 (ii) until the hitting time of $\partial\Gamma$, the trajectories of $(W_t)$ behave in distribution like a time change of the random walk with driving measure $\mu$ (See Proposition \ref{tau}.);\par 
 (iii) the process $(W_t)$ has a trace on the boundary $\partial\Gamma$ whose Dirichlet form is given by  $(\mathcal{E}^{\partial\Gamma,\mu},B_2(\mu))$ (See Theorem \ref{trace}.). }\par 
 The trace process is defined as the time-change of $(W_t)$ using the positive continuous additive functional whose Revuz measure (see Section 7 for the definition) is the harmonic measure $\nu$. See Section 7.1 for definitions and detail. \\\par
 
 The proofs of these claims use the notions of {\it smooth measures} (to be discussed in the next paragraph).  Roughly speaking, 
 in our context, 
 it consists in speeding up the original random walk so that it now hits the boundary in finite time and then prolongating  
 its life in such a way that the resulting process still behaves as the initial random walk when it is in $\Gamma$. 
 
 Under the assumption of the Ahlfors-regular conformal dimension being less than 2, we show that the construction 
 of such a reflecting random walk is possible with state space $\Gamma\cup\partial\Gamma$. 
 Properties (i), (ii) and (iii) follow from  general results from the theory of Dirichlet forms. Observe however that the regularity stated in Theorem \ref{reg}  and its close companion Theorem \ref{w}  are a crucial step to use Dirichlet form theory.\\\par

 {\bf Potential theoretic properties of harmonic measures.}\ \ \ 
 Harmonic measures of random walks on a non-elementary hyperbolic groups have been extensively studied
 and it is known that their behavior strongly depends on moment assumptions of the driving
 measures of the random walk. For instance, when the driving measure is finitely supported, 
 it is shown in \cite{BHM} that the associated Green metric on $\Gamma$ is hyperbolic and
 the corresponding harmonic measure belongs to the {\it Patterson-Sullivan class} (See \cite{Coo} and \cite{Ha2}.) 
 determined by the Green metric; the hyperbolicity of the Green metric is equivalent to {\it Ancona's inequality},
 which roughly means that the Green function is submultiplicative along geodesics. 
 
 Properties of harmonic measures are not so well understood when we only assume a weaker moment condition. 
 For instance it is shown in \cite{Gou} that for any non-elementary hyperbolic group $\Gamma$,
 there exists a symmetric probability measure on $\Gamma$ with some finite exponential moment for which 
 Ancona's inequality fails.
 Therefore, we cannot conclude in general that a harmonic measure belongs to the Patterson-Sullivan class 
 determined by the Green metric.  
 We mention here that results in \cite{T} imply that one can still compute the Hausdorff dimension of
 a harmonic measure on $\partial\Gamma$ as long as the driving measure has a finite first moment. \par
 
 When we are given a regular Dirichlet form and the corresponding Markov processes,
 we have potential theoretic objects associated to it
 such as capacities. Measures which do not charge sets of zero capacity are said to be {\it smooth}, and
 there is another potential theoretic notion for measures, called {\it measures of finite energy integral},
 which is a stronger property than smoothness.
Both notions are related to time changes of symmetric Markov processes. (See Subsection 5.2 for details.) 

 After proving the regularity of the Besov spaces, we study the smoothness property of harmonic measures
 of random walks on $\Gamma$. 
 For $\mu\in M_2$, we will introduce the set of all smooth measures (resp. the set of all measures of finite energy integral)
 with respect to the regular Dirichlet form
 $(\mathcal{E}^{\partial\Gamma,\mu},B_2(\mu))$; we  denote it with ${\cal S}(\partial\Gamma,\mu)$.
 (resp. ${\cal S}_0(\partial\Gamma,\mu)$)
 Similarly, for a metric $d$ in $J_{AR}(\partial\Gamma)$, we will define ${\cal S}(\partial\Gamma,d)$ 
 (${\cal S}_0(\partial\Gamma,d)$, resp.)
 to be the set of all smooth measures (resp. the set of all measures of finite energy integral)
 with respect to $(\mathcal{E}^{\partial\Gamma,d},B_2(d))$.
 (See Definition \ref{SS}.) In Theorem \ref{s}, we will prove that for any $d\in J_{AR}(\partial\Gamma)$ and
 $\mu\in M_2$ we have
 \begin{align}\label{coin}
 {\cal S}(\partial\Gamma,d)={\cal S}(\partial\Gamma,\mu),\ {\rm and}\ \  
 {\cal S}_0(\partial\Gamma,d)={\cal S}_0(\partial\Gamma,\mu).
 \end{align}
 We will denote the common set by ${\cal S}(\partial\Gamma)$ and by ${\cal S}_0(\partial\Gamma)$ respectively.
 We finally use (\ref{coin}) to study harmonic measures under very weak moment condition.
  %harmonic measures are of finite energy integral 
 %whenever the driving measure of the random walk has a finite first moment. (See Theorem \ref{1st}.)
 \begin{customthm}{\ref{1st}}
 Assume the Ahlfors-regular conformal dimension of $\partial\Gamma$ is strictly less than 2.
Then, both of ${\cal S}(\partial \Gamma)$ and ${\cal S}_0(\partial\Gamma)$ contain
 the harmonic measure $\nu$ of any random walk driven by $\mu$ in $M_1$.
  \end{customthm}
 When $\mu\in M_2$, the above claim can be relatively easily deduced  from 
 Poincar\'e-type inequalities on $\partial \Gamma$ as in Proposition \ref{pi}. (See Theorem \ref{s}.)
 To relax the moment condition to a finite first moment, in Section 9 we combine heat kernel estimates
 for jump processes from \cite{GHH} (See also \cite{CK,CKW}.)  and deviation inequalities from \cite{MS}.\\\par

 The paper is organized as follows. In Section \ref{pre-df}, we explain definitions and basic facts about Dirichlet forms. 
 We also introduce several examples of Dirichlet forms and the corresponding Markov processes. 
 Section \ref{besov-bp} is devoted to the explanation of the paper \cite{BP} including their construction
 of Besov spaces.
 In Section \ref{besov-hg}, we first introduce Besov spaces associated to random walks on a non-elementary
 hyperbolic group $\Gamma$. We then prove that Besov spaces introduced here and in \cite{BP} 
 are all isomorphic. Section \ref{besov-rw} starts with the proof of Theorem \ref{reg}. Then we further develop 
  the potential theory of Dirichlet forms and prove the smoothness property of harmonic measures.  
  In Section \ref{time-change}, we first introduce several general facts about time-change techniques 
 in the theory of Dirichlet forms. 
 In Section \ref{reflecting-rw}, we provide an interpretation of the Markov processes which correspond to the Besov spaces
  associated to random walks  using reflecting random walks. 
  In Section 9, we pursue further the potential theoretic aspect of harmonic measures using heat kernel estimates and deviation inequalities 
  and we prove Theorem \ref{1st}.

\section{Preliminary facts on Dirichlet forms}\label{pre-df}
 In this section, we briefly explain several basic facts about Dirichlet forms including their connection to probability theory. 
 See \cite{FOT,CF} for details of the theory of Dirichlet forms and their probabilistic aspects,
 especially the theory of symmetric Markov processes. 
 A Dirichlet form is a closed symmetric bilinear form on an $L^2$-space which satisfies a kind of contraction property,
 called the {\it Markovian} property. It is a general fact in functional analysis that there is a one to one correspondence
 between the collection of closed symmetric bilinear form defined on a Hilbert space ${\tt Hil}$ and the collection of
 non-positive definite self-adjoint operators on ${\tt Hil}$.
 Thus, we can associate a strongly continuous semigroup to a given closed symmetric form.
 In particular, when a given closed symmetric form satisfies the Markovian property, then the corresponding semigroup
 also has a kind of positivity preserving property, which is also called the {\it Markovian} property. 
 We will explain below what we quickly sketched out in detail.\\
 
 Let $E$ be a locally compact Hausdorff space, and $m$ be a positive Radon measure on $E$
 with full support.
 \begin{defn}
 \begin{description}
 \item[(1)] 
 We say that a bilinear form $(\ce,\cf)$ on a real Hilbert space
${\tt Hil}$ is a symmetric closed form if the following conditions are satisfied:
 \begin{itemize}
 \item$\cf$ is a dense linear subspace of ${\tt Hil}$, and $\ce:\cf\times\cf\to\mathbb{R}$ is 
 non-negative definite, symmetric and bilinear,
 \item for any $\alpha>0$, $(\cf,(\ce_{\alpha})^{1/2})$ is a Hilbert space, where 
 \begin{align*}
 \ce_{\alpha}(u,v):=\ce(u,v)+\alpha(u,v)_{H},\ \ u,v\in\cf.
 \end{align*} 
 \end{itemize}
 \item[(2)] A bilinear form $(\ce,\cf)$ is called a Dirichlet form on $L^2(E,m)$ if 
 $(\ce,\cf)$ is a closed symmetric form on $L^2(E,m)$,
 and for any $u\in\cf$, we have that 
 $v:=(0\vee u)\wedge1\in\cf$
 and $\ce(v,v)\leq\ce(u,u)$.
 The latter condition is called the Markovian property.
 \item[(3)] A linear operator $U:L^2(E,m)\to L^2(E,m)$ is said to be {\it Markovian}
 if for any $v\in L^2(E,m)$ with $0\leq v\leq1$ $m\mathchar`-a.e.$, we have that $0\leq Uv\leq1$
 $m\mathchar`-a.e.$
 \end{description}
\end{defn}
 \begin{theo}\cite[Theorem 1.3.1, Lemma 1.3.2, Theorem 1.4.1]{FOT}
 \begin{description}
 \item[(1)] There is a one to one correspondence between the collection of closed symmetric forms 
 $(\ce,\cf)$ on a real Hilbert space
${\tt Hil}$ and the collection of non-positive definite self-adjoint operators $A$ on ${\tt Hil}$.
 This correspondence is characterized by
 \begin{align}\label{generator}
 \begin{cases}
 {\rm Dom}(A)\subset\cf,\\
 \ce(u,v)=(-Au,v)_{H},\ u\in{\rm Dom}(A),v\in\cf.
 \end{cases}
 \end{align}
 \item[(2)] Let $A$ be a non-positive definite self-adjoint operator on ${\tt Hil}$.
 Then, $(T_t)_{t>0}:=(\exp(tA))_{t>0}$ is a strongly continuous semigroup on ${\tt Hil}$, and the generator of $(T_t)_{t>0}$
 coincides with $A$. Moreover, there is a unique strongly continuous semigroup whose generator is $A$. 
 \item[(3)] Let $(\ce,\cf)$ be a closed symmetric form on ${\tt Hil}$ and $(T_t)_{t>0}$
 be the corresponding strongly continuous semigroup on ${\tt Hil}$. Then, 
 $(\ce,\cf)$ is a Dirichlet form if and only if $T_t$ is Markovian for any $t>0$.
 \end{description}
\end{theo}

We next define an {\it extended Dirichlet space}, which will be used in what follows.
\begin{defn}\label{def:ext-dirichlet}
Let $(\ce,\cf)$ be a Dirichlet form on $L^2(E,m)$.
 We denote by $\cf_e$ the set of all $m$-measurable functions $u$ with the following properties:
 \begin{itemize}
 \item$|u|<\infty$ $m$-a.e. and
  \item there exists an approximating 
 sequence $(u_n)\subset\cf$ such that
 $\lim_{n\to\infty}u_n=u$ $m$-a.e. and $(u_n)$ is an $\ce$-Cauchy sequence, namely, for any $\varepsilon>0$, there exists an integer $N$ such that
 $\ce(u_n-u_m,u_n-u_m)<\varepsilon$ for any $n,m\geq N$. 
 \end{itemize}
 By the second property, it is obvious that for any $u\in\cf_e$ and its approximating sequence $(u_n)$,
 the limit $\ce(u,u):=\lim_{n\to\infty}\ce(u_n,u_n)$ exists and does not depend on the choice of the approximating sequence
 of $u$. We will call $(\cf_e,\ce)$ the extended Dirichlet space of $(\ce,\cf)$.  
 %Is there a good reason to use the notation $(\cf_e,\ce)$ instead of $(\ce,\cf_e)$?
 % I just followed FOT and Chen-Fukushima. Please let me know if you want to change this.
\end{defn}
The following theorem shows that the extended Dirichlet space $(\cf_e,\ce)$ characterizes
the original Dirichlet form $(\ce,\cf)$.
\begin{theo}\cite[Theorem 1.5.2]{FOT}\label{ex-dr}
Let $(\ce,\cf)$ and $(\cf_e,\ce)$ be a Dirichlet form on $L^2(E,m)$ and its extended Dirichlet space,
respectively. Then we have that $\cf=\cf_e\cap L^2(E,m)$.
\end{theo}
 We next explain the probabilistic aspect of the theory of Dirichlet forms, especially the connection to the theory of
 symmetric Markov processes. 
 Let $(X_t)$ be a Hunt process on $(E,m)$. We recall that Hunt processes are strong Markov processes 
 with certain regularity properties of the sample paths. 
 Then, the linear operator $T_t:L^2(E,m)\to L^2(E,m)$ given by 
 \begin{align}\label{semi-gr}
 T_tf(\bullet):=E_{\bullet}[f(X_t)]
 \end{align}
 defines a strongly continuous semigroup on $L^2(E,m)$. We say that $(X_t)$ is $m$-{\it symmetric}
 when $T_t$ is a symmetric operator on $L^2(E,m)$.
 Thus, it is a natural question to ask whether for a given Dirichlet form $(\ce,\cf)$,
  there exists a Hunt process $(X_t)$  such that the semigroup corresponding to $(\ce,\cf)$ coincides with that induced by
 the Markov process $(X_t)$.
 It is a well-known fact in the theory of Dirichlet forms that when the Dirichlet form $(\ce,\cf)$ on $L^2(E,m)$ is
 {\it regular}, which roughly means that the domain $\cf$ contains sufficiently many functions in
 $$C_0(E):=\{f:E\to\mathbb{R}\ ;\ f\ {\rm is\ a\ continuous\ function\ with\ compact\ support}\},$$
 then we have a correspondence between regular Dirichlet forms on $L^2(E,m)$ and
 $m$-symmetric Hunt processes on $E$. 
 \begin{remark}
We recommend interested readers
 to consult textbooks such as \cite{FOT,CF} for details including the precise definition of Hunt processes.
\end{remark}

\begin{defn}\label{def-reg}
A Dirichlet form $(\ce,\cf)$ on $L^2(E,m)$ is called regular if
$C_0(E)\cap\cf$ is dense both in $(C_0(E),\|\cdot\|_{\infty})$ and $(\cf,(\ce_1)^{1/2})$.
\end{defn}

\medskip
\noindent
{\bf FACT.} {\it There exists a correspondence, which is one to one in a certain sense, between
regular Dirichlet forms on $L^2(E,m)$ and $m$-symmetric Hunt processes on $E$. A Hunt process is a strong Markov process which has {\it cadlag} sample paths and certain additional properties.  
See \cite[Appendix A.2]{FOT} for the precise definition.\par
%%%{\sc I suppressed one line because - it is not used in the paper - it was not really correct since the DF of a Hunt process is in general not regular but only quasi-regular (if I am not mistaken).} 
%%%For a given $m$-symmetric Hunt process $(X_t)$, the corresponding Dirichlet form  is defined on $L^2(E,m)$ and it is related to the generator of $(X_t)$ by \eqref{generator}. 
 For a given regular Dirichlet form $(\ce,\cf)$, let $A$ be the self-adjoint operator determined by \eqref{generator}, and $(T_t)$ be the semigroup whose generator is $A$. Then, the corresponding Hunt process $(X_t)$ satisfies \eqref{semi-gr}. 
% For instance, for a given $m$-symmetric Hunt process $(X_t)$ and the corresponding semigroup is linked by the following  formula:
% $$T_tf(x):=E_{x}[f(X_t)]\ {\rm for}\ a.e.\ x\in E,$$
 See \cite[Chapter 7]{FOT} and \cite[Theorem 1.5.1]{CF} for the precise statement. }
 %{\s%c This 'FACT' is very abstract. It should be recalled here that through this correspondence, we have the probabilistic representation $T_tf(\bullet):=E_{\bullet}[f(X_t)]$. It might also be useful to state that Hunt processes have cad-lag paths.} 
 \\
 
 In the rest of this section, we will give two examples of regular Dirichlet forms and explain their probabilistic interpretation.
 The latter example will play a very important role in what follows.
\begin{exa}
 Consider the standard Dirichlet energy $\frac{1}{2}\int_{\mathbb{R}^n}(\nabla f\cdot\nabla g)dx$
 on the Euclidean space $\mathbb{R}^n$, where $u,v\in W^{1,2}(\mathbb{R}^n):=\{f\in L^2(\mathbb{R}^n,dx)\ ;\ 
 \frac{\partial f}{\partial x_i}\in L^2(\mathbb{R}^n,dx),\ {\rm for}\ i=1,...,n\}$.
 Then it is well-known that $(\frac{1}{2}\int_{\mathbb{R}^n}(\nabla f\cdot\nabla g)dx,W^{1,2}(\mathbb{R}^n))$
 is a regular Dirichlet form on $L^2(\mathbb{R}^n,dx)$.
 Moreover, the relation (\ref{generator}) implies that
 the corresponding non-positive self-adjoint operator is given by
 $\frac{1}{2}\Delta=\frac{1}{2}\sum_{i=1}^n\frac{\partial^2}{\partial x_i^2}$.
 Thus, the corresponding symmetric Hunt process is a standard $n$-$dim$ Brownian motion on $\mathbb{R}^n$.
\end{exa}
\begin{exa}\label{csrw}
Let $\mathbb{V}$ be either a finite set or a countable set.\ % which is typically a vertex set of a graph. 
 Let $c:\mathbb{V}\times\mathbb{V}\to\mathbb{R}_{\geq0}$ be a weight function which is symmetric $(${\it i.e.,} $c(x,y)=c(y,x))$. 
 We define a measure $m$ on $\mathbb{V}$ by $m(x):=\sum_{y\in\mathbb{V}}c(x,y)$, and assume that
 \begin{align}\label{assume-m}
 {\rm supp}(m)=\mathbb{V}\ \ {\rm and}\ \ \sup_{x\in\mathbb{V}}m(x)<\infty.
 \end{align}
 Now we define a bilinear form $(\ce,\cf)$ on $L^2(\mathbb{V},m)$ by %$L^2(V,m)$  
  \begin{align}\label{srw}
 \ce(u,v)&:=\frac{1}{2}\sum_{x,y\in \mathbb{V}}c(x,y)(u(x)-u(y))(v(x)-v(y)), \\
 \cf&=L^2(\mathbb{V},m),\nonumber
  \end{align}
  Then, it is shown in \cite[Theorem 2.2.2]{CF} that $(\ce,\cf)$ is a regular Dirichlet form on $L^2(\mathbb{V},m)$,
 and the corresponding $m$-symmetric Hunt process is the continuous time random walk $(X_t)_{t\geq0}$
 which is defined as follows:
 define $p(x,y):=c(x,y)/m(x)$, and let $(R_n)_{n\in\mathbb{N}}$ be a discrete time Markov chain with transition probabilities
 $(p(x,y))_{x,y\in \mathbb{V}}$. Let $(N_t)_{t\geq0}$ be a Poisson process with intensity 1 that is independent of $(R_n)$.
 Then, the $m$-symmetric Markov process $(X_t)_{t\geq0}$ is given by $X_t:=Y_{N_t}$.
 This construction of $(X_t)$ is equivalent to the fact that $(X_t)$ has random holding times given by {\it i.i.d.} exponential
 distributions with mean $1$ at all vertices. For this reason, the process $(X_t)_{t\geq0}$ is often called the ``constant 
 speed random walk''. See \cite[Section 2.2.1]{CF} for detail.

\end{exa}

\section{Besov spaces constructed by Bourdon and Pajot}\label{besov-bp}
In this section, we will give a summary of some results in \cite{BP},
in particular the construction of Besov spaces on a compact metric space.

\subsection{$\ell_p$-cohomology of simplicial complexes and  its invariance by \\ quasi-isometries}

We consider a simplicial complex $K$ equipped with a length metric, denoted by $|\cdot-\cdot|$, such that

\begin{itemize}
\item there exists a constant $C>0$ such that the diameter of all simplexes of $K$ are bounded by $C$, and
\item there exists a function $N:[0,\infty)\rightarrow\mathbb{N}$ such that all balls with radius $r$ contain
 at most $N(r)$ simplexes of $K$.
\end{itemize}
Simplicial complexes satisfying the 
above properties are called {\it geometric}. Now we define the $\ell_p$-cohomology of $K$.
We will say that $K$ is {\it uniformly contractible} if it is contractible and there exists a function
 $\phi:\mathbb{R}^{+}\rightarrow\mathbb{R}^{+}$ such that all balls $B_K(x,r)$ are contractible in $B_K(x,\phi(r))$. 
Let $K_i$ be the set of $i$-simplexes of $K$ and $\ell_pC^i(K)$ ($p\in[1,\infty]$) be the Banach space consisting
 of $\ell_p$-functions on $K_i$.  Define the coboundary operator $d_i:\ell_pC^i(K)\rightarrow\ell_pC^{i+1}(K)$
 by $(d_i\tau)(\sigma):=\tau(\partial\sigma)$, where $\tau\in\ell_pC^i(K)$ and $\sigma\in K_{i+1}$. 
 Note that if $K$ is a geometric simplicial complex, then $d_i$ is a bounded operator. The $i$-th $\ell_p$-cohomology group
 of $K$ is defined by 
\begin{equation*}
\ell_pH^i(K):=\ker{d_i}/\Ima d_{i-1}. 
\end{equation*}

The following theorem asserts the invariance of $\ell_pH^i(K)$ by quasi-isometries.
\begin{theo}\cite[Theorem 1.1]{BP}\label{inv}
Let $K$ and $K'$ be geometric uniformly contractible simplicial complexes. If $F:K\rightarrow K'$
 is a quasi-isometry, then it induces an isomorphism of topological vector spaces
 $N^{\bullet}:\ell_pH^{\bullet}(K')\rightarrow\ell_pH^{\bullet}(K)$.
\end{theo}
 We give a brief sketch of the construction of $N^1$, which will be used later.
 Let $C_i(K)$ be the vector space spanned by elements of $K_i$.
 First, define a map 
 \begin{align}\label{c0}
 c_0:K_0\rightarrow C_0(K')
 \end{align}
  by choosing an element of $K'_0$
 uniformly close to $F(x)$ for each $x\in K_0$.
 Next, define a map $c_1:K_1\rightarrow C_1(K')$, satisfying $\partial c_1(\sigma)=c_0(\partial\sigma)$
 for any $\sigma\in K_1$ in the following way: for an edge $a\in K_1$, denote its end points by $a_{+},a_{-}\in K_0$.
 Then we can find an element $c_1(a)$ of $C_1(K')$ with $\partial c_1(a)=c_0(a_{+})-c_0(a_{-})$.
 For $\tau\in\ell_pC^1(K')$, define a map $N^*(\tau):K_1\rightarrow\mathbb{R}$ by
 \begin{align}\label{n*}
 N^*(\tau)(\sigma):=\tau(c_1(\sigma)),\ \ \ \sigma\in K_1.
 \end{align}
 Then the isomorphism $N^1:\ell_pH^1(K')\rightarrow\ell_pH^1(K)$ is induced by the linear map \\ 
 $N^*:\ell_pC^1(K')\rightarrow\ell_pC^1(K)$. 

\subsection{The hyperbolic fillings by Bourdon and Pajot}
 In what follows, we always assume that a compact metric space $(Z,\rho)$ satisfies the following properties:%is a uniformly perfect compact metric space carrying a doubling measure.
 \begin{description}
 \item[(i)] $(Z,\rho)$ is {\it uniformly perfect}, namely there exists a constant $C>1$ such that for any $\xi\in Z$ and
 any $0<r\leq{\rm diam}(Z,\rho)$, we have that
 $$B_{\rho}(\xi,r)\setminus B_{\rho}(\xi,r/C)\neq\emptyset,$$
 where $B_{\rho}(\xi,r):=\{\eta\in Z:\rho(\xi,\eta)<r\}$.
 \item[(ii)] $(Z,\rho)$ carries a {\it doubling measure}, namely there exists a Borel measure $\theta$ on $Z$ such that
 there exists a constant $C'>1$ such that for any $\xi\in Z$ and any $r>0$, we have that
 $$\theta(B_{\rho}(\xi,2r))\leq C'\theta(B_{\rho}(\xi,r)).$$
 \end{description}
  We now introduce several definitions which will be important later.
\begin{defn}\label{def:arcd}
\begin{description}
\item[(1)] A metric $d$ on $Z$ is called Ahlfors-regular if there exists constants $C,C'>0$ such that
for any $\xi\in Z$ and any $0<r<{\rm diam}(Z,d)$, we have that
 \begin{align*}
 Cr^q \leq{\cal H}_{d}(B_{d}(\xi,r))\leq C'r^q,
 \end{align*}
 where ${\cal H}_{d}$ is a Hausdorff measure of $d$, and $q=\dim(Z,d)$.
\item[(2)] Two metrics $d,d'$ on $Z$ are called quasi-symmetric if there exists an
increasing homeomorphism $\alpha:[0,\infty)\to[0,\infty)$ such that for any distinct triple $\xi,\eta,\omega$
of $Z$ we have that
\begin{align*}
\frac{d(\xi,\eta)}{d(\xi,\omega)}\leq\alpha\left(\frac{d'(\xi,\eta)}{d'(\xi,\omega)}\right).
\end{align*}
\item[(3)] We denote by $J(Z,\rho)$ $($called the {\it conformal gauge} of $(Z,\rho))$ the set of all metrics on $Z$ which are
 quasi-symmetric to $\rho$, and by $J_{AR}(Z,\rho)$ $($called the {\it Ahlfors-regular conformal gauge} of $(Z,\rho))$
 the set of all Ahlfors regular metrics in $J(Z,\rho)$. 
 \item[(4)] We define the {\it Ahlfors-regular conformal dimension} of $(Z,\rho)$ %introduced in \cite{BP} 
 as the infimum of the Hausdorff dimension of metrics $d$ in the Ahlfors-regular conformal gauge $J_{AR}(Z,\rho)$. Namely, it is defined by $\inf\{\dim(Z,d):\ d\in J_{AR}(Z,\rho)\}.$
% \begin{align}\label{def:ar}
 %\inf\{\dim(Z,d):\ d\in J_{AR}(Z,\rho)\}.
 %\end{align}
\end{description}
\end{defn}
\medskip
\begin{defn}\label{df:visual} 
Let $\tilde{\Gamma}$ be a proper geodesic hyperbolic space in the sense of Gromov.
 A metric $d$ on the Gromov boundary $\partial\tilde{\Gamma}$ of $\tilde{\Gamma}$ is called a visual metric if the following condition holds:
 there exists a constant $C,C'>0$ and $a>0$ such that for any $\xi,\eta\in\partial\tilde{\Gamma}$ we have that
 $$Ce^{-a(\xi|\eta)_O}\leq d(\xi,\eta)\leq C'e^{-a(\xi|\eta)_O},$$
 where $(\xi|\eta)_O$ is the Gromov product extended to $\partial\tilde{\Gamma}$ with a fixed base point $O$. See \eqref{def:product-boundary}.
 Notice that the notion of visual metrics defines a class of metrics on $\partial\tilde{\Gamma}$ whose members are all quasi-symmetric.
See \cite{GH} for notions that appear here.
\end{defn}

For a uniformly perfect compact metric space $(Z,\rho)$ which carries a doubling measure, 
in \cite{BP}, it is shown that for any $d\in J(Z,\rho)$
 we can construct
 a geometric uniformly contractible simplicial complex $K_{d}$ 
 and a graph $\Gamma_d$ which is the 1-skeleton of $K_d$ with the following properties. 
 Precise definitions of $\Gamma_d$ and $K_d$ will be given later in this subsection.
\begin{theo}\cite[Proposition 2.1, Corollary 2.4.]{BP}\label{hf}

 Let $d\in J(Z,\rho)$. %let $\Gamma_{d}$ be the set of all vertices and edges in $K_{d}$. 
\begin{itemize}
%\item The inclusion $\Gamma_d\xhookrightarrow{}K_d$ is quasi-isometric when $\Gamma_d$ and $K_d$ are both equipped with the graph distance.
 \item The graph $\Gamma_{d}$ is of bounded degree.
\item The graph $\Gamma_{d}$ %equipped with the word length
 is hyperbolic in the sense of Gromov. 
 Moreover, the Gromov boundary of $\Gamma_d$ equipped with a visual metric is quasi-symmetric to $(Z,d)$.
 %Moreover, $(Z,d)$ is bi-Lipschitz equivalent to the Gromov boundary of $\Gamma_d$ equipped with some visual metric.% $e^{-(\cdot|\cdot)}$, where $(\cdot|\cdot)$ is the Gromov product of $\Gamma_d$.
 %its Gromov boundary equipped with a visual metric associated to the word length is bi-Lipschitz equivalent to $(Z,d)$.
\item For two metrics $d,d'\in J(Z,\rho)$, there exists a quasi-isometry
 $F:\Gamma_{d}\rightarrow \Gamma_{d'}$ which can be continuously %OK
  extended to the identity map on $Z$. 
\item Let $\tilde{\Gamma}$ be a proper and geodesic hyperbolic space in the sense of Gromov.
   Suppose that there exists a point
$O\in \tilde{\Gamma}$ and a constant $C\geq0$, such that all points in $\tilde{\Gamma}$
 are within distance $C$ from some geodesic ray starting at $O$.
 Assume that the Gromov boundary of $\tilde{\Gamma}$ is $Z$ and that $\rho$ is quasi-symmetric to a visual metric on $Z$ induced by the hyperbolic structure of $\tilde{\Gamma}$.
 Then for any $d\in J(Z,\rho)$, there exists a quasi-isometry $F:\Gamma_{d}\rightarrow \tilde{\Gamma}$ which can be continuously %OK
 extended to the identity map on $Z$.    
\end{itemize}
\end{theo}
It follows from Theorem \ref{inv} and the third claim of Theorem \ref{hf} %OK
  that $\ell_pH^{\bullet}(K_d)$ and \\ 
 $\ell_pH^{\bullet}(K_{d'})$ are isomorphic topological vector spaces.
 Hence these topological vector spaces can be considered to be an invariant with respect to quasi-symmetry,
 and we will denote it by $\ell_pH^{\bullet}(J(Z,\rho))$.\\

Here we explain the constructions of $K_{d}$ and $\Gamma_{d}$.
 Normalize the metric $d$ in such a way that diam$(Z,d)=1/2$.
 For each $l\geq0$, choose points $z^1_l,...,z^{k(l)}_l$ in $Z$ in such a way that for any
 $i,j\in\{1,...,k(l)\}$ with $i\neq j$, we have $d(z_l^i,z_l^j)\geq e^{-l}$,
 and for each $l\geq0$, the balls $B_l^i:=B_{d}(z_l^i,e^{-l})$, $1\leq i\leq k(l),$ cover $(Z,\rho)$.
 Denote by $S_l$ the cover $\{B_l^i;i\in\{1,..,k(l)\}\}$. Remark that $S_0$ must be the singleton $\{B_0^1\}$
 because of the normalization of the diameter. Now define $\Gamma_{d}$ as follows.
 The vertex set $V(\Gamma_{d})$ is
 the collection of balls $\{B_l^i;n\geq0,i\in\{1,...,k(l)\}\}$, and two distinct vertex $B,B'$ are connected by an edge if
 
 \begin{enumerate}  
 \item both $B$ and $B'$ belong to $S_l$ and $B\cap B'\neq\emptyset$, or if
 \item one of them belongs to $S_l$, the other belongs to $S_{l+1}$ and $B\cap B'\neq\emptyset$.
 \end{enumerate}
 We equip $\Gamma_{d}$ with a length metric, denoted by $|\cdot-\cdot|$,
 by identifying each edge with the Euclidean segment of length 1.
 Denote the vertex $B_0^1$ by $O$ and for each $x\in V(\Gamma_{d})$,
 let $B(x)$ be the set of all infinite geodesic rays
 starting at $O$ and passing through $x$. Now a simplicial complex $K_{d}$ is defined as follows: for $n\in\mathbb{N}$,
 the {\it $n$-th Rips complex of $\Gamma_{d}$} is a simplicial complex whose $k$-simplexes are sets of vertices 
 $\{x_1,...,x_{k+1}\}$ ($x_i\in V(\Gamma_{d})$, $1\leq i\leq k+1$) with $|x_i-x_j|\leq n$ for any $i,j\in\{1,...,k+1\}$.
 It is known that for $n$ large enough, the $n$-th Rips complex of $\Gamma_{d}$ is geometric and
 uniformly contractible, thus we let $K_{d}$ be
 the $n$-th Rips complex of $\Gamma_{d}$ for $n$ large enough. See \cite[Section 3]{BH} and \cite{BP}
  for the proof of these facts.\par
 Note that it is shown in \cite{BP} that for $d\in J(Z,\rho)$,
 \begin{equation}\label{coho}
 \ell_pH^{1}(J(Z,\rho))\simeq\{f:V(\Gamma_{d})\rightarrow\mathbb{R}\ 
 ;\ df\in\ell_p(E(\Gamma_{d}))\}/\ell_p(V(\Gamma_{d}))+\mathbb{R}.
 \end{equation}

 Relying on Theorem \ref{hf}, %OK
  in \cite{BP}, 
 the authors introduced the following Besov space associated to each metric in $J_{AR}(Z,\rho)$,
 and it is shown that
 the set of boundary values of elements in $\ell_pH^1(J(Z,\rho))$ coincides with the Besov space.
 \begin{defn}\label{def1}
 \begin{description}
 \item[(1)] Let $p\in[1,\infty)$. For a function $u:Z\rightarrow\mathbb{R}$ and a metric $d\in J_{AR}(Z,\rho)$
 of dimension $q$, define
 \begin{equation*}
 \|u\|_{p,d}:=\left(\int\int_{Z\times Z}\frac{|u(\xi)-u(\eta)|^p}{d(\xi,\eta)^{2q}}d\mathcal{H}_{d}(\xi)d\mathcal{H}_{d}(\eta)\right)^{1/p},
 \end{equation*}
 where $\mathcal{H}_{d}$ is the Hausdorff measure of $d$.
 Define $B_p(Z,d):=\{u:Z\rightarrow\mathbb{R};\ \|u\|_{p,d}<\infty\}$.
 We will call $(\|\cdot\|_{p,d},B_p(Z,d))$ a $p$-Besov space on $Z$ associated to $d$.
 Then $(B_p(Z,d)/\sim,\|\cdot\|_{p,d})$ is a Banach space,
 where $u\sim v$ means $u(\xi)-v(\xi)$ is a constant
 for $\mathcal{H}_{d}$-$a.e\ \xi$. 
 In the rest of the paper,
 we will write $B_p(Z,d)=B_p(d)$ when the choice of the space $Z$ is clear from the context.
 \item[(2)] In what follows, we will particularly focus on the $L^2$ case $(p=2)$, and thus employ the following 
 special notation: let $d\in J_{AR}(Z,\rho)$ be a metric of dimension $q$, and $u,v\in B_2(Z,d)$.
 Define
 \begin{align*}
 \mathcal{E}^{Z,d}(u,v):=\int\int_{Z\times Z}\frac{(u(\xi)-u(\eta))(v(\xi)-v(\eta))}{d(\xi,\eta)^{2q}}
 d\mathcal{H}_{d}(\xi)d\mathcal{H}_{d}(\eta).
 \end{align*} 
 \item[(3)] For $d\in J(Z,\rho)$ and $f:V(\Gamma_{d})
 \rightarrow\mathbb{R}$ with $df\in\ell_p(E(\Gamma_{d}))$, define 
 $f_{\infty}:Z\to\mathbb{R}$ (if it exists)
 \begin{equation*}
 f_{\infty}(\xi):=\lim_{n\to\infty}f(r(n)),\ \ \xi\in Z,
 \end{equation*}
 where $r$ is a geodesic ray of $\Gamma_{d}$
  starting at $O$ and converging to $\xi$. Note that if the limit exists, it does not 
 depend on the choice of $r$.% as $df$ tends to 0 at $\infty$.
 \end{description}
 \end{defn}

 In this paper, a linear map $T:D_1\to D_2$ between two Banach spaces $(D_1,\|\cdot\|_1)$ and $(D_2,\|\cdot\|_2)$
 is said to be a {\it Banach isomorphism} if $T$ is a bijection and a linear operator such that for any $u\in D_1$, we have
 $$C^{-1}\|u\|_1\leq\|Tu\|_2\leq C\|u\|_1$$ 
 for some constant $C>1$ which is independent of $u$.
 \begin{theo}\cite[Theorem 0.1, Theorem 3.1, Theorem 3.4.]{BP}\label{besov}
 \begin{description}
\item[(1)] Let $d\in J(Z,\rho)$ and $f:V(\Gamma_{d})\to\mathbb{R}$ be a function with
$df\in\ell_p(E(\Gamma_{d}))$. Then, for $\mathcal{H}_{d}$-$a.e.\ \xi\in Z$, the limit $f_{\infty}(\xi)$ exists and $f_{\infty}\in L^p(Z,\mathcal{H}_{d})$.
\item[(2)] For $d\in J(Z,\rho)$, the linear maps
\begin{align*}
I^{d}:\{f:V(\Gamma_{d})\rightarrow\mathbb{R};\ 
df\in\ell_p(E(\Gamma_{d}))\}&\rightarrow L^p(Z,\mathcal{H}_{d})\\
  f &\mapsto f_{\infty}\\
\bar{I}^{d}:\{f:V(\Gamma_{d})\rightarrow\mathbb{R};\ 
df\in\ell_p(E(\Gamma_{d}))\}/\mathbb{R}&\rightarrow L^p(Z,\mathcal{H}_{d})/\mathbb{R}\\
 \lbrack f \rbrack&\mapsto f_{\infty}\ \ mod\ \mathbb{R}
\end{align*}
are continuous. Moreover, ${\rm Ker}(I^d)=\ell_p(V(\Gamma_{d}))\ {\it and}\ {\rm Ker}(\bar{I}^d)=\ell_p(V(\Gamma_{d}))+\mathbb{R}$.

\item[(3)] When $d\in J_{AR}(Z,\rho)$, the map $\bar{I}^{d}$ induces a Banach isomorphism $\tilde{I}^{d}$ between
 $\ell_pH^1(J(Z,\rho))$ and $B_p(Z,d)/\sim$. 
 %Namely, there exists a constant $C>1$ such that
 %$C^{-1}\|f_{\infty}\|_{p,d}\leq\|df\|_{\ell_p(E(\Gamma_{d}))}\leq C\|f_{\infty}\|_{p,d}$
 %for any $[f]\in\ell_pH^1(J(Z,\rho))$.
  \end{description}
 \end{theo}

\section{Besov spaces associated to random walks on hyperbolic groups}\label{besov-hg} 
 In the previous section, we explained how to construct Besov spaces associated to metrics in the Ahlfors-regular conformal
 gauge of a given compact metric space. 
 In this section, we will choose, as a compact metric space,
 the Gromov boundary $\partial\Gamma$ of a non-elementary word hyperbolic
 group $\Gamma$ equipped with a visual metric, and
 we will introduce Besov spaces on $\partial\Gamma$ associated to random walks driven by probability measures
 with a finite second moment. Those Besov spaces will be introduced as sets of boundary values
 of harmonic functions on $\Gamma$. Moreover, we will show that sets of continuous functions in those Besov spaces,
 which are associated either to metrics in the Ahlfors-regular conformal gauge or to random walks driven by
 probability measures with finite second moment,
 are {\it canonical}. Namely, sets of continuous functions in Besov spaces do not depend on the choice of metrics in the Ahlfors-regular conformal gauge
 nor on probability measures with finite second moment.
 In addition, we will prove that for any choice of two Besov spaces among them, there exists a Banach isomorphism
 which coincides with the identity on the set of continuous functions. 
 \\

%%%%%%%%%%%%%%%%%%%%%%%%%%%%%%%%%%%%%%%%
%%%%%%%%%%%%%%%%%%%%%%%%%%%%%%%%%%%%%%%%
%%%%%%%%%%%%%%%%%%%%%%%%%%%%%%%%%%%%%%%%%

{\bf Notation}

Let $\Gamma$ be a non-elementary word hyperbolic group. 
We denote the neutral element by ${\it id}$.
 We will denote by $|\cdot-\cdot|_{\Gamma}$ a left-invariant word metric with respect to a fixed finite symmetric generating set,
 and let $\rho_{\Gamma}$ be a visual metric on the Gromov boundary $\partial \Gamma$ 
 constructed from $|\cdot-\cdot|_{\Gamma}$. 
 In what follows, we will write $J_{AR}(\partial\Gamma):=J_{AR}(\partial\Gamma,\rho_{\Gamma})$.
 For $d\in J_{AR}(\partial\Gamma)$,
 let $\mathcal{H}_{d}$ be the Hausdorff measure and $q$ be its Hausdorff dimension. 
 Recall that $\mathcal{H}_{d}$ is $q$-Ahlfors regular, namely, there exists a constant $C>1$ such that
 \begin{align*}
 C^{-1}r^q \leq\mathcal{H}_{d}(B_{d}(\xi,r))\leq Cr^{q}
 \end{align*}
 for any $\xi\in\partial\Gamma$ and 
 any $0<r<{\rm diam}(\partial\Gamma,\rho_{\Gamma})$.
 In what follows, we will write $B_2(d)$ for the Besov space $B_2(\partial\Gamma,d)$
 on the Gromov boundary $\partial\Gamma$ defined in Definition \ref{def1}.
 %In this paper, a linear map $T:D_1\to D_2$ between two Banach spaces $(D_1,\|\cdot\|_1)$ and $(D_2,\|\cdot\|_2)$ is said to be a {\it Banach isomorphism} if $T$ is a bijection and a bounded operator such that for any $u\in D_1$, we have
% $$C^{-1}\|u\|_1\leq\|Tu\|_2\leq C\|u\|_1$$ 
% for some constant $C>1$ which is independent of $u$.
\\

Let $\mu$ be a probability measure on $\Gamma$.
 Recall that we always assume that $\mu$ is symmetric ({\it i.e,}\ $\mu(x)=\mu(x^{-1})$ for $x\in\Gamma$),
 and admissible, which means that the support of $\mu$ generates $\Gamma$. 
 For $k\in\mathbb{N}$, let $M_k$ be the set of symmetric admissible probability measures on $\Gamma$ 
 with finite $k$-th moment
 ({\it i.e.,} $\sum_{x\in\Gamma}|id-x|_{\Gamma}^k\mu(x)<\infty$).

We consider the random walk with driving measure $\mu$:
 let $R_n$ be the position of the walk at time $n$. We denote with $({\cal G}_n)$ the filtration generated by the sequence $(R_n)$. 
 Let $\bp_x^{\mu}$ be the law of the random walk $(R_n)$
 starting at $x\in\Gamma$. 
 For $x,y\in\Gamma$, we have $\bp_x^{\mu}(R_n=y)=\mu^{*n}(x^{-1}y)$, where
 $\mu^{*n}$ is the $n$-th fold convolution power of $\mu$. Let $\be_x^\mu$ be the corresponding expectation. 
 Define $\bp^{\mu}=\bp^{\mu}_{id}$ and $\be^\mu=\be^\mu_{id}$. We use the notation 
 $G^{\mu}(x):=\sum_{k\geq0}\mu^{*k}(x)$ for the Green function associated to $\mu$.

 It is shown in \cite{Ka} that when $\mu\in M_1$ and $x\in\Gamma$,
 the random walk $(R_n)$ $\bp^{\mu}_x$-almost surely
 converges to a random point $R_{\infty}\in\partial\Gamma$ in the topology of $\Gamma\cup\partial\Gamma$. 
 Denote by $\nu_x$ the distribution of $R_{\infty}$ under the law $\bp^{\mu}_x$. Then $\nu_x$ is a probability measure on $\partial\Gamma$; 
 it is called the {\it harmonic measure} of $(R_n)$ starting at $x$. Define $\nu:=\nu_{id}$. 
 Note that $$\int_{\partial \Gamma}  u(\xi)\, d\nu_x(\xi)= \be_x^\mu[u(R_\infty)]=\be^\mu[u(x\cdot R_\infty)]$$ for any positive measurable function $u$ on $\partial \Gamma$. 
 It is known that, for all $x,y\in\Gamma$, the harmonic measures $\nu_x$ and $\nu_y$ are equivalent with a bounded density. The density of $\nu_x$ with respect to $\nu$ is given by the 
 {\it Martin kernel} 
 $${\cal K}^\mu(x,\xi)=\frac{d\nu_x}{d\nu}(\xi)\,$$
 see Definition \ref{naim kernel} and Lemma \ref{LB-naim}.  
 
 Consider the bilinear form 
\begin{align*}
{\cal E}^\mu(f,g)=\dfrac{1}{2}\sum_{x,y\in\Gamma} \mu(x^{-1}y)(f(x)-f(y))(g(x)-g(y))\,,
\end{align*}
and its full domain 
\begin{align*}
{\cal F}^\mu=\{f:\Gamma\ra{\mathbb R}\,;\, {\cal E}^\mu(f,f)<\infty\}\,.
\end{align*} 
Define 
\begin{align*}
\ell_2(\Gamma)&:=\left\{f:\Gamma\to\mathbb{R}\ ;\ \|f\|^2_{\ell_2(\Gamma)}:=\sum_{x\in\Gamma}f(x)^2<\infty\right\},  
\end{align*}
then for any $f\in{\cal F}^{\mu}$, we have that
\begin{align}\label{ele}
{\cal E}^{\mu}(f,f)=\frac{1}{2}\sum_{x,y\in\Gamma}\mu(x^{-1}y)(f(x)-f(y))^2
\leq\sum_{x,y\in\Gamma}\mu(x^{-1}y)(f(x)^2+f(y)^2)=2\|f\|_{\ell_2(\Gamma)}^2.
\end{align}
Therefore $\ell_2(\Gamma)\subset{\cal F}^\mu$. 

By Example \ref{csrw} and Proposition \ref{RD}, we will see that $(\ce^{\mu},\ell_2(\Gamma))$ is a regular Dirichlet form on $\ell_2(\Gamma)$
 (note that $m(x)=\sum_y\mu(x^{-1}y)=1$ for any $x\in\Gamma$), and the corresponding Hunt process is the constant 
 speed random walk $(X_t)$ on $\Gamma$, which is given by $(X_t)=(R_{N_t})$, where $(N_t)$ is an independent
 Poisson process with intensity $1$. 
 Notice that the trajectories of $(X_t)$ and $(R_n)$ are the same. As a consequence, for  
  $\mu\in M_1$ and $x\in\Gamma$,
 when $t$ tends to $\infty$, $X_t$ also $\bp^{\mu}_x$-almost surely converges to a random point in $\partial\Gamma$
 whose distribution is the harmonic measure $\nu_x$.\par

 We also define the discrete Laplacian $\Delta_{\mu}:\{f:\Gamma\to\mathbb{R}\}\to
 \{f:\Gamma\to\mathbb{R}\}$ by 
 \begin{align*}
 \Delta_{\mu}f(x):=\sum_{y\in\Gamma}\mu(x^{-1}y)f(y)-f(x).
 \end{align*} 
 We will say that $f:\Gamma\to\mathbb{R}$ is $\mu$-harmonic on $A\subset\Gamma$ if 
 $\Delta_{\mu}f(x)=0$ for any $x\in A$. 
%\begin{align*}
%\ell_2(\Gamma)&:=\{f:\Gamma\to\mathbb{R}\ ;\ \sum_{x\in\Gamma}f(x)^2<\infty\}, 
%\end{align*} 
We introduce the space of harmonic-Dirichlet functions: 
\begin{align*} 
\mathbb{HD}(\mu)&:=\{f\in{\cal F}^{\mu}\ ;\ \Delta_{\mu}f=0\ {\rm on}\ \Gamma\}.
\end{align*}
 We have the following decomposition for ${\cal F}^{\mu}$, see also \cite[Theorem 3.69]{Soa}. 
 
 \begin{prop}\label{RD}
 We have that $\ell_2(\Gamma)\subset{\cal F}^{\mu}$, and 
 for every $f\in{\cal F}^{\mu}$, there exists a unique pair of functions 
 $(f_0,f_{\mathbb{HD}})\in\ell_2(\Gamma)\times\mathbb{HD}(\mu)$ such that $f=f_0+f_{\mathbb{HD}}$.
 For such a pair of functions $(f_0,f_{\mathbb{HD}})$, we have that $\ce^{\mu}(f,f)=\ce^{\mu}(f_0,f_0)
 +\ce^{\mu}(f_{\mathbb{HD}},f_{\mathbb{HD}})$.
 In other words, the following orthogonal decomposition holds:
\begin{align}\label{royden}
{\cal F}^\mu=\ell_2(\Gamma)\bigoplus\mathbb {HD}(\mu).
\end{align}
Moreover, we have that $\ell_2(\Gamma)=\overline{C_0(\Gamma)}^{\ce^{\mu}_1}$, where
 $\ce^{\mu}_1(\cdot,\cdot):=\ce^{\mu}(\cdot,\cdot)+(\cdot,\cdot)_{\ell_2(\Gamma)}$ and 
 $$ C_0(\Gamma):=\{f':\Gamma\to\br\ ;\ |{\rm supp}(f')|<\infty\}.$$
\end{prop}

\begin{proof} 

It is easy to see that the space ${\cal F}^{\mu}/\mathbb{R}$ equipped with the norm
$\sqrt{\ce^{\mu}(\cdot,\cdot)}$
 is a Hilbert space. Let ${\cal F}^{\mu}_0$ be the closure of $C_0(\Gamma)$ with respect to 
 the metric $\sqrt{{\cal E}^{\mu}(\cdot,\cdot)}$.

Note that for any $f\in C_0(\Gamma)%:=\{f':\Gamma\to\br\ ;\ |{\rm supp}(f')|<\infty\}
$
 and any $g\in {\cal F}^{\mu}$, we have that
 \begin{align}\label{int-part}
 \ce^{\mu}(f,g)=-\sum_{x\in\Gamma}f(x)(\Delta_{\mu}g)(x).
 \end{align} 
 Therefore the spaces ${\cal F}^{\mu}_0$ and $\mathbb{HD}(\mu)$ are orthogonal for the scalar product $\ce^\mu$ and 
\begin{align*}
{\cal F}^{\mu}={\cal F}^{\mu}_0\bigoplus \mathbb{HD}(\mu)\,.
\end{align*}
%and for any $g_1\in{\cal F}^{\mu}_0$ and $g_2\in\mathbb{HD}(\mu)$, we have that ${\cal E}^{\mu}(g_1+g_2,g_1+g_2)=
%{\cal E}^{\mu}(g_1,g_1)+{\cal E}^{\mu}(g_2,g_2)$. 

We still have to prove that ${\cal F}^{\mu}_0=\ell_2(\Gamma)$.
We already noticed that $\ell_2(\Gamma)\subset{\cal F}^{\mu}_0$.
 On the other hand, every non-elementary hyperbolic group satisfies the following linear isoperimetric inequality: there exists a constant $C>0$ such that for any finite non-empty subset $A\subset\Gamma$, it holds that
 $$\#(A)\leq C\#(\partial_eA),$$
 where $\#$ denotes the cadinality of a set and $\partial_eA:=\{\text{an edge}\ e=\{x,y\}:\ x\in A\ \text{and}\ y\notin A\}$.
 %Theorem 4.27 in \cite{Soa} 
 Therefore, there exists a constant $C'>0$ such that 
 \begin{align}\label{P}
 \|f\|_{\ell_2(\Gamma)}^2\leq C'{\cal E}^{\mu}(f,f)
 \end{align} 
 for any $f\in C_0(\Gamma)$ (See Theorem 4.27 in \cite{Soa}). 
 %This implies that ${\cal F}^{\mu}_0\subset\ell_2(\Gamma)$. 
 Thus ${\cal F}^{\mu}_0$ is also the closure of $C_0(\Gamma)$ 
 in $\ell_2(\Gamma)$ and therefore coincides with $\ell_2(\Gamma)$.  

 \qed
 \end{proof}
\medskip

 We next study boundary values of functions in $\cf^{\mu}$. 
 %We remark that discussions below are very similar to those in Section 7 in \cite{KLW}. 
 A natural way to define a boundary value for $f\in\cf^{\mu}$ is to take a limit of $f(R_n)$.
 Here we will use the discrete-time process $(R_n)$ for simplicity of notation, but the
 same results hold for the continuous-time process $(X_t)$.
 By Proposition \ref{RD},
 for any $f\in\cf^{\mu}$, there exists a unique pair of functions $(f_0,f_{\mathbb{HD}})\in\ell_2(\Gamma)\times
 \mathbb{HD}(\mu)$ such that 
 $$f=f_0+f_{\mathbb{HD}}.$$
 By the definition of $\ell_2(\Gamma)$, for any $\varepsilon>0$ there exists a finite set $B\subset \Gamma$
 such that 
 $$\sum_{x\in \Gamma\setminus B}f_0(x)^2<\varepsilon.$$
 Since $(R_n)$ is transient, this implies that  
 $$\lim_{n\to\infty}f_0(R_n)(=\lim_{t\to\infty}f_0(X_t))=0$$
 $\bp^{\mu}_x$-almost surely  for any $x\in\Gamma$. Thus, we only need to consider the limit of $f_{\mathbb{HD}}(R_n)$. 
Since $f_{\mathbb{HD}}$ is $\mu$-harmonic,
 we have that for any $x\in\Gamma$, $(f_{\mathbb{HD}}(R_n))$ is a martingale under 
 $\bp^{\mu}_x$. It is shown in Theorem 9.11 in \cite{LP} that for any $x\in\Gamma$ we have that
 \begin{align*}
\sup_{n\in\mathbb{N}}\be^{\mu}_x[f_{\mathbb{HD}}(R_n)^2]\leq f_{\mathbb{HD}}(x)^2+2G^{\mu}(id)
\ce^{\mu}(f_{\mathbb{HD}},f_{\mathbb{HD}}),
 \end{align*}
 where $G^{\mu}(x):=\sum_{k\geq0}\mu^{*k}(x)$ is the Green function associated to $\mu$.
 Therefore, under  $\bp^{\mu}_x$, $(f_{\mathbb{HD}}(R_n))$ is a martingale which is bounded in $L^2$. By Doob's theorem,
 it converges almost surely and in $L^2$. Since $(\partial\Gamma,\nu)$ is the Poisson  boundary 
 of the random walk, see \cite{Ka}, 
 there exists a unique function $u\in L^2(\partial\Gamma,\nu)$ such that 
 \begin{align}\label{rw-conv} 
 \lim_{n\to\infty}f_{\mathbb{HD}}(R_n)=u(R_{\infty})\ {\it a.s.},
 \end{align}
 and 
 \begin{align}\label{he}
  f_{\mathbb{HD}}(x)=\int_{\partial\Gamma}u(\xi)d\nu_x(\xi)=:Hu(x).
 \end{align} 
 %%%{\sc I do not understand the next paragraph: why not just say that we let $K(x,\cdot)$ be the density of $\nu_x$ w.r.t. $\nu$?  There is no need for the discussion below and the references to \cite{Ka} or \cite{Woe}. Of course we should know this density exists.}
 %%%{\sc In fact with the definition of $Hu$ for $u \in L^2(\partial\Gamma,\nu)$ we already implicitly used the fact that $\nu_x$ is absolutely continuous w.r.t. $\nu$ with a bounded density. This is correct and can be proved as in Lemma \ref{LB-naim} below. Probably we don't want to insist on that point here. But we could suppress the discussion below.}
 %%%It is also known that for any $x\in\Gamma$, there exists a bounded measurable function $K(x,\cdot):\partial\Gamma\to\br_{\geq0}$ such that
 %\begin{align}\label{poi-mar}
 %Hu(x)=\int_{\partial\Gamma}K(x,\xi)u(\xi)d\nu(\xi).
 %\end{align}
 %We notice that $K(x,\cdot)$ is defined on $\partial\Gamma$ by employing a result in \cite{Ka}.
% It is shown in \cite{Ka} that the Gromov boundary equipped with the harmonic measure is isomorphic to the Poisson boundary of the walk, therefore it is measurably isomorphic to the Martin boundary equipped with the harmonic measure. See \cite[Chapter VI]{Woe} for details.

 We summarize  the discussions from the previous page in the following corollary. 
% {\sc I changed a bit the statement below.} 
 \begin{cor}\label{new}
 For $f\in{\cal F}^\mu$, let $f=f_0+f_{\mathbb {HD}}\ (f_0\in\ell_2(\Gamma),f_{\mathbb {HD}}\in\mathbb{HD}(\mu))$ be its Royden decomposition as in \eqref{royden}.
 Then the limit $\lim_{n\to\infty} f(R_n)=\lim_{n\to\infty} f_{\mathbb{HD}}(R_n)$ almost surely exists and the limiting function $u$ defined in \eqref{rw-conv} 
  belongs to
 $L^2(\partial \Gamma,\nu)$. 
%Moreover, let $u\in L^2(\partial\Gamma,\nu)$. Then $Hu$ belongs to $\mathbb{HD}(\mu)$ if and only if 
%there exists $f\in{\cal F}^\mu$ such that $\lim_{n\to\infty}f(R_n)=u(R_{\infty})$ a.s.

%we have 
% \begin{align}\label{hd}
%   &\{u\in L^2(\partial\Gamma,\nu)\ ;\ Hu\in\mathbb{HD}(\mu)\}\nonumber\\
%   =&\{u\in L^2(\partial\Gamma,\nu)\ ;\ {\rm there\ exists\ }f_{\mathbb {HD}}\in
% \mathbb{HD}(\mu)\ {\rm such\ that}\ \lim_{n\to\infty}f_{\mathbb {HD}}(R_n)=u(R_{\infty})\ a.s. \}.
% \end{align}
 \end{cor}

Let $u\in L^2(\partial\Gamma,\nu)$ and define $Hu$ as in \eqref{he}. Note that this definition makes sense since $\nu_x$ is absolutely continuous with respect to $\nu$ with a bounded density, see Lemma \ref{LB-naim}, and therefore $u\in L^2(\partial\Gamma,\nu_x)$ for all $x\in\Gamma$. 
 
  \begin{lem}\label{harm-ext}
For $u\in L^2(\partial\Gamma,\nu)$,  the sequence $(Hu(R_n))$ forms a martingale that is bounded in $L^2$ and converges almost surely and in $L^2$ towards $u(R_\infty)$. 
 \end{lem}
 \begin{proof} 
 
 Observe that 
 \begin{equation}\label{eq:he} 
 Hu(R_n)=\be^\mu_{R_n}[u(R_\infty)]=\be^\mu[u(R_\infty)\vert{\cal G}_n]\,. 
 \end{equation} 
 The first equality is the definition of $Hu$. The second equality is the Markov property applied to the random variable $u(R_\infty)$. 
 
 Since $u\in L^2(\partial\Gamma,\nu)$, the martingale $\be^\mu[u(R_\infty)\vert{\cal G}_n]$ is bounded in $L^2$. 
By Doob's theorem, it  converges almost surely and in $L^2$ towards $\be^\mu[u(R_\infty)\vert{\cal G}_\infty]=u(R_\infty)$. 
 \qed \end{proof}

  \medskip
The following result gives a motivation to introduce a class of Besov spaces associated to $\mu\in M_1$. We will see below that the Besov space associated to $\mu$ gives an alternative description of the collection of $u$'s in $L^2(\partial\Gamma,\nu)$ such that $Hu\in\mathbb{HD}(\mu)$.
%{\sc We could say that what we do now is give an alternative description of the set of $u$'s in $L^2(\partial\Gamma,\nu)$ such that $Hu\in\mathbb{HD}(\mu)$.} 
 
 \begin{defn}\label{naim kernel} 
 For $x,y\in\Gamma$, define the {\it Martin kernel} $\mathcal{K}^{\mu}(\cdot,\cdot)$ by
 \begin{align*}
 \mathcal{K}^{\mu}(x,y):=\frac{G^{\mu}(x^{-1}y)}{G^{\mu}(y)}.
 \end{align*} 
  It is shown in \cite{Dyn} 
  that $ \mathcal{K}^{\mu}$ can be extended to $\Gamma\times\mathcal{M}^{\mu}$,
 where $\mathcal{M}^{\mu}$ is the {\it Martin boundary} of $(\Gamma,\mu)$.\par
 For $x,y\in\Gamma$, define the {\it Na\"im kernel} $\Theta^{\mu}(\cdot,\cdot)$ by
 \begin{align*}
 \Theta^{\mu}(x,y):=\frac{G^{\mu}(x^{-1}y)}{G^{\mu}(x)G^{\mu}(y)}.
 \end{align*}
 It is shown in \cite{Sil} (See also \cite{Na}) that $\Theta^{\mu}$ can be extended to $\mathcal{M}^{\mu}\times\mathcal{M}^{\mu}
 \setminus\{(\omega,\omega')\in\mathcal{M}^{\mu}\times\mathcal{M}^{\mu}\ ;\ \omega=\omega'\}$.\par 
 It is also shown in \cite{Dyn} that the restriction of $\mathcal{K}^\mu(x,.)$ to $\mathcal{M}^{\mu}$ is a version of the 
 Radon-Nikodym derivative of $\nu_x$ with respect to $\nu$. 
% Define the Besov space $B_2(\mu)$ associated to $\mu\in M_1$ as follow: $$B_2(\mu):=\left\{u\\right\}.$$
 \end{defn}
 
 \begin{lem}\label{LB-naim} 
  The Martin kernel $\mathcal{K}^{\mu}(x,y)$ $(x,y\in\Gamma)$ has the following lower bound:
$$ \frac {G^\mu(id)}{G^\mu(x)}\geq\mathcal{K}^\mu(x,y)\geq \frac{G^\mu(x)}{G^\mu({\it id})}\ \ {\rm for\ any}\ x,y\in\Gamma.$$
 The Na\"im kernel $\Theta^{\mu}(x,y)$ $(x,y\in\Gamma)$ has the following lower bound:
$$ \Theta^{\mu}(x,y)\geq 1/G^{\mu}({\it id})\ \ {\rm for\ any}\ x,y\in\Gamma.$$
\end{lem}
We will give the proof for the sake of completeness. We will need this lemma to ensure that the Besov space associated to $\mu$ is included in $L^2(\partial\Gamma,\nu)$. See Definition \ref{emu} below.
 
% {\sc This Lemma will be used later but it is strange to put it here without any explanation.} 
\begin{proof} 

By the Markov property, we have 
$$G^{\mu}(z)=G^{\mu}({\it id})\bp^{\mu}(R_k=z\ {\rm for\ some}\ k\geq1).$$ 
%With the symmetry, we also get 
%$$=G^{\mu}({\it id})\bp^{\mu}_{\it z}(Z_k={\it id}\ {\rm for\ some}\ k\geq1).$$ 
%But this is not needed here. 
Therefore 
\begin{align*} 
& \mathcal{K}^\mu(x,y)=\frac{ \bp^{\mu}_{x}(R_k={y}\ {\rm for\ some}\ k\geq1)} {\bp^{\mu}_{\it id}(R_k={y}\ {\rm for\ some}\ k\geq1)}. 
\end{align*} 
But 
\begin{align*} &\bp^{\mu}_{\it id}(R_k={y}\ {\rm for\ some}\ k\geq1)\\ 
&\geq \bp^{\mu}_{\it id}(R_k={x}\ {\rm for\ some}\ k\geq1)\, \bp^{\mu}_{\it x}(R_k={y}\ {\rm for\ some}\ k\geq1). 
\end{align*} 
Therefore 
$$\mathcal{K}^\mu(x,y)\leq \frac 1{\bp^{\mu}_{\it id}(R_k={x}\ {\rm for\ some}\ k\geq1)} = \frac{G^\mu({\it id})}{G^\mu(x)}\,.$$ 
The lower bound on $\mathcal{K}^\mu(x,y)$ is proved the same way. The bound on $ \Theta^{\mu}(x,y)$ follows at once. 

%Therefore, we obtain that
%\begin{align*}
%&\ \ \ \ \Theta^{\mu}(x,y)=\frac{G^{\mu}(x^{-1}y)}{G^{\mu}(x)G^{\mu}(y)}\\
%&=\frac{1}{G^{\mu}({\it id})}\cdot\frac{\bp^{\mu}_x(Z_k=y\ {\rm for\ some}\ k\geq1)}{\bp^{\mu}(Z_k=x\ {\rm for\ some}\ k\geq1)\bp^{\mu}(Z_k=y\ {\rm for\ some}\ k\geq1)}\\
%&=\frac{1}{G^{\mu}({\it id})}\cdot\frac{\bp^{\mu}_x(Z_k=y\ {\rm for\ some}\ k\geq1)}{\bp^{\mu}_x(Z_k={\it id}\ {\rm for\ some}\ k\geq1)\bp^{\mu}_{\it id}(Z_k=y\ {\rm for\ some}\ k\geq1)}\\
%&\geq\frac{1}{G^{\mu}({\it id})},
%\end{align*}
%where the strong Markov property is used in the last step.
\qed
\end{proof}
 \medskip
  
 \begin{remark}The Na\"im kernel is pointwisely defined on the Martin boundary. In Definition \ref{emu} and thereafter, we use a version of the Na\"im 
kernel that is $\nu\times\nu$-almost surely defined on the Gromov boundary.
% as in the formula \eqref{poi-mar}. 
It is shown in \cite{Ka} that the Gromov boundary equipped with the harmonic measure is isomorphic to the Poisson boundary of the walk, therefore it is measurably isomorphic to the Martin boundary 
equipped with the harmonic measure. See \cite[Chapter VI]{Woe} for details.
%{\sc Here is where we do need that discussion about the Gromov-Martin-Poisson boundary from the comment of formula \eqref{poi-mar}.} 
\end{remark}

  By \eqref{he} and \cite[Theorem 3,5]{Sil}, we have the following result, which is an extension to the discrete setting of the Douglas integral (see \eqref{circle} and \eqref{douglas}) introduced in Section 1.
 %{\sc Did we define $B_2(\mu)$?} 
 %{\sc I change the statement.} 

\begin{prop} \label{=}\cite[Theorem 3.5]{Sil} 
Suppose that we have a function $u\in L^2(\partial\Gamma,\nu)$. 
%such that %Then $u$ belongs to $B_2(\mu)$ if and only if $Hu\in\mathbb{HD}(\mu)$. 
Then we have   
 \begin{align}\label{naim}
{\cal E}^\mu(Hu,Hu)=\int\int_{\partial\Gamma\times \partial\Gamma}
 (u(\xi)-u(\eta))^2\Theta^\mu(\xi,\eta) d\nu(\xi)d\nu(\eta)\,,
\end{align}
where $\Theta^\mu$ is the Na\"im kernel and with the understanding that the right-hand side of \eqref{naim} is finite if and only if $Hu\in\mathbb{HD}(\mu)$. 
%Moreover, for functions $f_{\mathbb{HD}}\in\mathbb{HD}(\mu)$ and $u\in L^2(\partial\Gamma,\nu)$ such that
% $\lim_{n\to\infty}f_{\mathbb{HD}}(R_n)=u(R_{\infty})$\ {\it a.s.}, ${\cal E}^\mu(f_{\mathbb{HD}},f_{\mathbb{HD}})$
% coincides with the right hand side of \eqref{naim}.
\end{prop}

%{\s%c The statement in Silverstein is more general: it says the right hand side in \ref{naim} is finite iff $u$ is such that $Hu$ is in the domain. I believe we also need that.} 

%{\sc We could give more details on the Naim kernel: explain how it relates to the measure on paths as in Silverstein equation 1.11. This is very reminiscent of the Geodesic flow, see Kaimanovich Crelle 94.} 

We now introduce the following bilinear form associated to $\mu$.
\begin{defn}\label{emu}Define
$$\mathcal{E}^{\partial\Gamma,\mu}(u,u)=\int\int_{\partial\Gamma\times \partial\Gamma}
 (u(\xi)-u(\eta))^2 \Theta^\mu(\xi,\eta) d\nu(\xi)d\nu(\eta),$$ 
with domain $B_2(\mu):=\{u:\partial \Gamma\rightarrow\mathbb{R}\ ;\mathcal{E}^{\partial\Gamma,\mu}(u,u)<\infty\}$. 
We will call $(\mathcal{E}^{\partial\Gamma,\mu},B_2(\mu))$ the Besov space associated to $\mu\in M_1$.
  % implies that $\Theta(\xi,\eta)\geq1/G^{\mu}({\it id})$.%\eqref{hd} and Proposition \ref{=}.
 \end{defn}

Note that $B_2(\mu)\subset L^2(\partial\Gamma,\nu)$ by Lemma \ref{LB-naim}.
 By Proposition \ref{=}, we have 
 \begin{align}\label{besov=hu}
 B_2(\mu)=\{u\in L^2(\partial\Gamma,\nu);\ Hu\in\mathbb{HD}(\mu)\}.
 \end{align} 
 Thus if $u\in B_2(\mu)$ then $Hu\in \mathcal{F}^{\mu}$ and Lemma \ref{harm-ext} implies that $\lim_{n\to\infty}f(R_n)=u(R_\infty)$. 
 Reciprocally,   if $f\in \mathcal{F}^{\mu}$ then $f_{\mathbb{HD}}=Hu$ for some $u$ in $B_2(\mu)$ and Corollary \ref{new} and Lemma \ref{harm-ext} imply that $\lim_{n\to\infty}f(R_n)=u(R_\infty)$. 
 We summarize these findings in the next corollary.

\begin{cor}\label{new-new}
 For $f\in\mathcal{F}^{\mu}$, let $\tilde{U}(\mu)f$ be the function on $\partial\Gamma$ defined by  
 $$\lim_{n\to\infty}f(R_n)=\tilde{U}(\mu)f(R_{\infty})\ {\it a.s.}$$ Then $\tilde{U}(\mu)$ defines a surjective linear map from $\mathcal{F}^{\mu}$ to $B_2(\mu)$ and ${\rm Ker}(\tilde{U}(\mu))=\ell_2(\Gamma)$. 
\end{cor}

%{\sc I removed the proof.}
%\begin{proof}
%The surjectivity of $\tilde{U}(\mu)$ follows from \eqref{he}, Corollary \ref{new} and Proposition \ref{=}. The claim for ${\rm Ker}(\tilde{U}(\mu))$ is immediate from Lemma \ref{alp}. \qed
%%{\s%c I don't understand the surjectivity. Starting from $u\in B_2(\mu)$, the way it is stated, \eqref{naim} does not immediately ensure that $Hu$ is in $\mathbb{HD}$. Do I miss something obvious?} 
%\end{proof}
\medskip

  We will study the relations between $B_2(\mu)$ and $B_2(d)$.
  Hereafter, we assume $\mu\in M_2$, which is needed for the next proposition.
 It is a very important assumption since it plays in our probabilistic construction a similar role to 
  quasi-isometries in the geometric context. 
  In what follows, if $f$ and $g$ are two functions defined on a set $A$, $f\asymp g$ means that there exists a constant
  $C>1$ such that $C^{-1}g(a)\leq f(a)\leq Cg(a)$ for any $a\in A$.
 \begin{prop}\cite[Lemma 2.1]{PSC}\label{psc}
 For any $\mu,\mu'\in M_2$, we have that 
 \begin{align*}
 \cf^{\mu}=\cf^{\mu'},\ {\rm and}\ \ \ce^{\mu}(f,f)\asymp\ce^{\mu'}(f,f)\ {\rm for\ any}\ f\in\cf^{\mu}=\cf^{\mu'}.
 \end{align*}
 \end{prop}
 Before giving the main result of this section, we will prove the following results.
 
 \begin{lem}\label{c}
 For any
 $d\in J_{AR}(\partial\Gamma)$ and any $\mu\in M_2$,
 we have that $C(\partial \Gamma)\cap B_2(\mu)=C(\partial \Gamma)\cap B_2(d)$.
 We will denote the common set by ${\cal C}$.
 \end{lem}
 \begin{proof} 
 Let $ {\cal E}^{{\rm SRW}}$ be the bilinear form associated to the simple random walk on $\Gamma$
 with respect to a fixed finite symmetric generating set.
 Choose $d\in J_{AR}(\partial\Gamma)$ and $\mu\in M_2$ arbitrarily.
 Take $u\in C(\partial\Gamma)\cap B_2(d)$. 
  Define $g:V(\Gamma_{d})\rightarrow\mathbb{R}$ by
 \begin{align}\label{exx}
 g(x):=\frac{1}{\mathcal{H}_{d}(B(x))}\int_{B(x)}u(\xi)d\mathcal{H}_{d}(\xi).
 \end{align}
 It is shown in the proof of Theorem 3.4 in \cite{BP} that 
 \begin{align*}
 \mathcal{E}^{\partial\Gamma,d}(u,u)\asymp \|dg\|^2_{\ell_2(E(\Gamma_{d}))},\ {\rm and}\ 
 I^{d}(g)=u,\ {\cal H}_{d}\mathchar`-a.e.
 \end{align*}
 Let $F_{d}:\Gamma\rightarrow V(\Gamma_{d})$
 be a quasi-isometry which continuously %OK
  extends to the identity on $\partial\Gamma$, as was shown to exist in Theorem \ref{hf}.
 We have that
 \begin{align*}
 {\cal E}^{{\rm SRW}}(g\circ F_{d},g\circ F_{d})\asymp \|dg\|^2_{\ell_2(E(\Gamma_{d}))}.
 \end{align*}
  by the stability of Dirichlet forms under quasi-isometries (\cite[Theorem 3.10]{Woe}).
 Note that $\Gamma_{d}$ satisfies the assumption in \cite[Theorem 3.10]{Woe} since
 $\Gamma_{d}$ is of bounded degree. See Theorem \ref{hf}. %OK
 By Proposition \ref{psc}, for $\mu\in M_2$
 we have that
 \begin{align*}
 {\cal E}^{{\rm SRW}}(g\circ F_{d},g\circ F_{d})\asymp {\cal E}^{\mu}(g\circ F_{d},g\circ F_{d}).
 \end{align*}
 Thus $g\circ F_{d}\in{\cal F}^{\mu}$.
 By Proposition \ref{=}, $g\circ F_{d}$ has a limit along a path of the random walk driven by $\mu$ with probability 1,
 and the limiting function $v:\partial \Gamma\rightarrow\mathbb{R}$ belongs to $B_2(\mu)$.
 Since it is shown in \cite{Ka} that $(R_n)$ converges to a random element
 $R_{\infty}\in\partial\Gamma$ in the topology of the compactified space $\Gamma\cup\partial\Gamma$, the sequence
 $(F_{d}(R_n))$ also converges to $R_{\infty}$. 
 On the other hand,
 by the continuity of $u$ and the definition of $g$, for any sequence $(h_n)\subset\Gamma_{d}$ converging to
 a point $\eta\in\partial\Gamma$ we get that $\lim_{n\to\infty}g(h_n)=u(\eta)$.
 This observation together with the above argument implies that 
 $v(R_\infty)=\lim_{n\to\infty}g\circ F_{d}(R_n)=u(R_{\infty})$ $\bp^{\mu}$-{\it a.s.} 
Thus we get that $v=u\ \nu$-{\it a.s.},
 and this implies that $u\in B_2(\mu)$. \par
 For $v'\in C(\partial\Gamma)\cap B_2(\mu)$, 
 define its harmonic extension $Hv':\Gamma\rightarrow\mathbb{R}$ with respect to $\mu$ as in (\ref{he}).
 Then by Proposition \ref{=}, we have that $Hv'\in\mathbb{HD}(\mu)$ and ${\cal E}^{\mu}(Hv',Hv')=\mathcal{E}^{\partial\Gamma,\mu}(v',v')$.
 Let $\tilde{F}_{d}:V(\Gamma_{d})\rightarrow\Gamma$
 be a quasi-isometry such that $F_{d}\circ\tilde{F}_{d}$ and $\tilde{F}_{d}\circ F_{d}$
 are within bounded distance from $id_{\Gamma_{d}}$ and $id_{\Gamma}$ respectively.
 Since $F_{d}$ continuously %OK
  extends to the identity on $\partial\Gamma$,
 $\tilde{F}_{d}$ does so as well.
 Then we have that
 \begin{align*}
 {\cal E}^{\mu}(Hv',Hv')\asymp {\cal E}^{{\rm SRW}}(Hv',Hv')\asymp
 \|d(Hv'\circ \tilde{F}_{d})\|^2_{\ell_2(E(\Gamma_{d}))}<+\infty.
 \end{align*}
 By Theorem \ref{besov}, the function $Hv'\circ \tilde{F}_{d}$ has a limit along ${\cal H}_{d}$-almost
 every geodesics and the limiting function $u':\partial \Gamma\rightarrow\mathbb{R}$ belongs to $B_2(d)$.
 On the other hand, %it is shown in 
 by \cite[Lemma 2.2]{Ka}, %that
  for any sequence $(g_n)\subset\Gamma$
 converging to a point $\eta\in\partial\Gamma$,
 %the sequence $(\nu_{g_n})$ weakly converges to a Dirac measure at $\eta$. In particular, 
 we have that 
 $\lim_{n\to\infty}Hv'(g_n)=v'(\eta)$.
  Now $u'$ is the limit of $Hv'\circ \tilde{F}_{d}$ along ${\cal H}_{d}$-almost
 every geodesics,
 and $(\tilde{F}_{d}(g_n))$ converges to $\eta$ whenever $(g_n)$ converges to $\eta$.
 Hence we get that $u'=v'\ {\cal H}_{d}$-{\it a.e.,}
  and this implies the conclusion.
 \qed\par
 \end{proof}
 \medskip
 %{\sc The explanation below is a bit obscure. For instance the map $N^*$ is defined in \eqref{n*} for a quasi-isometry between simplicial complexes and the quasi-isometry we use here is $c_0$ which is only introduced after $N^*$ is used.} 
 
 Now we wish to relate the two Besov spaces $B_2(d)$ and $B_2(d')$ for $d,d'\in J_{AR}(\partial\Gamma)$.
 To do so, we use
 a bijective linear map between them given in \cite{BP}, 
 and show that it coincides with the identity map on ${\cal C}$. 
 We now recall the construction of a Banach isomorphism between
 $(B_2(d)/\sim,\mathcal{E}^{\partial\Gamma,d})$ and $(B_2(d')/\sim,\mathcal{E}^{\partial\Gamma,d'})$ in \cite{BP}.
 Let $u\in B_2(d)$, and define $g:V(\Gamma_{d})\rightarrow\mathbb{R}$ as in (\ref{exx}).
 Now we define $\tilde{T}(d\to d'):B_2(d)%(\mathcal{E}^{\partial\Gamma,d},B_2(d))
 \rightarrow
 B_2(d')%(\mathcal{E}^{\partial\Gamma,d'},B_2(d'))
 $ as follows: $$\tilde{T}(d\to d')u:=I^{d'}(g\circ c_0)\in B_2(d').$$
 See \eqref{c0} for the definition of $c_0$.
 \begin{lem}\label{nap}
  For any $d,d'\in J_{AR}(\partial\Gamma)$,
  the linear map $\tilde{T}(d\to d'):B_2(d) \to B_2(d')$ introduced above satisfies
 $\tilde{T}(d\to d')|_{{\cal C}}={\rm Id}_{{\cal C}}$. Moreover, it induces a Banach isomorphism $T(d\to d')$ between $(B_2(d)/\sim,\mathcal{E}^{\partial\Gamma,d})$ and $(B_2(d')/\sim,\mathcal{E}^{\partial\Gamma,d'})$.
% there exists a linear map $\tilde{T}(d\to d'):B_2(d)%(\mathcal{E}^{\partial\Gamma,d},B_2(d)) \to B_2(d')$ with $\tilde{T}(d\to d')|_{{\cal C}}={\rm Id}_{{\cal C}}$ which induces a Banach isomorphism $T(d\to d')$ between $(B_2(d)/\sim,\mathcal{E}^{\partial\Gamma,d})$ and $(B_2(d')/\sim,\mathcal{E}^{\partial\Gamma,d'})$.
 \end{lem}
 \begin{proof}
 %We first recall the construction of a Banach isomorphism between
 %$(B_2(d)/\sim,\mathcal{E}^{\partial\Gamma,d})$ and $(B_2(d')/\sim,\mathcal{E}^{\partial\Gamma,d'})$ given in \cite{BP}.
 %Let $u\in B_2(d)$, and define $g:V(\Gamma_{d})\rightarrow\mathbb{R}$ as in (\ref{exx}).
 %Now we define $\tilde{T}(d\to d'):B_2(d)%(\mathcal{E}^{\partial\Gamma,d},B_2(d))
 %\rightarrow
% B_2(d')%(\mathcal{E}^{\partial\Gamma,d'},B_2(d'))
 %$ as follows: $\tilde{T}(d\to d')u:=I^{d'}(N^*(g))\in B_2(d')$.
 %See \eqref{n*} for the definition of $N^*$.
  It is shown in \cite{BP} that the linear map $\tilde{T}(d\to d')$ induces a Banach isomorphism
 between $(B_2(d)/\sim,\mathcal{E}^{\partial\Gamma,d})$ and $(B_2(d')/\sim,\mathcal{E}^{\partial\Gamma,d'})$.
 Therefore, it suffices to prove that $\tilde{T}(d\to d')|_{{\cal C}}={\rm Id}_{{\cal C}}$.\par 
  Under the identification (\ref{coho}), we have that 
$N^*(g)=g\circ c_0$, where $c_0:V(\Gamma_{d'})\rightarrow V(\Gamma_{d})$
  is a quasi-isometry which continuously extends to the identity on $\partial\Gamma$. 
  See the third statement in Theorem \ref{hf}.
 When $u\in{\cal C}$, it is obvious that for any sequence $(h_n)\subset\Gamma_{d}$
 converging to $\eta\in\partial\Gamma$, we have that $\lim_{n\to\infty}g(h_n)=u(\eta)$.
 This implies that $\tilde{T}(d\to d')|_{\cal C}={\rm Id}_{\cal C}$.
 It is shown in \cite{BP} that the linear map $\tilde{T}(d\to d')$ induces a Banach isomorphism
 between $(B_2(d)/\sim,\mathcal{E}^{\partial\Gamma,d})$ and $(B_2(d')/\sim,\mathcal{E}^{\partial\Gamma,d'})$.
 \qed\end{proof}
 \medskip
 We next prove that for any $\mu,\mu'
 \in M_2$, there exists an isomorphism between the two Besov spaces
 $(\mathcal{E}^{\partial\Gamma,\mu},B_2(\mu))$ and $(\mathcal{E}^{\partial\Gamma,\mu'},B_2(\mu'))$ which coincides with the identity map on ${\cal C}$.
 Before proving this claim, we need some preparation.

 \begin{lem}\label{umu}
 \begin{description}
 \item[(1)] The functional space $({\cal E}^{\mu},\mathbb{HD}(\mu)/\sim)$ is a Hilbert space.
 \item[(2)] The linear map $\tilde{U}(\mu)|_{\mathbb{HD}(\mu)}:\mathbb{HD}(\mu)\to B_2(\mu)$ induces 
 a Banach isomorphism $U(\mu)$ between
 the two Hilbert spaces $(\mathbb{HD}(\mu)/\sim,{\cal E}^{\mu})$ and $(B_2(\mu)/\sim,\mathcal{E}^{\partial\Gamma,\mu})$.
 See Corollary \ref{new-new} for the definition of $\tilde{U}(\mu):\mathcal{F}^{\mu}\to B_2(\mu)$.
  %$(\mathcal{E}^{\partial\Gamma,\mu},B_2(\mu)/\sim)$ and $({\cal E}^{\mu},\mathbb{HD}(\mu)/\sim)$ are isomorphic.
  %We will denote the isomorphism by $U(\mu):(\mathcal{E}^{\partial\Gamma,\mu},B_2(\mu)/\sim)\rightarrow({\cal E}^{\mu},\mathbb{HD}(\mu)/\sim)$.
 \end{description}
  \end{lem}
  \begin{proof}
 We first prove the first claim.
 Recall that by Proposition \ref{RD},
  $\ell_2(\Gamma)$ is a closed subspace of the Hilbert space $({\cal F}^{\mu}/\sim,{\cal E}^{\mu})$.
 This fact together with the decomposition \eqref{royden}
 %${\cal F}^{\mu}=\ell_2(\Gamma)\bigoplus\mathbb{HD}(\mu)$ 
 implies the result.\par
 We next prove the second claim. 
% Define $\tilde{U}(\mu):{\cal F}^{\mu}\rightarrow B_2(\mu)$ as follows: for $f\in{\cal F}^{\mu}$, let $\tilde{U}(\mu)f$ be the limit of $f$ along a path of random walk driven by $\mu$.
 %By Proposition \ref{=}, 
 By Corollary \ref{new-new} and the decomposition \eqref{royden}, $\tilde{U}(\mu)|_{\mathbb{HD}(\mu)}$ is bijective.
 Moreover, by Proposition \ref{=}, for $f\in\mathbb{HD}(\mu)$ we have that
 \begin{align*}
 {\cal E}^{\mu}(f,f)=\mathcal{E}^{\partial\Gamma,\mu}(\tilde{U}(\mu)f,\tilde{U}(\mu)f).
 \end{align*}
 %In particular, $\tilde{U}(\mu)|_{\mathbb{HD}(\mu)}$ is injective.
 Thus, $\tilde{U}(\mu)|_{\mathbb{HD}(\mu)}$ induces a Banach isomorphism between
 $(B_2(\mu)/\sim,\mathcal{E}^{\partial\Gamma,\mu})$ and $(\mathbb{HD}(\mu)/\sim,{\cal E}^{\mu})$.
  \qed
  \end{proof}
  \medskip
 We now construct a bijective linear map between $B_2(\mu)$ and $B_2(\mu')$.
 By Lemma \ref{umu}, we already have two bijective linear maps $\tilde{U}(\mu)|_{\mathbb{HD}(\mu)}:\mathbb{HD}(\mu)\to B_2(\mu)$ and $\tilde{U}(\mu')|_{\mathbb{HD}(\mu')}:\mathbb{HD}(\mu')\to B_2(\mu')$. Therefore, we wish to relate 
 $\mathbb{HD}(\mu)$ and $\mathbb{HD}(\mu')$. We define a linear map ${\it HD}(\mu\rightarrow\mu'):\mathbb{HD}(\mu)\rightarrow\mathbb{HD}(\mu')$ as follows:
 for $f\in\mathbb{HD}(\mu)$, let $f=f_1+f_2$ be its Royden decomposition with respect to $\mu'$
 , where $f_1\in\ell_2(\Gamma)$
 and $f_2\in\mathbb{HD}(\mu')$. Now we define $${\it HD}(\mu\rightarrow\mu')f:=f_2,$$
  and $\tilde{T}(\mu\rightarrow\mu'): B_2(\mu)\to B_2(\mu')$ by
 \begin{align}\label{def:tmumu}
 \tilde{T}(\mu\rightarrow\mu'):=\tilde{U}(\mu')
 \circ {\it HD}(\mu\rightarrow\mu')\circ (\tilde{U}(\mu)|_{\mathbb{HD}(\mu)})^{-1}
 \end{align}
 \begin{prop}\label{mumu}
 For any $\mu,\mu'\in M_2$, the linear map
 $\tilde{T}(\mu\rightarrow\mu'):%(\mathcal{E}^{\partial\Gamma,\mu},B_2(\mu))\rightarrow(\mathcal{E}^{\partial\Gamma,\mu'},B_2(\mu'))
 B_2(\mu)\to B_2(\mu')$ constructed above
 satisfies $\tilde{T}(\mu\to\mu')|_{\cal C}={\rm Id}_{\cal C}$. Moreover,
 it induces a Banach isomorphism 
 $T(\mu\rightarrow\mu'):(B_2(\mu)/\sim,\mathcal{E}^{\partial\Gamma,\mu})\rightarrow(B_2(\mu')/\sim,\mathcal{E}^{\partial\Gamma,\mu'})$.
 \end{prop}
 \begin{proof}
 We first show that ${\it HD}(\mu\rightarrow\mu')$ induces an isomorphism between 
 $(\mathbb{HD}(\mu)/\sim,{\cal E}^{\mu})\rightarrow(\mathbb{HD}(\mu')/\sim,{\cal E}^{\mu'})$, which implies the second claim.
 Recall that by Proposition \ref{psc}, we have ${\cal F}^{\mu}={\cal F}^{\mu'}$
 and ${\cal E}^{\mu}(f,f)\asymp{\cal E}^{\mu'}(f,f)$ for $f\in{\cal F}^{\mu}={\cal F}^{\mu'}$.
% Define a linear map ${\it HD}(\mu\rightarrow\mu'):\mathbb{HD}(\mu)\rightarrow\mathbb{HD}(\mu')$ as follows: for $f\in\mathbb{HD}(\mu)$, let $f=f_1+f_2$ be its Royden decomposition with respect to $\mu'$, where $f_1\in\ell_2(\Gamma)$ and $f_2\in\mathbb{HD}(\mu')$. Now define ${\it HD}(\mu\rightarrow\mu')f:=f_2$. 
 When ${\it HD}(\mu\rightarrow\mu')f=0$, we have $f\in\ell_2(\Gamma)\cap\mathbb{HD}(\mu)$,
 hence $f=0$.
 Thus ${\it HD}(\mu\rightarrow\mu')$ is injective.
 On the other hand, take $g\in\mathbb{HD}(\mu')$ arbitrarily. Let $g=g_1+g_2$ be its Royden decomposition
 with respect to $\mu$, where $g_1\in\ell_2(\Gamma)$ and $g_2\in\mathbb{HD}(\mu)$.
 Then $g_2=-g_1+g$, hence we have ${\it HD}(\mu\rightarrow\mu')g_2=g$. Therefore ${\it HD}(\mu\rightarrow\mu')$
 is surjective. Moreover, for $f\in\mathbb{HD}(\mu)$ we have
 \begin{align*}
 {\cal E}^{\mu}(f,f)=\min_{h\in\ell_2(\Gamma)}{\cal E}^{\mu}(f+h,f+h)&\asymp
 \min_{h\in\ell_2(\Gamma)}{\cal E}^{\mu'}(f+h,f+h)\\
 &={\cal E}^{\mu'}({\it HD}(\mu\rightarrow\mu')f,{\it HD}(\mu\rightarrow\mu')f).
 \end{align*}
 Thus, ${\it HD}(\mu\rightarrow\mu')$ induces an isomorphism between 
 $(\mathbb{HD}(\mu)/\sim,{\cal E}^{\mu})\rightarrow(\mathbb{HD}(\mu')/\sim,{\cal E}^{\mu'})$.\par
% Hence, we obtain the conclusion if we prove $\tilde{U}(\mu')
% \circ {\it HD}(\mu\rightarrow\mu')\circ (\tilde{U}(\mu)|_{\mathbb{HD}(\mu)})^{-1}
 %:B_2(\mu)\rightarrow B_2(\mu')$ coincides with the identity map on ${\cal C}$.
 %See Corollary \ref{new-new} for the definition of $\tilde{U}(\mu)$ and $\tilde{U}(\mu')$.\par
 We next show that $\tilde{T}(\mu\to\mu')|_{\cal C}={\rm Id}_{\cal C}$. Let $u\in{\cal C}$. By Lemma 2.2 in \cite{Ka}, 
 we have that for any sequence $(g_n)\subset\Gamma$ converging to $\eta\in\partial\Gamma$, 
 $\lim_{n\to\infty}Hu(g_n)=u(\eta)$.
 This implies that 
 $u=\tilde{U}(\mu)(Hu)$ and $u=\tilde{U}(\mu')(Hu)$.
 %$Hu=(\tilde{U}(\mu)|_{\mathbb{HD}(\mu)})^{-1}u$.
 Let $Hu=h_1+h_2$ be the Royden decomposition with respect to $\mu'$
  , where $h_1\in\ell_2(\Gamma)$ and $h_2\in\mathbb{HD}(\mu')$.
  Then by Proposition \ref{=}, we have that
  \begin{align*}
  u&=\tilde{U}(\mu')(Hu)=\tilde{U}(\mu')h_2
  =\tilde{U}(\mu')\circ {\it HD}(\mu\rightarrow\mu')(Hu)\\
  &=\tilde{U}(\mu')\circ {\it HD}(\mu\rightarrow\mu')\circ (\tilde{U}(\mu)|_{\mathbb{HD}(\mu)})^{-1}u.
  \end{align*}
  Therefore, we get the conclusion.
 \qed\end{proof}
 \medskip
 We will give the main results of this section below. 
  
\begin{theo}\label{iso}
 For any $d\in J_{AR}(\partial \Gamma)$ and $\mu\in M_2$,
 there exist linear maps $\tilde{T}(d\rightarrow\mu):
 B_2(d)\rightarrow B_2(\mu)$ and 
 $\tilde{T}(\mu\rightarrow d):B_2(\mu)\rightarrow B_2(d)$
 with $\tilde{T}(\mu\to d)|_{\cal C}=\tilde{T}(d\to\mu)|_{\cal C}={\rm Id}_{\cal C}$
 which induce Banach isomorphisms
 $T(d\rightarrow\mu):
 (B_2(d)/\sim,\mathcal{E}^{\partial\Gamma,d})\rightarrow(B_2(\mu)/\sim,\mathcal{E}^{\partial\Gamma,\mu})$ and 
 $T(\mu\rightarrow d):(B_2(\mu)/\sim,\mathcal{E}^{\partial\Gamma,\mu})\rightarrow (B_2(d)/\sim,\mathcal{E}^{\partial\Gamma,d})$.
\end{theo}
\begin{proof} 
Take a probability measure $\mu'$ on $\Gamma$ with a finite support. 
%{\sc Changed some formulations below.} 
It is shown in  \cite{BHM} that there exists a visual metric on $\partial\Gamma$ called the 
{\it Green visual metric} and denoted with $\rho(G^{\mu'})$ 
 which belongs to the Ahlfors-regular conformal gauge $J_{AR}(\partial\Gamma)$ and is such that
 $(\mathcal{E}^{\partial\Gamma,\mu'},B_2(\mu'))=(\mathcal{E}^{\partial\Gamma,\rho(G^{\mu'})},B_2(\rho(G^{\mu'})))$.
 See Corollary 1.2 and Section 3.2 in \cite{BHM}. 
 Now we choose $d\in J_{AR}(\partial\Gamma)$ and $\mu\in M_2$ arbitrarily.
 We define $\tilde{T}(d\rightarrow\mu):
 B_2(d)\rightarrow B_2(\mu)$ and 
 $\tilde{T}(\mu\rightarrow d): B_2(\mu)\rightarrow B_2(d)$ 
 by
 \begin{align*}
 \tilde{T}(d\rightarrow\mu)&:=\tilde{T}(\mu'\to\mu)\circ \tilde{T}(d\to \rho(G^{\mu'})),\\  
 \tilde{T}(\mu\rightarrow d)&:=\tilde{T}(\rho(G^{\mu'})\to d)\circ \tilde{T}(\mu\to\mu'),
 \end{align*}
 respectively.
 By Lemma \ref{nap} and Proposition \ref{mumu},
 it is obvious that the above two linear maps coincide with the identity on ${\cal C}$ and induce  Banach isomorphisms.
 \qed
\end{proof}
\medskip

\section{Besov spaces associated to random walks and the theory of Dirichlet forms}\label{besov-rw}
 In this section, we first prove that when the Ahlfors-regular conformal dimension of the 
 Gromov boundary $\partial\Gamma$ is strictly less than $2$,
 Besov spaces on $\partial\Gamma$ associated either to metrics $d\in J_{AR}(\partial\Gamma)$ and to
 random walks driven by $\mu\in M_2$
  give rise to regular Dirichlet forms on the boundary.
 Secondly, we will study a potential theoretic property of
 Hausdorff measures of metrics in $J_{AR}(\partial\Gamma)$ and harmonic measures of random walks on $\Gamma$.
 Specifically, we will prove that those Hausdorff measures and harmonic measures are smooth in a
 potential theoretic
 sense with respect to any regular Dirichlet form on the boundary given by the Besov spaces. 
 
 \subsection{Regularity of Besov spaces and smoothness of harmonic measures}\label{subsec:regu}
From now on, we will assume that
there exists a metric $d_0\in J_{AR}(\partial\Gamma)$ such that $q_0:=\dim(\partial\Gamma,d_0)<2$. 
In other words, we will assume that
 the Ahlfors-regular conformal dimension of $(\partial\Gamma,\rho_{\Gamma})$ is strictly less than 2. 
Let ${\bf Lip}_0$ be the set of Lipschitz functions with respect to $d_0$. 
By a straightforward computation, we can check that 
${\bf Lip}_0\subset B_2(d_0)$ under the assumption $q_0<2$.
By Lemma \ref{c}, this implies that 
\begin{align}\label{0=c}
{\bf Lip}_0\subset{\cal C},
\end{align}
which will be important in the proof of the regularity of Besov spaces.
 In Proposition \ref{mumu}, we also checked that isomorphisms between Besov 
spaces can be arranged in such a way that functions in ${\bf Lip}_0$ are invariant.
We now claim the regularity of Dirichlet forms associated to $d\in J_{AR}(\partial\Gamma)$ and
$\mu\in M_2$. 
\begin{theo}\label{reg}
Assume the Ahlfors-regular conformal dimension of $\partial\Gamma$ is strictly less than 2.
Then for any $d\in J_{AR}(\partial\Gamma)$ and any $\mu\in M_2$,
$(\mathcal{E}^{\partial\Gamma,d},B_2(d))$ and $(\mathcal{E}^{\partial\Gamma,\mu},B_2(\mu))$ are regular Dirichlet forms
 on $L^2(\partial\Gamma,{\cal H}_{d})$ and on $L^2(\partial\Gamma,\nu)$ respectively. 
\end{theo}
Before giving the proof, we introduce the following estimates which are reminiscent of the Poincar\'e inequality.
\begin{lem}\label{bpoin}
Let $d\in J_{AR}(\partial\Gamma)$.
 For $u\in B_2(d)$, let $g:V(\Gamma_{d})\rightarrow\mathbb{R}$ be the function defined in (\ref{exx}).
 Then, there exists a constant $C>0$ such that for any $u\in B_2(d)$ we have that
 \begin{align*}
 \|u-g(O)\|^2_{L^2(\mathcal{H}_{d})}\leq C\mathcal{E}^{\partial\Gamma,d}(u,u).
 \end{align*}
\end{lem}
\begin{proof}
 The claim immediately follows from Theorem 3.1 and Lemma 3.2 in \cite{BP}. See the argument below Lemma 3.2
 in \cite{BP}.
\end{proof}
\begin{lem}\label{ac}
 Let $(R_k)_{k\geq0}$ be a random walk on %$\Gamma$
 a non-amenable finitely generated group driven by a symmetric admissible probability measure $\mu$. %$\mu\in M_2$.
 Then, there exist constants $a>1$ and $C>0$ such that for any $g\in{\cal F}^{\mu}$ we have that
 \begin{align}\label{pi; mu}
 \sum_{k=0}^{\infty}\be^{\mu}[(g(R_{k+1})-g(R_k))^2]a^k\leq C{\mathcal E}^\mu(g,g)\,.
 \end{align}
%(This is actually true for any non-amenable group.) 
\end{lem}
\begin{proof}
Since every %non-elementary hyperbolic group
non-amenable Cayley graph satisfies a linear isoperimetric inequality, by Theorem 10.3 in \cite{Woe},
 there exists a constant $c>0$ such that $\mu^{*k}(id)\leq e^{-ck}$ for any $k\in\bn$.
 Moreover, using the elementary inequality
 $\mu^{*k}(x)\leq\sqrt{\mu^{*2k}(id)}$, we have that
 $\mu^{*k}(x)\leq e^{-ck}$ for any $x\in\Gamma$ and any $k\in\bn$.
 On the other hand, we have that
\begin{align*}
\sum_{k}\be^{\mu}[(g(R_{k+1})-g(R_k))^2]a^k=\sum_ka^k\sum_x\mu^{*k}(x)\sum_{y}\mu(x^{-1}y)(g(y)-g(x))^2.
\end{align*}
Combining this formula with the exponential decay of $\{\mu^{*k}(x)\}_{k\geq1}$,
 we get the desired estimate  for sufficiently small $a>1$ and some constant $C>0$.
 \qed
\end{proof}
\medskip

{\bf Proof of Theorem \ref{reg}.}
Since ${\bf Lip}_0\subset{\cal C}$, the set ${\cal C}$ separates points. Thus, ${\cal C}$ is $\|\cdot\|_{\infty}$-dense
in $C(\partial\Gamma)$ by the Stone-Weierstrass theorem. 
Therefore, we only need to show that ${\cal C}$ is
 dense both in $(B_2(d),%(\mathcal{E}^{\partial\Gamma,d}+\|\cdot\|^2_{L^2({\cal H}_{d})})^{1/2}
 (\mathcal{E}^{\partial\Gamma,d}_1)^{1/2})$
 and $(B_{2}(\mu),(\mathcal{E}^{\partial\Gamma,\mu}_1)^{1/2})$.
 %and this is proved in \cite[Proposition 3.13]{Cos}. 
 %But we will give another proof.
  We first prove the claim for $d\in J_{AR}(\partial\Gamma)$.
 By \cite[Proposition 3.13]{Cos}, %{\s%c Costea was used here.} 
 for any $u\in B_2(d)$ there exists a sequence $(w_n)\subset{\cal C}$ such that
 $\mathcal{E}^{\partial\Gamma,d}(u-w_n,u-w_n)\rightarrow0$. Since $\mathcal{E}^{\partial\Gamma,d}(v-c',v-c')=\mathcal{E}^{\partial\Gamma,d}(v,v)$ for any $v\in B_2(d)$
 and $c'\in\mathbb{R}$, it suffices to show that there exists a constant $c_n\in\mathbb{R}$ such that
 \begin{align}\label{c_n}
 \lim_{n\to\infty}\|u-w_n-c_n\|_{L^2({\cal H}_{d})}=0.
 \end{align}
  By Lemma \ref{bpoin}, there exists a constant 
 $C>0$ such that for any $v\in B_2(d)$
 \begin{align*}
 \|v-c(v)\|_{L^2({\cal H}_{d})}^2\leq C\mathcal{E}^{\partial\Gamma,d}(v,v),\ \ {\rm where}\ \ c(v)=\dfrac{1}{\mathcal{H}_d(\partial\Gamma)}\int_{\partial\Gamma}v(\xi)d\mathcal{H}_d(\xi).
 \end{align*}
 Choosing $c_n:=c(u-w_n)$,
 the above inequality together with \eqref{c_n} implies that ${\cal C}$ is
 dense in $(B_2(d),(\mathcal{E}^{\partial\Gamma,d}_1)^{1/2})$ for $d\in J_{AR}(\partial\Gamma)$. 
 %{\s%c Too quick: how was $c$ chosen?} 
 \par
  We next prove the claim for $\mu\in M_2$. Take $v\in B_2(\mu)$
  and let $Hv$ be its harmonic extension with respect to $\mu$ as in (\ref{he}).
  By applying Lemma \ref{ac} to $Hv$, we get that there exist constants $a>1$ and $C>0$ such that
 \begin{align*}
 \sum_k\be^{\mu}[(Hv(R_{k+1})-Hv(R_k))^2]a^k\leq C{\mathcal E}^\mu(Hv,Hv)=C \mathcal{E}^{\partial\Gamma,\mu}(v,v).
 \end{align*}
    Since the limit of $Hv$ along a path of $(R_n)$ coincides with $v$, 
    by the Cauchy-Schwarz inequality we get 
 we obtain that 
  \begin{align}\label{cs-ineq}
   \|v-Hv(id)\|_{L^2(\nu)}^2%\|v-g(O)\|_{L^2(\nu)}^2
   &=\be^{\mu}[(v(R_{\infty})-Hv(id))^2]%\be^{\mu}[(v(R_{\infty})-g(O))^2]
   \nonumber\\
  &=\be^{\mu}\left[ \left(\sum_{k=0}^{\infty}(Hv(R_{k+1})-Hv(R_k))\right)^2\right]\nonumber\\
   &=\be^{\mu}\left[ \left(\sum_{k=0}^{\infty}(Hv(R_{k+1})-Hv(R_k))a^{k/2}\cdot a^{-k/2}\right)^2\right]\nonumber\\
&\leq \left(\sum_{k=0}^{\infty}a^{-k}\right)\cdot\sum_{k=0}^{\infty}\be^{\mu}\left[(Hv(R_{k+1})-Hv(R_k))^2\right]a^k\nonumber\\
  &\leq C \sum_{k=0}^{\infty}\be^{\mu}[(Hv(R_{k+1})-Hv(R_k))^2]a^k.
  \end{align}
  Therefore, we get that
  \begin{align}\label{nu}
  %\|v-g(O)\|_{L^2(\nu)}^2
  \|v-Hv(id)\|_{L^2(\nu)}^2\leq C\mathcal{E}^{\partial\Gamma,\mu}(v,v).
  \end{align}
  Thus, we can get the density of $\cal C$ in $(B_{2}(\mu),(\mathcal{E}^{\partial\Gamma,\mu}_1)^{1/2})$ 
 by a similar argument as for the density in $(B_2(d), (\mathcal{E}^{\partial\Gamma,d}_1)^{1/2})$ .
 \qed

\subsection{Smooth measures and measures of finite energy integral}
In what follows, we will study some potential theoretic property of Hausdorff measures associated to
the Ahlfors-regular conformal gauge $J_{AR}(\partial\Gamma)$ and harmonic measures associated to $M_2$.
 In this subsection, we give several general definitions about measures and potential theory of Dirichlet forms.
   We will explain later their probabilistic interpretation,
 especially how those measures arise in the study of symmetric Markov processes and their time changes. 
\\
 
 Let $E$ be a locally compact Hausdorff space, and $m$ be a positive Radon measure on $E$.
 Assume that we have a regular Dirichlet form $({\cal E},{\cal F})$ on $L^2(E,m)$.
 \begin{defn}
 For an open subset $U\subseteq E$, we define
 \begin{align*}
 L_U:=\{u\in{\cal F};\ u\geq1\ m\mathchar`-{\rm a.e.\ on}\ U\},
 \end{align*}
 and
 \begin{align*}
 {\rm Cap}(U):=\begin{cases}
 \inf_{u\in L_U}{\cal E}_1(u,u),\ \ {\rm if}\ L_U\neq\emptyset\\
 \infty,\ \ \ \ \ \ \ \ \ \ \ \ \ \ \ \ \ \ {\rm if}\ L_U=\emptyset.
 \end{cases}
 \end{align*}
 For any subset $A\subseteq E$, we define
 \begin{align*}
 {\rm Cap}(A)=\inf_{U:{\it open}, A\subseteq U}{\rm Cap}(U).
 \end{align*}
 The value of ${\rm Cap}(A)$ is called (1-)capacity of $A$.
 \end{defn}
 \begin{defn}\label{sm-measure}
 Let $\kappa$ be a positive Borel measure on $E$. We say that $\kappa$ is smooth 
 with respect to a regular Dirichlet form $(\ce,\cf)$ 
 when the following conditions
 are satisfied.
 \begin{description}
 \item[(1)] $\kappa(B)=0$ whenever ${\rm Cap}(B)=0$ and
 \item[(2)] there exists an increasing sequence $(C_n)$ of closed subsets of $E$ such that
 \begin{align*}
 \kappa(C_n)&<\infty\ \ {\rm for\ any}\ n\in\mathbb{N}\ {\rm and}\\
 \lim_{n\to\infty}{\rm Cap}(K\setminus C_n)&=0\ \ {\rm for\ any\ compact\ subset}\ K\subseteq E.
 \end{align*}
 \end{description}
 \end{defn}
 \begin{remark}\label{rem-cpt}
 In what follows, we will choose the Gromov boundary 
 $\partial\Gamma$ as the state space $E$.
 Since $\partial\Gamma$ is compact, the second condition of smoothness is always satisfied in our framework.
  %thus the second condition of 
 %smoothness is not important in this paper.
 \end{remark}
  \begin{defn}
 Let $\kappa$ be a positive Radon measure on $E$. We say that $\kappa$ is of {\it finite energy integral}
 with respect to a regular Dirichlet form $(\ce,\cf)$ 
 if there exists a constant $C>0$ such that for any $v\in{\cal F}\cap C_0(E)$ we have that
 \begin{align*}
 \int_E|v(x)|d\kappa(x)\leq C\sqrt{{\cal E}_1(v,v)}.
 \end{align*}
 Note that by Riesz's representation theorem, a positive Radon measure $\kappa$ on $E$ is of finite energy integral
 with respect to a regular Dirichlet form $(\ce,\cf)$ 
 if and only if for each $\alpha>0$, there exists a unique function $U_{\alpha}\kappa\in {\cal F}$ such that
 \begin{align*}
 \int_Ev(x)d\kappa(x)={\cal E}_{\alpha}(U_{\alpha}\kappa,v)
 \end{align*}
 for any $v\in{\cal F}\cap C_0(E)$.
  \end{defn}
  Measures of finite energy integral are known to form a subclass of smooth measures.
  \begin{prop}\cite[Section 2.2]{FOT}
 Any positive Radon measure $\kappa$ which is of finite energy integral with respect to a regular Dirichlet form
 $({\cal E},{\cal F})$ is smooth with respect to $({\cal E},{\cal F})$. 
  \end{prop}
  
Now we introduce the following notions concerning capacity and smoothness of measures with respect to
regular Dirichlet forms associated to $d\in J_{AR}(\partial\Gamma)$ and $\mu\in M_2$.
\begin{defn}\label{SS}
Assume the Ahlfors-regular conformal dimension of $\partial\Gamma$ is strictly less than 2.
\begin{description}
\item[(1)]Thanks to Theorem \ref{reg}, for any $d\in J_{AR}(\partial\Gamma)$ and any $\mu\in M_2$,
 we can define 
 ${\cal S}(\partial \Gamma,d)$ $({\cal S}(\partial \Gamma,\mu)$, resp.$)$ as the collection of all smooth measures
 with respect to the regular Dirichlet form $(\mathcal{E}^{\partial\Gamma,d},B_2(d))$
 on $L^2(\partial\Gamma,{\cal H}_{d})$ $((\mathcal{E}^{\partial\Gamma,\mu},B_2(\mu))$
 on $L^2(\partial\Gamma,\nu)$, resp.$)$.
 \item[(2)]Similarly, we define ${\cal S}_0(\partial \Gamma,d)$ $({\cal S}_0(\partial \Gamma,\mu)$, resp.$)$ as
 the collection of all measures
 of finite energy integral with respect to the regular Dirichlet form $(\mathcal{E}^{\partial\Gamma,d},B_2(d))$
 on $L^2(\partial\Gamma,{\cal H}_{d})$ $((\mathcal{E}^{\partial\Gamma,\mu},B_2(\mu))$
 on $L^2(\partial\Gamma,\nu)$, resp.$)$.
 \item[(3)]We also define $0(\partial \Gamma,d)$ $(0(\partial \Gamma,\mu)$, resp.$)$
 as the collection of all subsets of $\partial\Gamma$
 with zero capacity
  with respect to the regular Dirichlet forms $(\mathcal{E}^{\partial\Gamma,d},B_2(d))$
 on $L^2(\partial\Gamma,{\cal H}_{d})$ and \\ 
 $(\mathcal{E}^{\partial\Gamma,\mu},B_2(\mu))$
 on $L^2(\partial\Gamma,\nu)$.
 \end{description}
\end{defn}

\subsection{Poincar\'e-type inequalities and sets of $0$ capacity} 

Below we consider two random walks with respective driving measures $\mu$ and $\mu'$ 
(always in $M_2$) and respective harmonic measures $\nu$ and $\nu'$. 
Now recall that by Theorem \ref{iso}, for any $\mu\in M_2$, $d\in J_{AR}(\partial\Gamma)$ and any $u\in{\cal C}$, 
$\mathcal{E}^{\partial\Gamma,\mu}(u,u)$ and $\mathcal{E}^{\partial\Gamma,d}(u,u)$ are comparable up to multiplicative constants.
\begin{prop}\label{pi}
Let $\mu,\mu'\in M_2$ and $d\in J_{AR}(\partial\Gamma)$.
Denote by $\nu$ and $\nu'$ the harmonic measures of the random walks driven by $\mu$ and $\mu'$ respectively.
Then, there exists a constant $C>0$ such that
for any $u\in {\cal C}$, we have that
\begin{align*}
\left(\int_{\partial\Gamma} u\, d\nu-\int_{\partial\Gamma}  u\, d\nu'\right)^2&\leq C\mathcal{E}^{\partial\Gamma,\mu}(u,u),\ {\it and}\\
\left(\int_{\partial\Gamma} u\, d\nu-\mathcal{H}_{d}(\partial\Gamma)^{-1}\int_{\partial\Gamma}  u\, d{\mathcal H}_d\right)^2&\leq C\mathcal{E}^{\partial\Gamma,\mu}(u,u).
\end{align*}
\end{prop}
\begin{proof}
\if0
Let $u\in {\cal C}$ and define $g:V(\Gamma_{d})\rightarrow\mathbb{R}$ as in (\ref{exx}). 
Then by Lemma \ref{bpoin}, we get that 
\begin{align}\label{pi1}
\left(g(O)-\int u\, d{\mathcal H}_d\right)^2\leq C\mathcal{E}^{\partial\Gamma,d}(u,u).
\end{align}
for sufficiently large $C>0$.\fi
%On the other hand, by combining (\ref{nu}) with H\"older's inequality and taking a sufficiently large constant $C$ again, we obtain that 
 Let $u\in\mathcal{C}$ and $F_{d}:\Gamma\rightarrow V(\Gamma_{d})$
 be a quasi-isometry which continuously extends to the identity on $\partial\Gamma$.
 Moreover, we will assume that $F_{d}(id)=O$ without loss of generality.
 Then if we define a function $g:V(\Gamma_{d})\to\mathbb{R}$ as in \eqref{exx}, we have $g\circ F_{d}\in {\cal F}^{\mu}$. 
 %Note that by the Cauchy-Schwarz inequality, for $a>1$ we get 
%\begin{align*}
%\left(\sum_{k=0}^{\infty}(g\circ F_{d}(R_{k+1})-g\circ F_{d}(R_k))\right)^2&\leq
 %\left(\sum_{k=0}^{\infty}(g\circ F_{d}(R_{k+1})-g\circ F_{d}(R_k))a^{k/2}\cdot a^{-k/2}\right)^2\\
%&\leq \left(\sum_{k=0}^{\infty}(g\circ F_{d}(R_{k+1})-g\circ F_{d}(R_k))^2
%a^k\right)\cdot\left(\sum_{k=0}^{\infty}a^{-k}\right).
%\end{align*}
   Since the limit of $g\circ F_{d}$ along a path of $(R_n)$ coincides with $u$, 
   by the same argument as in \eqref{cs-ineq}, we get that
   \begin{align*}
    \|u-g(O)\|_{L^2(\nu)}^2=
   \|u-g\circ F_{d}(id)\|_{L^2(\nu)}^2&=\be^{\mu}[(u(R_{\infty})- g\circ F_{d}(id))^2]\\
   &\leq  C\sum_{k=0}^{\infty}\be^{\mu}[(g\circ F_{d}(R_{k+1})-g\circ F_{d}(R_k))^2]a^k.
  \end{align*}
  By Lemma \ref{ac}, we get that
  \begin{align*}
  \|u-g(O)\|_{L^2(\nu)}^2\leq C{\cal E}^{\mu}(g\circ F_{d},g\circ F_{d}).
  \end{align*}
  Since 
  \begin{align*}
  {\cal E}^{\mu}(g\circ F_{d},g\circ F_{d})\asymp \|dg\|^2_{\ell_2(E(\Gamma_{d}))} \asymp \mathcal{E}^{\partial\Gamma,d}(u,u),
  \end{align*}
  we get that
  \begin{align}\label{pi; corrected}
  \|u-g(O)\|_{L^2(\nu)}^2\leq C\mathcal{E}^{\partial\Gamma,d}(u,u).
  \end{align}
 Recalling $g(O):=\mathcal{H}_{d}(\partial\Gamma)^{-1}\int_{\partial\Gamma}  u\, d{\mathcal H}_d$, we get that
\begin{align}\label{pi2}
\left(\int_{\partial\Gamma} u\, d\nu-\mathcal{H}_{d}(\partial\Gamma)^{-1}\int_{\partial\Gamma}  u\, d{\mathcal H}_d\right)^2\leq C\mathcal{E}^{\partial\Gamma,d}(u,u).
\end{align}
We obtain the first estimate by applying the inequality (\ref{pi2}) to two harmonic measures $\nu,\nu'$ and
 combining them with the triangle inequality.
\qed
\end{proof}

\begin{theo}\label{s}
Assume the Ahlfors-regular conformal dimension of $\partial\Gamma$ is strictly less than 2.
For any $d\in J_{AR}(\partial\Gamma)$ and $\mu\in M_2$,
we have that ${\cal S}_0(\partial\Gamma,d)={\cal S}_0(\partial\Gamma,\mu)$.
 The similar statement holds for ${\cal S}(\partial\Gamma,d)$ and ${\cal S}(\partial\Gamma,\mu)$, also for
 $0(\partial\Gamma,d)$ and $0(\partial\Gamma,\mu)$.
 We will denote those common sets by ${\cal S}_0(\partial\Gamma), {\cal S}(\partial\Gamma)$ and 
 $0(\partial\Gamma)$ respectively.
\end{theo}
\begin{proof}Choose $d\in J_{AR}(\partial\Gamma)$ and $\mu\in M_2$ arbitrarily.
 The first claim together with Theorem 2.2.3 in \cite{FOT} implies the third one.
 The second claim follows from the third one and the definition of smooth measures. See Definition \ref{sm-measure} and Remark \ref{rem-cpt}.
 Therefore, it suffices to prove the first claim.
 Let $\kappa$ be a positive Radon measure and
 assume that $\kappa\in{\cal S}_0(\partial\Gamma,d)$, namely,
 there exists a constant $C>0$ such that for any $v\in{\cal C}$
 \begin{align}\label{nuh}
 \int_{\partial\Gamma}|v(\xi)|d\kappa(\xi)\leq C\left(\mathcal{E}^{\partial\Gamma,d}_1(v,v)\right)^{1/2}.
 \end{align}
 We will show that $\kappa\in{\cal S}_0(\partial\Gamma,\mu)$.
 Recall that for any $v\in{\cal C}$, we have that $\mathcal{E}^{\partial\Gamma,d}(v,v)\asymp \mathcal{E}^{\partial\Gamma,\mu}(v,v)$.
 By Lemma \ref{bpoin}, we have that
 \begin{align*}
 \inf_{c}\|v-c\|^2_{L^2({\cal H}_{d})}\leq C\mathcal{E}^{\partial\Gamma,d}(v,v)\leq C'\mathcal{E}^{\partial\Gamma,\mu}(v,v),
 \end{align*}
 for a constant $C'>0$ independent of $v$.
 On the other hand, we have that
 \begin{align*}
  \inf_{c}\|v-c\|^2_{L^2({\cal H}_{d})}&=\inf_c\left({\cal H}_{d}(\partial\Gamma)\cdot
  c^2-2c\left(\int_{\partial\Gamma}vd{\cal H}_{d}\right)+\|v\|^2_{L^2({\cal H}_{d})}\right)\\
  &=\|v\|^2_{L^2({\cal H}_{d})}-{\cal H}_{d}(\partial\Gamma)^{-1}\left(\int_{\partial\Gamma}vd{\cal H}_{d}\right)^2.
%\\  &\geq  \|v\|^2_{L^2({\cal H}_{d})}-{\cal H}_{d}(\partial\Gamma)^{-1}\|v\|_{L^1({\cal H}_{d})}^2.
 \end{align*}
 Thus, by Proposition \ref{pi} we get that
 \begin{align*}
 \|v\|^2_{L^2({\cal H}_{d})}&\leq C'\mathcal{E}^{\partial\Gamma,\mu}(v,v)+{\cal H}_{d}(\partial\Gamma)^{-1}\left(\int_{\partial\Gamma}vd{\cal H}_{d}\right)^2%\|v\|_{L^1({\cal H}_{d})}^2
 \\
 &\leq C'\mathcal{E}^{\partial\Gamma,\mu}(v,v)+C{\cal H}_{d}(\partial\Gamma)\cdot
 \left(\|v\|_{L^1(\nu)}+\sqrt{\mathcal{E}^{\partial\Gamma,\mu}(v,v)}\right)^2\\
 &\leq \left(C'+2C{\cal H}_{d}(\partial\Gamma) \right)\cdot \mathcal{E}^{\partial\Gamma,\mu}(v,v)
 +2{\cal H}_{d}(\partial\Gamma)\|v\|^2_{L^1(\nu)}\\
 &\leq \left(C'+2C{\cal H}_{d}(\partial\Gamma) \right)\cdot \mathcal{E}^{\partial\Gamma,\mu}(v,v)
 +2{\cal H}_{d}(\partial\Gamma)\|v\|^2_{L^2(\nu)},
 \end{align*}
 where we used Jensen's inequality in the last step.
 By substituting the above estimate for the inequality (\ref{nuh}), we get that $\kappa\in{\cal S}_0(\partial\Gamma,\mu)$.
 By using the estimate (\ref{nu}) and Proposition \ref{pi},
 the converse claim can be proved similarly. 
 \qed
\end{proof}\ 
\begin{remark}\label{rem6}
Note that both of ${\cal S}(\partial \Gamma)$ and ${\cal S}_0(\partial\Gamma)$ contain any Hausdorff measure
${\cal H}_{d}$
of a metric $d\in J_{AR}(\partial\Gamma)$
 and any harmonic measure $\nu$ of a random walk driven by a probability measure $\mu\in M_2$
 since ${\cal H}_{d}$ is smooth with respect to $(\mathcal{E}^{\partial\Gamma,d},B_2(d))$
 and $\nu$ is smooth with respect to $(\mathcal{E}^{\partial\Gamma,\mu},B_2(\mu))$.
\end{remark}

\section{Time changes of processes associated with Dirichlet forms}\label{time-change} 
In Section 4, we introduced notions such as smooth measures and measures of finite energy integral,
 which concern the relation between measures and potential theory of Dirichlet forms.  
 In this section, we will introduce several general facts on Dirichlet forms to
 explain the probabilistic interpretation of smooth measures and measures of finite energy integral,
 which are heavily related to time changes of symmetric Markov processes.\par
 
 \subsection{Positive continuous additive functionals}\label{ssc:pcaf}

Let $E$ be a locally compact separable metric space and $m$ be a positive Radon measure on $E$ with full support.
 If we are given a regular Dirichlet form $(\ce,\cf)$ on $L^2(E,m)$,
 we can associate an $m$-symmetric Hunt process $(X_t,\mathbf{P}_x)$ on $E$. 
 In the theory of Dirichlet forms, it is well-known that there is a relationship,
 called {\it the Revuz correspondence},
 between {\it smooth measures} on $E$ and {\it positive continuous additive functionals} (PCAFs in short). 
 See \cite{FOT,CF} for the precise definition of PCAFs.
 We denote the set of smooth measures on $E$ by $S$, and the set of
 PCAFs of the Hunt process $X$ by $\mathbb{A}_c^{+}$. In what follows, we denote the extended Dirichlet space
 of $(\ce,\cf)$ by $(\cf_e,\ce)$.
 
 For a given $A\in\mathbb{A}_c^{+}$, define the measure $\kappa_A$, called the {\it Revuz measure} of $A$, by the following formula: for any $f\in\cb_+(E):=\{f:E\to\mathbb{R}_{\geq0}\ ;\ f\ {\rm is\ Borel\ measurable}\}$,
\begin{eqnarray*}
\langle\kappa_A,f\rangle&=&\lim_{t\downarrow0}\frac{1}{t}\mathbf{E}_m\left[\int_0^tf(X_s)dA_s\right],
\end{eqnarray*}
where $\mathbf{E}_m[\cdot]:=\int \mathbf{E}_x[\cdot]m({\rm d}x)$.
Two positive additive functionals $A,B\in\mathbb{A}_c^+$ are called {\it $m$-equivalent}
 if $\mathbf{P}_m(A_t=B_t)=1$ for every $t>0$.
\begin{theo}\cite[Theorem 5.1.3]{FOT}\cite[Theorem 4.1.1]{CF}\label{7.1}
\begin{description}
 \item[(1)] For any $A\in\mathbb{A}_c^+$, $\kappa_A\in S$.
 \item[(2)] For any $\kappa\in S$, there exists $A\in\mathbb{A}_c^+$ which satisfies $\kappa_A=\kappa$. $A$ is unique 
 up to $m$-equivalence.
 \item[(3)] For $A\in\mathbb{A}_c^+$ and $\kappa\in S$, the following conditions are equivalent.\\
 (a) $\kappa_A=\kappa$\\
 (b) For any $f,h\in\cb_+(E)$ and any $t>0$,
 \begin{align}\label{revuz-corr}
 \mathbf{E}_{h\cdot m}\left[\int_0^{t}f(X_s)dA_s\right]=\int_0^t\langle T_{s}h,f\cdot\kappa\rangle ds,
 \end{align}
 where $(T_s)$ is the semigroup associated to the process $X$.
\end{description}
\end{theo}

 \begin{exa}\label{abs-pcaf}
Let $(\ce,\cf)$ be a regular Dirichlet form on $L^2(E,m)$ and $(X_t)$ be the corresponding $m$-symmetric Hunt process. Then for a nonnegative function $g\in L^1(E,m)$, the absolutely continuous measure
$$g\cdot m(dx):=g(x)m(dx)$$
is smooth with respect to $(\ce,\cf)$. This immediately follows from the absolute continuity and the fact that $m$ itself is a smooth measure with respect to $(\ce,\cf)$.\par
Let us prove that the Revuz correspondence relates $g\cdot m$ to the PCAF
$$A^g_t:=\int_0^t g(X_s)ds.$$ 
We will verify \eqref{revuz-corr}. Let $f,h\in\cb_+(E)$ and $t>0$. By Fubini's theorem, we have that
\begin{align*}
\mathbf{E}_{h\cdot m}\left[\int_0^{t}f(X_s)dA^g_s\right]&=
{\bf E}_{h\cdot m} \left[\int_0^t f(X_s)g(X_s)ds\right]\\
&=\int_0^t ds\,  {\bf E}_{h\cdot m}[(fg)(X_s)].
\end{align*}
Using the $m$-symmetry of $X$, we finally get that
\begin{align*}
\int_0^t ds\,  {\bf E}_{h\cdot m}[ (fg)(X_s)]&=\int_0^t ds\, \langle T_s(fg), h\cdot m\rangle
=\int_0^t ds\, \langle T_sh, fg\cdot m\rangle ,
\end{align*}
which is \eqref{revuz-corr}.
\end{exa}

The following theorem gives the probabilistic interpretation of sets of zero capacity.
\begin{theo}\cite[Theorem 4.2.1]{FOT}
 Let $(\ce,\cf)$ be a regular Dirichlet form on $L^2(E,m)$ and $(X_t)$ be the $m$-symmetric Hunt process
 which corresponds to $(\ce,\cf)$. Then for a set $C\subseteq E$, we have that ${\rm Cap}(C)=0$
 if and only if
 \begin{align*}
 \mathbf{P}_x(X_t\notin C\ {\rm for\ any}\ t\in[0,\infty))=1
 \end{align*}
 for $m$-almost every $x$. 
\end{theo}
For a PCAF $A^{\kappa}\in\mathbb{A}_c^+$ whose Revuz measure is
 $\kappa\in S$, define its right continuous inverse $(A^{\kappa})^{-1}_t$ by
\begin{eqnarray*}
(A^{\kappa})^{-1}_t:=\inf\{s>0; A_s^{\kappa}>t\}.
\end{eqnarray*}
Then the time-changed process $Y_t:=X\circ (A^{\kappa})^{-1}_t$ is a $\kappa$-symmetric Markov process. The new process $(Y_t)$ may not be defined for any $t>0$ even if $(X_t)$ is so: 
when $A^{\kappa}_{\infty}:=\lim_{t\to\infty}A^{\kappa}_t$ is finite, $(Y_t)$ is killed at $A^{\kappa}_{\infty}$.
%{\sc Should we not indicate $Y$ may not be defined for all times?} 
 The Dirichlet form of $Y_t$ is characterized by the following theorem.

\begin{theo}\label{tc}\cite[Theorem 6.2.1]{FOT}\cite[Theorem 5.2.2]{CF}
 For $B\in\cb(E)$, define the hitting time $\sigma_B$ of $B$ by 
 \begin{eqnarray*}
 \sigma_B:=\inf\{t>0; X_t\in B\}.
 \end{eqnarray*} 
 Let $A^{\kappa}$ be a PCAF whose Revuz measure is $\kappa\in S$, and $(A^{\kappa})^{-1}_t$ be the right continuous inverse
 of $A^{\kappa}$. Denote the support of $A^{\kappa}$ by F. (See \cite[page 234]{FOT} for the definition.)
 Define $(\check{\ce},\check{\cf})$ by 
 \begin{align*}
 %\check{\cf}_e|_F&:=&\cf_e|_F,\\
 \check{\cf}&:=\{u\in L^2(F,\kappa)\ ;\ u=\tilde{\varphi}\ \kappa\mathchar`-a.e.\ {\rm on}\ F\ 
 {\rm for\ some}\ \varphi\in{\cal F}_e\},\\
  \check{\ce}(u,v)&:=\ce(\mathbf{H}_Fu,\mathbf{H}_Fv)\ \ {\it for}\ u,v\in\check{\cf},
 \end{align*}
 where $\tilde{\varphi}$ is a quasi-continuous modification of $\varphi$ (See Theorem 2.1.3 in \cite{FOT}).
  and
 $\mathbf{H}_Ff(x):=\mathbf{E}_x\left[f(X_{\sigma_F}); \sigma_F<\infty\right],$\ $x\in E,\ f\in\cb_+(E)$.
 Then $(\check{\ce},\check{\cf})$ is the regular Dirichlet form on $L^2(F,\kappa)$ which
 corresponds to $Y_t:=X\circ (A^{\kappa})^{-1}_t$.
 We will call $(\check{\ce},\check{\cf})$ the trace of $(\ce,\cf)$ on $F$ with respect to $\kappa$.
\end{theo}
%Note that in the above statement, the regular Dirichlet form $(\check{\ce},\check{\cf})$
%on $L^2(F,\mu)$ does not depend the choice of $m$, which is the reference measure of 
 %the original Dirichlet form $(\ce,\cf)$, in the following sense:
 %assume that there is a positive Radon measure $m'$ such that $()$

%%%%%%%%

 \section{Reflecting random walks on $\Gamma\cup\partial\Gamma$ and their trace processes on $\partial\Gamma$}\label{reflecting-rw}

In this section, we will give a probabilistic interpretation of the strong Markov
process of jump type associated to the regular Dirichlet form $(\mathcal{E}^{\partial\Gamma,\mu},B_2(\mu))$. 
%We consider  the random walk $(R_n)$ on $\Gamma$ driven by $\mu$ with $\mu\in M_2$. 
%We first construct a {\it reflecting} random walk on $\Gamma\cup\partial\Gamma$ 
%using the notion of {\it reflected Dirichlet spaces} introduced in the previous subsection. Then we prove that the trace process 
%of the reflecting random walk 
% on $\partial\Gamma$ coincides with the jump process associated to $(\mathcal{E}^{\partial\Gamma,\mu},B_2(\mu))$.\par
 In the next subsection we will illustrate what we described above by introducing an example. Specifically, we will generalize the example \eqref{douglas} of the Douglas integral by using notions in Section \ref{time-change}.

 \subsection{Motivating example: Brownian motions and boundary conditions}
 Let $D\subset\mathbb{R}^d$ be a bounded Euclidean domain whose boundary $\partial D$ is sufficiently smooth. There exist at least two natural diffusion processes defined on $D$.\par
 The first one is the absorbing Brownian motion, which we denote by $(\widetilde{{\tt BM}}_t)$. This is a $D$-valued stochastic process which behaves as a $d$-$\dim$ Brownian motion on $\mathbb{R}^d$ until it reaches $\partial D$ and dies out upon hitting the boundary. It is sometimes called the {\it killed} Brownian motion. (Dying, for a Markov process, means being sent to a cemetery point that does not belong to $D$ and staying there forever.) 
 It is known that the corresponding Dirichlet form is given by
 \begin{align}\label{def:df-ab}
 \mathcal{E}^D(f,g)&:=\frac{1}{2}\int_D \left(\sum_{i=1}^d\dfrac{\partial f}{\partial x_i}\cdot\dfrac{\partial g}{\partial x_i}\right)dx_1...dx_d,
 \intertext{with  domain}
 W_0^{1,2}(D)&:=\overline{C_c(D)}^{W^{1,2}(D)}\nonumber,
 \end{align}
 which is regular on $L^2(D)$. Here $C_c(D)$ denotes the set of continuous functions with compact support contained in $D$. 
 See \cite[Example 4.4.1]{FOT}. The domain $ W_0^{1,2}(D)$ consists of functions in the Sobolev space $W^{1,2}(D)$ with zero boundary values, and this absence of nontrivial boundary values of functions in the domain intuitively corresponds to the inactivity of $(\widetilde{\tt BM}_t)$ after its hitting to the boundary.
 \par
 The second one is a reflecting Brownian motion, which we denote by $({\tt BM}_t)$. This is a $(D\cup\partial D)$-valued stochastic process, and it also behaves as a $d$-$\dim$ Brownian motion on $\mathbb{R}^d$ until it reaches $\partial D$, but it reflects off the boundary $\partial D$, then returns to the interior $D$, and keeps moving. The difference between the Dirichlet form corresponding to $({\tt BM}_t)$ and the form $(\mathcal{E}^D, W_0^{1,2}(D))$ is in the domain. Namely, for a reflecting Brownian motion, the domain is replaced by
 \begin{align}\label{def:df-ref}
 W^{1,2}(D):=\left\{f\in L^2(D):\ \dfrac{\partial f}{\partial x_i}\in L^2(D)\ \text{for any}\ 1\leq i\leq d\right\}.
 \end{align}
 It is known that $(\mathcal{E}^D, W^{1,2}(D))$ is also a regular Dirichlet form on $L^2(D\cup\partial D)$. 
 See \cite[Example 1.6.1]{FOT}, \cite[Example 3.5.9]{CF}.\medskip\\
({\bf N.B.} This is a convention used in the theory of Dirichlet forms, which is a bit confusing: since $\partial D$ is a null set, $L^2(D\cup\partial D)$ could be identified with $L^2(D)$. Notice, however, that the definition of the regularity (Definition \ref{def-reg}) involves the set of continuous functions on the state space. Therefore, the claim above should be interpreted as follows: $C(D\cup\partial D)\cap W^{1,2}(D)$ is dense both in $(C(D\cup \partial D),\|\cdot\|_{\infty})$ and $(W^{1,2}(D), (\mathcal{E}^D_1)^{1/2})$).
\par \medskip
Note that $W^{1,2}(D)$ is the maximal domain for the Dirichlet integral \eqref{def:df-ab} to make sense. Furthermore, functions in $W^{1,2}(D)$ display a variety of boundary values and, intuitively speaking, this fact allows the corresponding process to remain active after its hitting to the boundary.\par
Applying Theorem \ref{tc} to the Dirichlet form $(\mathcal{E}^D, W^{1,2}(D))$ corresponding to a reflecting Brownian motion, we can construct its trace form $(\mathcal{E}^{\partial D}, \mathcal{F}^{\partial D})$ on $\partial D$ as follows: let $\kappa$ be the boundary hitting distribution of $({\tt BM}_t)$. Define 
\begin{align}\label{eq:trace-bm}
 \mathcal{F}^{\partial D}&:=\{u\in L^2(\partial D,\kappa);\ u=\tilde{f}\ \kappa\text{-}a.e.\ \text{on}\ \partial D\ \text{for some}\ f\in W^{1,2}(D)\},\\
\mathcal{E}^{\partial D}(u,v)&:=\mathcal{E}^{D}({\bf H}_{\partial D}u,{\bf H}_{\partial D}v)\ \ \text{for}\ u,v\in \mathcal{F}^{\partial D}.
\end{align}
Note that by Definition \ref{def:ext-dirichlet}, it is easy to see that $W_e^{1,2}(D)=W^{1,2}(D)$.
Furthermore, it can be shown that 
$$ \mathcal{F}^{\partial D}=\{u\in L^2(\partial D,\kappa);\ {\bf H}_{\partial D}u\in W^{1,2}(D)\}.$$
See \cite[Section 7.2, Example 7.2.10]{CF} for related discussions.
Therefore, the example \eqref{douglas} of the Douglas integral is exactly what we get when we apply this construction to the two dimensional disc.\par\medskip
In the next subsection, we will address an analogous problem in the context of random walks on hyperbolic groups. Namely, we will first construct a reflecting random walk driven by $\mu\in M_2$, which is a $\Gamma\cup\partial\Gamma$-valued process, by showing that the Dirichlet form 
with maximal domain built from $(\mathcal{E}^\mu,\mathcal{F}^\mu)$ is regular on $L^2(\Gamma\cup\partial\Gamma)$.
We then show that its trace form on $\partial\Gamma$ with respect to the harmonic measure $\nu$ coincides with the Besov space $(\mathcal{E}^{\partial\Gamma,\mu},B_2(\mu))$. 
By Theorem \ref{tc}, this implies the following: Markov processes corresponding to $(\mathcal{E}^{\partial\Gamma,\mu},B_2(\mu))$ is actually the time-change of the reflecting random walk by the right-continuous inverse of the PCAF, which is related to $\nu$ by the Revuz correspondence (Theorem \ref{7.1}).
This gives a concrete probabilistic interpretation of Markov processes corresponding to Besov spaces.
 
 \subsection{Reflected random walks}
 In this subsection, we will finally discuss reflecting extensions in the context of random walks on hyperbolic groups.
 Let us first explain the construction of the reflecting random walk on $\Gamma\cup\partial\Gamma$. 
 Recall the definition of the constant speed random walk from Example \ref{csrw}. 
 Let $\mu\in M_2$, and $(X_t)$ be the constant speed random walk associated to the regular Dirichlet form
 $({\cal E}^{\mu},\ell_2(\Gamma))$ on $\ell_2(\Gamma)$. 
 The process $(X_t)$ does not reach $\partial\Gamma$ in finite time since its holding time at any vertex in $\Gamma$
 is distributed as the exponential distribution with mean 1. \par
 We first take a time change of $(X_t)$
 in such a way that the time-changed process reaches the boundary $\partial\Gamma$ within finite time and dies. This can be viewed as an analogous object to an absorbing Brownian motion on a domain discussed in the last subsection.
 We do it as follows: let $\omega$ be a finite measure on $\Gamma$ with full support.
 Define $(Y_t)$ by taking a time change of $(X_t)$ such that at a vertex $x\in\Gamma$,
 $(Y_t)$ has a holding time distributed as the exponential distribution with mean $\omega(x)$. 
% Note that $(Y_t)$ is only defined for $t<\int_0^\infty \omega(X_s)ds$.
 Then, $(Y_t)$ corresponds to the regular Dirichlet form $({\cal E}^{\mu},\ell_2(\Gamma))$
 on $L^2(\Gamma,\omega)$, and $(Y_t)$ is the time change of $(X_t)$
 by the right-continuous inverse of the PCAF $\int_0^t\omega(X_s)ds$ which is related to $\omega$ by the Revuz correspondence 
 as in Part \ref{ssc:pcaf}. 
(Since $\omega$ has full support, it is obvious that $\omega$ is smooth with respect to
the regular Dirichlet form $({\cal E}^{\mu},\ell_2(\Gamma))$ on $\ell_2(\Gamma)$.) 
Note that we have
\begin{align*}
\mathbb{E}^{\mu}\left[\int_0^{\infty}\omega(X_s)ds\right]&=\int_0^{\infty}\mathbb{E}^{\mu}\left[\omega(X_s)\right]ds\\
&=\int_0^{\infty}\sum_{x\in\Gamma}\mathbb{P}^{\mu}(X_s=x)\omega(x)ds\\
&=\sum_{x\in\Gamma}\omega(x)\int_0^{\infty}\mathbb{P}^{\mu}(X_s=x)ds\\
&\leq \left(\int_0^{\infty}\mathbb{P}^{\mu}(X_s=id)ds\right)\sum_{x\in\Gamma}\omega(x)<\infty,
\end{align*}
where we use the fact that $\omega$ is a finite measure in the last step.
Therefore, it holds that 
\begin{align}\label{eq:pcaf}
\int_0^{\infty}\omega(X_s)ds<\infty\ a.s.
\end{align}
Recall that $(Y_t)$ is the time-change of $(X_t)$ by the right-continuous inverse of $\int_0^t\omega(X_s)ds$.
This together with \eqref{eq:pcaf} implies that almost all trajectories of $(Y_t)$ reach $\partial\Gamma$ at time $\int_0^{\infty}\omega(X_s)ds<\infty$ and die out afterwards. 
 \par \medskip
We next define a new form that gives a reflecting extension of the killed random walk $(Y_t)$. Analogously to the maximal domain \eqref{def:df-ref}, we replace the previous domain $\ell_2(\Gamma)$ with 
$${\cal F}^{\mu}\cap L^2(\Gamma,\omega).$$ 
This domain is indeed the maximal one in the following sense: 
$\mathcal{F}^{\mu}$ is the collection of all functions with finite energy.
On the other hand, a domain of a Dirichlet form, by definition, needs to be a subspace of $L^2$ on which it is defined. Thus we need to take the intersection. Observe however that 
$$({\cal E}^{\mu}, {\cal F}^{\mu}\cap L^2(\Gamma,\omega))\ \text{is }{\it not}\text{ regular on}\  L^2(\Gamma,\omega).$$
 Intuitively, this can be seen as follows:
 since $C_0(\Gamma)$ consists of functions with finite support,
 $C_0(\Gamma)\cap{\cal F}^{\mu}\cap L^2(\Gamma,\omega)$ only contains functions that vanish at infinity. On the other hand, as opposed to the previous domain $\ell_2(\Gamma)$, the new one ${\cal F}^{\mu}\cap L^2(\Gamma,\omega)$ contains functions that have non-zero boundary values since $\omega$ is a finite measure on $\Gamma$. Therefore, $C_0(\Gamma)\cap{\cal F}^{\mu}\cap L^2(\Gamma,\omega)$ cannot be dense in $(({\cal F}^{\mu}\cap L^2(\Gamma,\omega))_1,\mathcal{E}^{\mu})$.

% since $C_0(\Gamma)$ is the set of all functions on $\Gamma$ when $\Gamma$ is equipped with the discrete topology, and ${\cal F}^{\mu}\cap L^2(\Gamma,\omega)$ is not large enough to be dense in $(C_0(\Gamma),\|\cdot\|_{\infty})$.
 In order to show the regularity of the new Dirichlet form with the extended domain, we need to enlarge the space on which it is considered.
 We use the Gromov compactification to obtain the regularity on the larger state space $\Gamma\cup\partial\Gamma$ in the following manner:
  consider $\omega$ as a measure on $\Gamma\cup\partial\Gamma$ that gives zero weight 
 to the boundary. Note that $\omega$ still has full support. Then we look at %all
  functions defined on 
 $\Gamma\cup\partial\Gamma$ whose restrictions to $\Gamma$
 are in $L^2(\Gamma,\omega)$ and in ${\cal F}^{\mu}$.
 Recall that for any function in ${\cal F}^{\mu}$, we can define its boundary value on $\partial\Gamma$
 thanks to the discussion around \eqref{rw-conv}. 
 We now introduce some new notation.
 We denote by $\tilde{\cal F}^{\mu}$ the set of all functions on $\Gamma\cup\partial\Gamma$ that are extensions of functions in ${\cal F}^{\mu}$. We define
 \begin{align*}
 %\tilde{\cal F}^{\mu}&:=\{f:\Gamma\cup\partial\Gamma\to\br\ ;\ f|_{\Gamma}\in {\cal F}^{\mu}\ 
 %{\rm and}\ 
 %\lim_{n\to\infty}f(R_n)=f(R_{\infty})\ {\bp}^{\mu}\mathchar`-a.s.\},\\
 \tilde{\cal F}^{\mu}(\omega)&:=%\tilde{\cal F}^{\mu}\cap L^2(\Gamma\cup\partial\Gamma,\omega)
 \{f\in\tilde{\cal F}^{\mu} ;\ f|_{\Gamma}\in L^2(\Gamma,\omega)\},\\
 \tilde{\cal E}^{\mu}(f,f)&:={\cal E}^{\mu}(f|_{\Gamma},f|_{\Gamma})\ \ {\rm for}\ f\in\tilde{\cal F}^{\mu}.
 \end{align*}
 %{\red I suggest to remove the next sentence whose meaning may be confusing.} Remark that extending functions in ${\cal F}^{\mu}$ to $\Gamma\cup\partial\Gamma$ has a non-trivial content   
 %since the definition of regular Dirichlet forms (See Definition \ref{def-reg}.) takes into account the topological structures of  the state space.\par
 %Precisely speaking, we always assume that the value of $f\in\tilde{\cal F}^{\mu}$ on the boundary is the one given by the extension of its restriction $f|_{\Gamma}$ to $\Gamma$ provided by the discussion above Lemma \ref{0to0}.
  The above extension does not change the $L^2$ norm, nor the Dirichlet norm or the $L^{\infty}$ norm.
  In other words, the two Dirichlet forms $({\cal E}^{\mu}, {\cal F}^{\mu}\cap L^2(\Gamma,\omega))$ and $(\tilde{\cal E}^{\mu}, \tilde{\cal F}^{\mu}(\omega))$ are equivalent in the sense in \cite[Appendix A.4, p 422]{FOT}.
 This kind of maximal extensions of domains of Dirichlet forms can be discussed in the more general context using the notion called {\it reflected Dirichlet spaces} introduced in \cite{Ch1}. We refer interested readers to \cite[Chaper 6]{CF} for detailed discussions. 
 In \cite[Section 6.5]{CF}, the following result is proved:
 consider a regular Dirichlet form $(\ce,L^2(\mathbb{V},m))$, introduced in \eqref{srw}, corresponding to a random walk on a general discrete graph. Then the domain of the reflected Dirichlet space of $(\ce,\cf)$ coincides with the set of all functions with finite energy.
 Therefore, what we study here is a specific example of this concept.
 \par\medskip
 %Moreover, we deduce that $(\tilde{\cal E}^{\mu}, \tilde{\cal F}^{\mu}(\omega))$ is a Dirichlet form on $L^2(\Gamma\cup\partial\Gamma,\omega)$  and the reflected Dirichlet space of $(\tilde{\cal E}^{\mu}, \tilde{\cal F}^{\mu}(\omega))$ is $\tilde{\cal F}^{\mu}$ since similar claims hold for $({\cal E}^{\mu},{\cal F}^{\mu}\cap L^2(\Gamma,\omega))$.
 We now prove that $(\tilde{\cal E}^{\mu}, \tilde{\cal F}^{\mu}(\omega))$
 is a regular Dirichlet form on $L^2(\Gamma\cup\partial\Gamma,\omega)$.

\begin{theo}\label{w}
Assume that the Ahlfors-regular conformal dimension of $\partial\Gamma$ is strictly less than 2.
Then for $\mu\in M_2$, we have the following:
\begin{itemize}
\item $(\tilde{\cal E}^{\mu}, \tilde{\cal F}^{\mu}(\omega))$
 is a regular Dirichlet form on
 $L^2(\Gamma\cup\partial\Gamma,\omega)$.
 \item The extended Dirichlet space of $(\tilde{\cal E}^{\mu}, \tilde{\cal F}^{\mu}(\omega))$ is
 $(\tilde{\mathcal{F}}^{\mu},\tilde{\mathcal{E}}^{\mu})$.
 \end{itemize}
 \end{theo}
 
\begin{proof}
We start with proving the second item.
Choose $f\in\mathcal{F}^{\mu}$ arbitrarily. 
%Let $\Gamma_n$ be a increasing sequence of finite subsets of $\Gamma$ which exhausts $\Gamma_n$, and 
Let $|\cdot|$ be the word metric with respect to a fixed finite symmetric generating set
and define $f_n$ by 
$f_n(x):=f(x){\bf 1}_{\{|x|\leq n\}}$. Obviously we have $f_n\in\tilde{\cal F}^{\mu}(\omega)$ and $\lim_{n\to\infty}f_n=f\ \omega$-$a.e.$
By Proposition \ref{psc}, It suffices to show that
\begin{align}\label{lim0}
\lim_{n,m\to\infty}{\cal E}^{\rm SRW}(f_n-f_m,f_n-f_m)=0,
\end{align}
where ${\cal E}^{\rm SRW}$ is the energy form associated with a simple random walk on the Cayley graph. 
Without loss of generality, we assume $n>m$. It holds that
\begin{align*}
{\cal E}^{\rm SRW}(f_n-f_m,f_n-f_m)\leq \sum_{x,y; |x|,|y|\geq m-1}(f(x)-f(y))^2\to 0
\end{align*}
as $n,m\to\infty$, which yields the claim.
\par
We next prove the first claim.
%In the following proof, we will write $\tilde{\cal F}^{\mu}$ for $\tilde{\cal F}^{\mu}(\omega)$ for simplicity of notation.
In order to prove the theorem, it suffices to show that 
$C(\Gamma\cup\partial\Gamma)\cap\tilde{\cal F}^{\mu}(\omega)$ is dense both in
 $(C(\Gamma\cup\partial\Gamma),\|\cdot\|_{\infty})$ and in
 $(\tilde{\cal F}^{\mu}(\omega), (\tilde{\cal E}^{\mu}_1)^{1/2})$.
 Note that
 \begin{align*}
 C(\Gamma\cup\partial\Gamma)\cap\tilde{\cal F}^{\mu}(\omega)=
 C(\Gamma\cup\partial\Gamma)\cap\tilde{\cal F}^{\mu},
 \end{align*}
 since $\omega$ is a finite measure on $\Gamma\cup\partial\Gamma$.
 \par
We first prove that $C(\Gamma\cup\partial\Gamma)\cap\tilde{\cal F}^{\mu}$
 is dense in $(C(\Gamma\cup\partial\Gamma),\|\cdot\|_{\infty})$.
 By Stone-Weierstrass theorem, we just need to prove that $C(\Gamma\cup\partial\Gamma)\cap\tilde{\cal F}^{\mu}$
 separates points in $\Gamma\cup\partial\Gamma$.
 Obviously, it is enough to prove the claim for points in $\partial\Gamma$. 
 
 Recall the definition of the space ${\bf Lip}_0$ from Part \ref{subsec:regu}. 
 For any $\eta,\xi\in\partial\Gamma$ with $\eta\neq\xi$, there exists a function $u\in{\bf Lip}_0$
 such that $u(\eta)=1$ and $u(\xi)=0$. Note that $u\in{\cal C}$ since ${\bf Lip}_0\subset{\cal C}$. 
 Let $Hu:\Gamma\rightarrow\mathbb{R}$ be the harmonic extension of $u$ as in (\ref{he}). Then
 by Proposition \ref{=} and Theorem \ref{iso}, we get $Hu\in{\cal F}^{\mu}$. 
 
 Take any sequence $(g_n)\subset\Gamma$ converging to $\tau\in\partial\Gamma$.
 Then, by \cite[Lemma 2.2]{Ka}, the sequence of harmonic measures $(\nu_{g_n})$ weakly converges to
 the dirac measure of $\tau$. Since $u$ is bounded and continuous,
 this implies $\lim_{n\to\infty}Hu(g_n)=u(\tau)$. 
Therefore  the function $Hu\cdot{\bf 1}_{\Gamma}+
 u\cdot{\bf 1}_{\partial\Gamma}$ belongs to  $C(\Gamma\cup\partial\Gamma)$.
  This is enough to prove that $C(\Gamma\cup\partial\Gamma)\cap\tilde{\cal F}^{\mu}$
 is dense in $(C(\Gamma\cup\partial\Gamma),\|\cdot\|_{\infty})$. 
 
 We next prove that $C(\Gamma\cup\partial\Gamma)\cap\tilde{\cal F}^{\mu}$ is dense in 
 $( \tilde{\cal F}^{\mu}(\omega), (\tilde{\cal E}^{\mu}_1)^{1/2})$.
 Take $h \in\tilde{\cal F}^{\mu}(\omega)$. 
 Note that by Theorem 1.4.2 (i),(iii) in \cite{FOT}, we can assume that $h$ is a %nonnegative
 bounded function without loss of generality. 
 By applying the Royden decomposition to $h$, we get that
 $$h=h_0+h_1,\ {\rm where}\ h_0\in\ell_2(\Gamma)\ {\rm and}\ h_1\in\mathbb{HD}(\mu).$$
 By the definition of $\ell_2(\Gamma)$, for any $\varepsilon>0$ there exists $k\in\bn$ such that
 $$\sum_{x\in\Gamma:|id-x|_{\Gamma}>k}h_0(x)^2<\varepsilon.$$
 This implies that $h_0$ can be extended to a continuous function on $\Gamma\cup\partial\Gamma$ by setting $h_0|_{\partial\Gamma}=0$.
 Notice that $\ell_2(\Gamma)\subset\tilde{\cal F}^{\mu}(\omega)$ since functions in $\ell_2(\Gamma)$ are bounded
 and belong to ${\cal F}^{\mu}$ by Proposition \ref{RD}.
 Thus we have that $h_0\in\tilde{\cal F}^{\mu}(\omega)$. 
 Therefore, it suffices to show the claim for $h_1$. We know that $h_1$ is  bounded and belongs to $\mathbb{HD}(\mu)\cap\tilde{\cal F}^{\mu}(\omega)$.\par
 %By Theorem 1.4.2 (i),(iii) in \cite{FOT} , we can further assume that $h_1$ is a nonnegative bounded function without loss of generality. 
 %{\s%c This done using truncations and then the truncated $h_1$ may not be harmonic any more. It would be better to truncate the boundary function of $h_1$ and then take the harmonic extension.}
 Define 
 $$C_1:=\sup_{x\in\Gamma}h_1(x).$$
 We let $v:\partial\Gamma\rightarrow\mathbb{R}$ be such that $\lim_{t\to\infty}h_1(X_t)=v(R_\infty)$ as in \eqref{rw-conv}. 
Then $v\in B_2(\mu)$. Moreover, by Corollary \ref{new} we have that $h_1=Hv.$ Since ${\cal C}$ is dense in $(B_2(\mu),(\mathcal{E}^{\partial\Gamma,\mu}_1)^{1/2})$, we can take a sequence $(w_n)\subset{\cal C}$ which converges to $v$ in $(B_2(\mu),(\mathcal{E}^{\partial\Gamma,\mu}_1)^{1/2})$.
 Moreover, by Theorem 1.4.2 (v) in \cite{FOT},  we have the same convergence for the sequence $w_n':= w_n\wedge C_1$, namely
 %we can assume that there exists a constant $C>0$ such that
 \begin{align}\label{e1to0}
 \mathcal{E}^{\partial\Gamma,\mu}(v-w_n',v-w_n')+\|v-w_n'\|^2_{L^2(\nu)}\rightarrow0.
 \end{align}
 Since we have that $\mathcal{E}^{\partial\Gamma,\mu}(u,u)={\cal E}^{\mu}(Hu,Hu)$ for any $u\in B_2(\mu)$ and $h_1=Hv$, this implies that
 \begin{align}\label{eq:conv}
 {\cal E}^{\mu}(h_1-Hw_n',h_1-Hw_n')\rightarrow0.
 \end{align}
 It is easy to see that \eqref{eq:conv} implies that %there exists a sequence $(c_n)\subset\mathbb{R}$ such that
 \begin{align}\label{p-w-con}
 \lim_{n\to\infty}\left(h_1(x)-Hw_n'(x)-c_n\right)=0\ \ {\rm for\ any}\ x\in\Gamma,
 \end{align}
 where
 $$c_n:=h_1(id)-Hw_n'(id).$$ 
 
 First observe that 
$${\cal E}^{\mu}(h_1-Hw_n'-c_n,h_1-Hw_n'-c_n)={\cal E}^{\mu}(h_1-Hw_n',h_1-Hw_n')$$ still converges to $0$. 
 It remains to check that $\|h_1-Hw_n'-c_n\|_{L^2(\Gamma\cup\partial\Gamma,\omega)}\rightarrow0$. 
 
By \eqref{he} and \eqref{e1to0} we have that
 \begin{align*}
 |c_n|&=|h_1(id)-Hw_n'(id)|=|H(v-w_n')(id)|\leq \int_{\partial\Gamma}%K({\it id},\xi)
 |v(\xi)-w_n'(\xi)|d\nu(\xi)\\
 &\leq%\|K({\it id},\cdot)\|_{L^{\infty}(\nu)}\times
 \|v-w_n'\|_{L^1(\nu)}\leq%\|K({\it id},\cdot)\|_{L^{\infty}(\nu)}\times
 \|v-w_n'\|_{L^2(\nu)}.
 \end{align*} Therefore  the sequence $(c_n)$ converges to $0$ as $n\to\infty$ and  
 $\sup_{m\geq1}|c_m|<\infty$.\par
 
 We have that for any $x\in\Gamma$
\begin{align*}
 |h_1(x)-Hw_n'(x)-c_n|\leq 2C_1+\sup_{m\geq0}|c_m|<\infty.
\end{align*}
Since $\omega$ is a finite measure on $\Gamma$,  the dominated convergence theorem implies the conclusion.
 \qed
\end{proof}
\medskip
By Theorem \ref{w}, there exists a $\Gamma\cup\partial\Gamma$-valued, $\omega$-symmetric process
$(W_t)$ associated to the regular Dirichlet form 
$(\tilde{\cal E}^{\mu}, \tilde{\cal F}^{\mu}(\omega))$.
 Note that $1\in\tilde{\cal F}^{\mu}$ and therefore the process $(W_t)$ is recurrent. 
 (See Theorem 1.6.3 in \cite{FOT}.)

 For a process $(S_t)$ on $\Gamma\cup\partial\Gamma$ and a subset $A\subset\Gamma\cup\partial\Gamma$,
 we define 
 \begin{align*}
 \sigma_S(A):=\inf\{t>0\ ;\ S_t\notin A\}.
 \end{align*} 
The next result shows that $(W_t)$ is an extension of $(Y_t)$.

\begin{prop}\label{tau}
 $(Y_t\ ;\ 0\leq t<\sigma_Y(\Gamma))\overset{(d)}{=}(W_t\ ;\ 0\leq t<\sigma_W(\Gamma)).$
\end{prop}
\begin{proof}By Theorem \ref{ex-dr}, it suffices to prove that the extended Dirichlet spaces associated to the above two processes coincide.
 By the inequality (\ref{P}), there exists a constant $C>0$ such that
 \begin{align*}
 \|f\|_{\ell_2(\Gamma)}^2\leq C{\cal E}^{\mu}(f,f)
 \end{align*} 
 for any $f\in\ell_2(\Gamma)$. Hence the extended Dirichlet space associated to  $(Y_t\ ;\ 0\leq t<\sigma_Y(\Gamma))$
 is $\ell_2(\Gamma)$.\par
 On the other hand, by Theorem 3.4.9 in \cite{CF}, the extended Dirichlet space associated to $(W_t\ ;\ 0\leq t<\sigma_W(\Gamma))$
 is given by $\{f\in\tilde{\cal F}^{\mu}\ ;\ \tilde{f}=0\ q.e.\ {\rm on}\ \partial\Gamma\}$, where $\tilde{f}$ is a quasi-continuous
 modification of $f$. 
 Now we have the decomposition ${\cal F}^{\mu}=\ell_2(\Gamma)\bigoplus\mathbb{HD}(\mu)$ and
  any function in $\mathbb{HD}(\mu)$ with zero boundary value should be identically zero. Thus we get the conclusion.  
\qed
\end{proof}
\medskip

Finally, the next theorem gives a probabilistic interpretation of the regular Dirichlet form \\ $(\mathcal{E}^{\partial\Gamma,\mu},B_2(\mu))$
on $L^2(\partial\Gamma,\nu)$.
Recall that by Remark \ref{rem6}, $\nu$ is smooth with respect to $(\mathcal{E}^{\partial\Gamma,\mu},B_2(\mu))$, therefore there exists a PCAF $A^{\nu}$ related to $\nu$ by the Revuz correspondence. 
%{\sc Did we say somewhere that $\nu$ is regular?}\UTF{00A0}%
\begin{theo}\label{trace}
Assume that the Ahlfors-regular conformal dimension of $\partial\Gamma$ is less than 2.
Then for $\mu\in M_2$, the regular Dirichlet form $(\mathcal{E}^{\partial\Gamma,\mu},B_2(\mu))$
on $L^2(\partial\Gamma,\nu)$ coincides with the trace of 
$(\tilde{\cal E}^{\mu}, \tilde{\cal F}^{\mu}(\omega))$
 on $\partial\Gamma$ with respect to $\nu$.
 In other words, the regular Dirichlet form $(\mathcal{E}^{\partial\Gamma,\mu},B_2(\mu))$
on $L^2(\partial\Gamma,\nu)$ corresponds to the $\nu$-symmetric Hunt process $W\circ (A^{\nu})_t^{-1}$ on $\partial\Gamma$.
%which is the  time change of the reflecting random walk $(W_t)$ with respect to the PCAF corresponding to $\nu$ by the Revuz correspondence. 
\end{theo} 
%{\sc In the statement of this Theorem, might be good to explicitly introduce the local time i.e. say that \par
%-$\nu$ is regular and therefore there is a PCAF that we call local time and denote with $L$\par 
%-$(\mathcal{E}^{\partial\Gamma,\mu},B_2(\mu))$ is the trace i.e. the Dirichlet form of the time-changed process $W\circ L^{-1}$.}
\begin{remark}
\begin{itemize}
 \item Notice that the trace of 
$(\tilde{\cal E}^{\mu}, \tilde{\cal F}^{\mu}(\omega))$
 on $\partial\Gamma$ with respect to $\nu$ does not depend on the choice of
 $\omega$.
 \item The formula \eqref{revuz-corr} implies that the PCAF $A^{\nu}_t$ only increases at times when the reflecting random walk is on $\partial\Gamma$.
 This is consistent with what we described in the footnote in page 2 for a Brownian motion on the 2-dim disc. 
 \end{itemize}
\end{remark}
\begin{proof} 
Let $(\tilde{E},\tilde{F})$ be the trace of
 $(\tilde{\cal E}^{\mu}, \tilde{\cal F}^{\mu}(\omega))$ on $\partial\Gamma$ with respect to $\nu$.
Then by Theorem \ref{tc},
$(\tilde{E},\tilde{F})$ is given by
\begin{align*}
\tilde{E}(u,u)&:={\cal E}^{\mu}(\mathbf{H}_{\partial\Gamma}u,\mathbf{H}_{\partial\Gamma}u),\\
\tilde{F}&:=\{u\in L^2(\partial\Gamma,\nu)\ ;\ u=\tilde{g}\ \nu\mathchar`-a.e.\ {\rm on}\ \partial\Gamma
\ {\rm for\ some}\ g\in\tilde{\cal F}^{\mu}\},
\end{align*}
where $\tilde{g}$ is a quasi-continuous modification of $g$ and
$\mathbf{H}_{\partial\Gamma}u(g):=\mathbb{E}_g[u(W_{\sigma_W(\Gamma)})]$.
By Proposition \ref{tau}, we get that
\begin{align*}
\mathbb{E}_g[u(W_{\sigma_W(\Gamma)})]=\int_{\partial\Gamma}u(\eta)d\nu_g(\eta)=Hu(g).
\end{align*}
Thus, it suffices to prove that $\tilde{F}=B_2(\mu)$.
 We first show that $\tilde{F}\supset B_2(\mu)$. Take $v\in B_2(\mu)$ arbitrarily.
 By Proposition \ref{=}, there exists $f\in\tilde{\cal F}^{\mu}$ such that 
 $\lim_{t\to\infty}f(X_t)=v(X_{\infty})$ ${\bp}^{\mu}$-$a.s.$
 This implies that $\tilde{f}=v$ $\nu$-$a.e.$ on $\partial\Gamma$.\par
 We next show that $\tilde{F}\subset B_2(\mu)$. Take $u\in\tilde{F}$ arbitrarily.
 Then, we have that $u=\tilde{g}\ \nu\mathchar`-a.e$.\ on\ $\partial\Gamma$
\ for\ some\ $g\in\tilde{\cal F}^{\mu}$. On the other hand, by Proposition \ref{=} again, 
there exists $w\in B_2(\mu)$ such that $\lim_{t\to\infty}g(X_t)=w(X_{\infty})$ ${\bp}^{\mu}$-$a.s.$
 Hence, we get that $u=w$ $\nu$-$a.e.$, and this implies that $u\in B_2(\mu)$.
\qed
\end{proof}

\section{More on the potential theory of harmonic measures}
\subsection{Measures of finite energy integral and heat kernel estimates}
 In this subsection, we will prove an integral condition 
 for measures to be of finite energy integral, see (\ref{log}).
 We will use the heat kernel estimates for non-local Dirichlet forms
 from \cite{GHH,CKW,CK}, which we state below. 
 Recall that the metric $d_0$ belongs to the Ahlfors-regular conformal gauge $J_{AR}(\partial\Gamma)$,
 and $q_0:=\dim(\partial\Gamma,d_0)<2$. This last condition ensures that results
 obtained in \cite{GHH} can be applied to the regular Dirichlet form $(\mathcal{E}^{\partial\Gamma,d_0},B_2(d_0))$ 
 on $L^2(\partial\Gamma,\mathcal{H}_{d_0})$.  We refer to \cite{GHH}, in particular statement (1.16), for the following

% Since in the paper \cite{GHH,CKW}, the authors studied heat kernels of non-local Dirichlet forms in very general setting,
% and notation required to state their results is very complicated, 
% \begin{remark}
% Precisely speaking, the papers \cite{GHH,CKW} aim to deal with non-local Dirichlet forms whose jump index
% are greater than or equal to 2. The jump index of $(\mathcal{E}^{\partial\Gamma,d_0},B_2(d_0))$ coincides with the Hausdorff dimension
% $q_0$, which is strictly less than $2$. In fact, the paper \cite{CK} studies heat kernel estimates of non-local
% Dirichlet forms whose index are strictly less than 2 in a general setting. But in \cite{CK}, it is assumed that 
% the state space possesses a kind of scaling property, which $\partial\Gamma$ may not satisfy in general.
% On the other hand, the proofs of heat kernel estimates in \cite{GHH,CKW} do not rely on such a scaling property.
% \end{remark}
 
 %The following claim is proved in  , see in particular statement (1.16). 
 %Interested readers may consult \cite{CK1}, where the authors studied heat kernels of non-local Dirichlet forms
 %defined on closed subsets of Euclidean spaces including self-similar sets, 
 %to confirm that $q_0<2$ is a crucial assumption to obtain heat kernel estimates.
 \begin{theo}\label{ck}
  There exists a jointly measurable transition density function
 $p_t^0(\cdot,\cdot)$ on $(0,\infty)\times\partial\Gamma\times\partial\Gamma$ 
 associated to the regular Dirichlet form
 $(\mathcal{E}^{\partial\Gamma,d_0},B_2(d_0))$ on $L^2(\partial\Gamma,\mathcal{H}_{d_0})$,
 and $p_t^0(\cdot,\cdot)$ satisfies the following estimates:
 define $P_t(\cdot,\cdot):(0,\infty)\times\partial\Gamma\times\partial\Gamma\to\mathbb{R}_{\geq0}$ by
 \begin{align*}
 P_t(\xi,\eta)=
 \begin{cases}
 t^{-1},\ \ &{\rm when}\ t\geq d_0(\xi,\eta)^{q_0},\\
 \dfrac{t}{d_0(\xi,\eta)^{2q_0}},\ \ &{\rm when}\ 0<t\leq d_0(\xi,\eta)^{q_0},
 \end{cases}
 \end{align*}
 then there exist constants $C_1,C_2>0$ such that 
 \begin{align*}
 (C_1)^{-1}P_t(\xi,\eta)\leq p_t^0(\xi,\eta)\leq C_1P_t(\xi,\eta)
 \end{align*}
 for any $t\in(0,1)$ and any $\xi,\eta\in\partial\Gamma$, and such that
 \begin{align*}
 (C_2)^{-1}\leq p_t^0(\xi,\eta)\leq C_2
 \end{align*}
 for any $t\geq1$ and any $\xi,\eta\in\partial\Gamma$.
 \end{theo}
 We now introduce the following criterion using
 $p_t^0(\cdot,\cdot)$ for measures of finite energy integral.
 \begin{lem}\label{rho2}
 Let $\kappa$ be a positive Radon measure on $\partial\Gamma$. Then $\kappa\in{\cal S}_0(\partial\Gamma)$
 if and only if
 \begin{align}\label{ene}
 \int_{\partial\Gamma}\int_{\partial\Gamma}\left(\int_0^{\infty}e^{-t}p_t^0(\xi,\eta)dt\right)d\kappa(\xi)
 d\kappa(\eta)<\infty.
 \end{align}
 \end{lem}
 \begin{proof}
 The above equivalence immediately follows from Theorem \ref{s} and Problem 4.2.1 in \cite{FOT}.
 \qed
 \end{proof}
 \medskip
 Using the estimates of $p_t^0(\cdot,\cdot)$ given in Theorem \ref{ck} and
  computing the left hand side
 of (\ref{ene}), we obtain the following criterion for measures of finite energy integral. 
 
 \begin{prop}\label{criterion}
 Let $\kappa$ be a positive Radon measure on $\partial\Gamma$. Then $\kappa\in{\cal S}_0(\partial\Gamma)$
 if and only if 
 \begin{align}\label{log}
 \int_{\partial\Gamma}\int_{\partial\Gamma}|\log d(\xi,\eta)|d\kappa(\xi)d\kappa(\eta)<\infty,
 \end{align}
 for some ($\Leftrightarrow$ {\it any}) metric $d\in J_{AR}(\partial\Gamma)$.
 \end{prop}
 \begin{proof} By Corollary 11.5 in \cite{Hei}, we know that the finiteness of the integral (\ref{log}) does not depend on the choice of
 $d\in J_{AR}(\partial\Gamma)$. Thus, in the light of Lemma \ref{rho2},
  we only need to compare the two integrals (\ref{ene}) and
 (\ref{log}) when we choose $d=d_0$ in (\ref{log}).
 Since ${\rm diam}(\partial\Gamma,d_0)<\infty$ and $\kappa$ is a positive Radon measure, we have that
 \begin{align*}
 \int_{\partial\Gamma}\int_{\partial\Gamma}|\log d_0(\xi,\eta)|
 {\bf 1}_{\{d_0(\xi,\eta)>1\}}d\kappa(\xi)d\kappa(\eta)<\infty.
 \end{align*}
 By using the estimates in Theorem \ref{ck}, it is immediate to check that
 \begin{align*}
 \int_{\partial\Gamma}\int_{\partial\Gamma}\left(\int_0^{\infty}e^{-t}p_t^0(\xi,\eta)dt\right){\bf 1}_{\{d_0(\xi,\eta)>1\}}
 d\kappa(\xi)d\kappa(\eta)<\infty.
 \end{align*}
 Note that by Theorem \ref{ck}, $p_t^0(\cdot,\cdot)$ is of constant order for $t\geq1$.
 Thus in order to show the converse claim,
 it suffices to prove that  
 \begin{align*}
 \int_{\partial\Gamma}\int_{\partial\Gamma}|\log d_0(\xi,\eta)|
 {\bf 1}_{\{d_0(\xi,\eta)\leq1\}}d\kappa(\xi)d\kappa(\eta)<\infty,
 \end{align*}
 if and only if
  \begin{align*}
 \int_{\partial\Gamma}\int_{\partial\Gamma}\left(\int_0^{1}e^{-t}p_t^0(\xi,\eta)dt\right)
 {\bf 1}_{\{d_0(\xi,\eta)\leq1\}}d\kappa(\xi)d\kappa(\eta)<\infty 
 \end{align*} 
 or equivalently if and only if 
   \begin{align*}
 \int_{\partial\Gamma}\int_{\partial\Gamma}\left(\int_0^{1}p_t^0(\xi,\eta)dt\right)
 {\bf 1}_{\{d_0(\xi,\eta)\leq1\}}d\kappa(\xi)d\kappa(\eta)<\infty.
 \end{align*} 

Let $\xi$ and $\eta$ be such that $d_0(\xi,\eta)\leq1$. It immediately follows from the bounds in Theorem \ref{ck} that 
$$(C_1)^{-1}(\frac 12 -q_0\log d_0(\xi,\eta))
\leq \int_0^1  p_t^0(\xi,\eta)dt \leq C_1(\frac 12 -q_0\log d_0(\xi,\eta)).$$

 Thus, we get the conclusion.\ \ \ \qed
 \end{proof}
 \medskip
 We next give an alternative proof of one implication of Proposition \ref{criterion} which does not involve either the heat kernel estimates or the assumption $q_0<2$.
 We will prove that if \eqref{log} is finite then $\kappa\in {\cal S}_0(\partial\Gamma)$. 
 
  \begin{proof}
 %By Corollary 11.5 in \cite{Hei}, we know that the finiteness of the integral (\ref{log}) does not depend on the choice of $d\in J_{AR}(\partial\Gamma)$. Fix a metric $d\in J_{AR}(\partial\Gamma)$.
 Let $v\in{\cal C}$ and define $g:V(\Gamma_{d})\to\br$ as in \eqref{exx}.
 Notice that there exists a surjective continuous map from the set of geodesic rays ${\cal R}$
  in $\Gamma_{d}$ emanating from $O$
 to $\partial\Gamma$, which is defined by $r\mapsto \lim_{t\to\infty}r(t)$ ($r\in{\cal R}$).
 Now we take a measurable section $\xi\in\partial\Gamma\to r_{\xi}\in{\cal R}$.
 See \cite[Chapter I]{Par} for the existence of such a measurable section.  
 %as in \cite[page 97]{BP}
 Remark that by the continuity of $v$,
  for any sequence $(x_n)\subset V(\Gamma_{d})$ and $\xi\in\partial\Gamma$
 with $\lim_{n\to\infty}x_n=\xi$ we have that $\lim_{n\to\infty}g(x_n)=v(\xi)$.
 Thus for any positive Radon measure $\kappa$ on $\partial\Gamma$, we have that
 \begin{align*}
 \int_{\partial\Gamma}vd\kappa-\int_{\partial\Gamma}vd{\cal H}_{d}
 =\int_{\partial\Gamma}d\kappa(\xi)\left(\sum_{e\in\beta(\xi)}dg(e)\right),
 \end{align*}
 where $\beta(\xi)$ is the set of edges in $\Gamma_{d}$ that $r_{\xi}$ passes through. 
 By the Cauchy-Schwartz inequality, we get that
 \begin{align*}
 \int_{\partial\Gamma}d\kappa(\xi)\left(\sum_{e\in\beta(\xi)}dg(e)\right)
 &=\sum_{e\in E(\Gamma_{d})}\left(\int_{\partial\Gamma}d\kappa(\xi){\bf 1}_{\{e\in\beta(\xi)\}}\right)dg(e)\\
 &\leq  \|dg\|_{\ell_2(E(\Gamma_{d}))}\cdot
 \sqrt{
 \sum_{e\in E(\Gamma_{d})}\left(\int_{\partial\Gamma}d\kappa(\xi){\bf 1}_{\{e\in\beta(\xi)\}}\right)^2
 }.
\end{align*}
 For any fixed $e\in E(\Gamma_{d})$, we have that
 \begin{align*}
 \left(\int_{\partial\Gamma}d\kappa(\xi){\bf 1}_{\{e\in\beta(\xi)\}}\right)^2
 =\int_{\partial\Gamma}d\kappa(\xi)\int_{\partial\Gamma}d\kappa(\eta){\bf 1}_{\{e\in\beta(\xi)\cap\beta(\eta)\}}.
 \end{align*}
 Thus,
 \begin{align*}
 \sum_{e\in E(\Gamma_{d})}\left(\int_{\partial\Gamma}d\kappa(\xi){\bf 1}_{\{e\in\beta(\xi)\}}\right)^2
 =\int_{\partial\Gamma}d\kappa(\xi)\int_{\partial\Gamma}d\kappa(\eta)|\beta(\xi)\cap\beta(\eta)|,
 \end{align*}
 where $|A|$ denotes the cardinality of the set $A$.
 By using the tree approximation for $d$-hyperbolic metric spaces (for instance Theorem 1.1 in \cite{Coo}, which refers to \cite{Gro}), we know that there exists a constant $C>0$ such that $|\beta(\xi)\cap\beta(\eta)|$ is bounded above by
 $(\xi|\eta)_O^{d}+C$, where $(\xi|\eta)_O^{d}$ is the Gromov product of $\xi$ and $\eta$ on $\Gamma_{d}$
 with respect to the base point $O$. Noticing that by Proposition 2.1 in \cite{BP}, $\exp(-(\xi|\eta)^{d}_O)$ is comparable to
 $d(\xi,\eta)$, hence $|\beta(\xi)\cap\beta(\eta)|$ is bounded above by $|\log d(\xi,\eta)|+C$. Thus we get the desired result. \qed
 \end{proof}

\subsection{Driving measures with a finite first moment} 

We only proved so far that harmonic measures with driving measures in $M_2$ are of finite energy integral. 
Now we extend Theorem \ref{s} to a harmonic measure of a random walk driven by $\mu\in M_1$. 
 The proof uses Proposition \ref{criterion} and
 the deviation inequality shown in \cite{MS} which controls how a path of a random walk on $\Gamma$ deviates
 from geodesics.

\begin{theo}\label{1st}
Assume the Ahlfors-regular conformal dimension of $\partial\Gamma$ is strictly less than 2.
Then, both of ${\cal S}(\partial \Gamma)$ and ${\cal S}_0(\partial\Gamma)$ contain
 any harmonic measure $\nu$ of a random walk driven by a probability measure $\mu\in M_1$.
\end{theo}
 
\begin{proof}
Denote by $(R_n)$ the random walk driven by $\mu\in M_1$.
 If we choose as $d$ a visual metric $\rho_{\Gamma}$ on $\partial\Gamma$ for (\ref{log}),
 the integral (\ref{log}) is finite if and only if
 \begin{align*}
 \be^{\mu}[(R_{\infty},R'_{\infty})_{id}]<\infty,
 \end{align*}
 where $(\cdot,\cdot)_{id}$ is the Gromov product with respect to the base point $id$
 which is computed in the word metric, and $(R'_n)$ is a random walk
 driven by $\mu$ starting at $id$ which is independent of $(R_n)$.
 By the symmetry of $(R_n)$ and $(R'_n)$, we observe that for any $n,m\in\mathbb{N}$
 \begin{align*}
 \be^{\mu}[(R_{n},R'_{m})_{id}]=\be^{\mu}[(id,R_{n+m})_{R_n}].
 \end{align*}
 By the second statement in \cite[Theorem 11.1]{MS}, we have that
 \begin{align*}
\sup_{n,m\in\mathbb{N}}\be^{\mu}[(id,R_{n+m})_{R_n}]<\infty,
\end{align*}
 which implies the conclusion.
 \qed
 \end{proof}

\medskip
  {\it Acknowledgments.} 
  The authors would like to thank an anonymous referee for numerous helpful comments.
  The first author would like to thank Professor T. Kumagai and the Research Institute for Mathematical Sciences for
   kind hospitality during his stay in Kyoto.
  The second author would like to thank Professor T. Kumagai for his constant encouragement.
  The second author was supported by Grant-in-Aid for JSPS Research Fellow 16J02351.
  This work was supported by University Grants
for student exchange between universities in partnership under Top Global University
Project of Kyoto University.

\end{document}